\documentclass[a4paper,12pt,reqno]{amsart}
\usepackage[left=3.2 cm,top=2.4cm,bottom=3.2cm,right=3.2cm]{geometry}
\usepackage[utf8]{inputenc}
\usepackage{amsmath}
\usepackage{amssymb}
\usepackage{amsthm}
\usepackage{bbm}
\usepackage{xspace}
\usepackage[normalem]{ulem}
\usepackage{footmisc}
\usepackage{mathrsfs}
\usepackage[shortlabels]{enumitem}
\usepackage{comment}
\usepackage{graphicx}
\usepackage{tikz}
\usetikzlibrary{
  math,
  bbox,
  knots,
  hobby,
  decorations.pathreplacing,
  shapes.geometric,
  calc
}

\newtheorem{definition}{Definition}[section]
\newtheorem{proposition}[definition]{Proposition}
\newtheorem{corollary}[definition]{Corollary}
\newtheorem{theorem}[definition]{Theorem}
\newtheorem{lemma}[definition]{Lemma}

\theoremstyle{definition}
\newtheorem{example}[definition]{Example}

\theoremstyle{remark}

\newtheorem{remark}[definition]{Remark}

\newcommand{\ip}[2]{\left\langle#1,#2\right\rangle}

\renewcommand{\P}{\mathcal{P}}
\newcommand{\M}{\mathcal{M}}

\newcommand{\Span}{\operatorname{span}}

\renewcommand{\Im}{\operatorname{Im}}

\newcommand{\col}{\,\colon\,}

\makeatletter

\newcommand{\Dom}{\D}
\newcommand\restr[1]{\raisebox{-0.5ex}{\ensuremath{|_{#1}}}}

\makeatother

\newcommand{\Hil}{\mathcal{H}}
\newcommand{\Rl}{\mathbb{R}}
\newcommand{\Nl}{\mathbb{N}}
\newcommand{\Cl}{\mathbb{C}}

\newcommand{\Om}{\Omega}
\newcommand{\te}{\theta}
\newcommand{\la}{\lambda}
\newcommand{\ot}{\otimes}

\newcommand{\eps}{\varepsilon}
\newcommand{\Ran}{{\rm Ran}}
\newcommand{\Strip}{{\mathbb S}}
\newcommand{\CTSym}{{\mathcal T}_{\rm Sym}}
\newcommand{\Dia}{\mathscr D}
\newcommand{\HB}{{\mathbb H}^\infty_{\rm c.bv.}}
\newcommand{\COs}{{\mathscr O}}
\newcommand{\Man}{{\mathscr M}}
\newcommand{\Ws}{{\mathscr W}}

\def\ker{\operatorname{ker}}

\def\dim{\operatorname{dim}}

\def\id{\operatorname{id}}

\DeclareMathOperator*{\wlim}{w-lim}

\def\Im{\operatorname{Im}}

\def\id{\operatorname{id}}

\newcommand{\D}{\mathcal{D}}

\mathchardef\mhyphen="2D

\newcommand{\CA}[0]{\mathcal{A}} \newcommand{\CB}[0]{\mathcal{B}}
 
 \newcommand{\CF}[0]{\mathcal{F}}
 
\newcommand{\CI}[0]{\mathcal{I}} 
\newcommand{\CK}[0]{\mathcal{K}} \newcommand{\CL}[0]{\mathcal{L}}
 \newcommand{\CN}[0]{\mathcal{N}}
\newcommand{\CO}[0]{\mathcal{O}} \newcommand{\CP}[0]{\mathcal{P}}
 \newcommand{\CR}[0]{\mathcal{R}}
 
\newcommand{\CU}[0]{\mathcal{U}}

\newcommand{\N}{{\mathcal N}}
\newcommand{\B}{{\mathcal B}}
\newcommand{\CC}[0]{\mathcal{C}}

\newcommand{\POm}{P_{\Om}}

\newcommand{\Pst}{{\mathrm{P}}}

\newcommand{\Tw}{{\mathcal T}_{\geq}}
\newcommand{\Tws}{{\mathcal T}_{>}}
\newcommand{\scpr}{\langle\,\cdot\,,\,\cdot\,\rangle}

\makeatletter
\newcommand{\pushright}[1]{\ifmeasuring@#1\else\omit\hfill$\displaystyle#1$\fi\ignorespaces}
\newcommand{\pushleft}[1]{\ifmeasuring@#1\else\omit$\displaystyle#1$\hfill\fi\ignorespaces}
\makeatother

\usepackage{xcolor}

\definecolor{green(munsell)}{rgb}{0.0, 0.66, 0.47}
\definecolor{BlueGreenn}{rgb}{0.3,0.5,0.8}
\definecolor{DB}{rgb}{0.3,0.3,0.3}
\definecolor{DOr}{rgb}{0.7,0.3,0.3}
\definecolor{DGr}{rgb}{0.3,0.7,0.3}
\definecolor{DBl}{rgb}{0.1,0.3,0.5}
\definecolor{arylideyellow}{rgb}{0.91, 0.84, 0.42}
\definecolor{burntorange}{rgb}{0.8, 0.33, 0.0}
\definecolor{chromeyellow}{rgb}{1.0, 0.65, 0.0}
\usepackage[hypertexnames=true, citecolor = burntorange, colorlinks = true, linkcolor = BlueGreenn, pdffitwindow = false, urlbordercolor = white,pdfborder={0 0 0}]{hyperref}

\title[Modular Structure of Twisted Araki-Woods Algebras]{Modular Structure and Inclusions of Twisted Araki-Woods Algebras}

\author{Ricardo Correa da Silva}
\address{Department Mathematik, FAU Erlangen-Nürnberg, ricardo.correa.silva@fau.de}
\author{Gandalf Lechner}
\address{Department Mathematik, FAU Erlangen-Nürnberg, gandalf.lechner@fau.de}

\date{November 30, 2022}

\def\ker{\operatorname{ker}}

\def\dim{\operatorname{dim}}

\def\id{\operatorname{id}}

\def\Im{\operatorname{Im}}

\def\id{\operatorname{id}}

\numberwithin{equation}{section}

\begin{document}

\begin{abstract}
	In the general setting of twisted second quantization (including Bose/Fermi second quantization, $S$-symmetric Fock spaces, and full Fock spaces from free probability as special cases), von Neumann algebras on twisted Fock spaces are analyzed. These twisted Araki-Woods algebras $\CL_{T}(H)$ depend on the twist operator $T$ and a standard subspace $H$ in the one-particle space. Under a compatibility assumption on $T$ and $H$, it is proven that the Fock vacuum is cyclic and separating for $\CL_{T}(H)$ if and only if $T$ satisfies a standard subspace version of crossing symmetry and the Yang-Baxter equation (braid equation). In this case, the Tomita-Takesaki modular data are explicitly determined.

	Inclusions $\CL_{T}(K)\subset\CL_{T}(H)$ of twisted Araki-Woods algebras are analyzed in two cases: If the inclusion is half-sided modular and the twist satisfies a norm bound, it is shown to be singular. If the inclusion of underlying standard subspaces $K\subset H$ satisfies an $L^2$-nuclearity condition, $\CL_{T}(K)\subset\CL_{T}(H)$ has type III relative commutant for suitable twists $T$.

	Applications of these results to localization of observables in algebraic quantum field theory are discussed.
\end{abstract}

\maketitle

\tableofcontents

\section{Introduction}

Second quantization von Neumann algebras play a prominent role in many areas of mathematical physics and mathematics and exist in numerous variations: Weyl algebras and their weak closures describe Bosonic interaction-free systems such as Bose gases or free quantum field theories, and similarly CAR algebras describe the corresponding Fermionic models \cite{BratteliRobinson:1997,Petz:1990}. Whereas these algebras are naturally represented on Bose/Fermi Fock spaces, there exist also more general Fock spaces that are useful for describing generalized statistics, including anyons \cite{LiguoriMintchev:1995-1,DaletskiiKalyuzhnyLytvynovProskurin:2020}. Very similar spaces -- sometimes called $S$-symmetric Fock spaces -- arise from other representations of symmetric groups on tensor powers and form convenient representation spaces for integrable quantum field theories with prescribed two-body scattering operator $S$, and carry families of von Neumann algebras generalizing the CCR/CAR setting \cite{Lechner:2008,AlazzawiLechner:2016}. Related constructions also occur in the representation theory of Wick algebras, where the canonical (anti-)commutation relations are deformed \cite{JorgensenSchmittWerner:1995,JorgensenProskurinSamoilenko:2001}.

Another variation of second quantization algebras arises in free probability: Here one considers representations of free group factors \cite{Voiculescu:1985} and other von Neumann algebras \cite{Shlyakhtenko:1997} (free Araki-Woods factors) on unsymmetrized (``Boltzmann'' or ``full'') Fock spaces, and this generalizes to $q$-deformed Araki-Woods von Neumann algebras, interpolating between the Bose $(q=1)$ and Fermi $(q=-1)$ situations, and containing the free factors at $q=0$ as a special case.

In this article, we adopt a general framework that includes all these examples and goes beyond the setting of Fock spaces symmetrized by symmetric group actions. Our main goal is to study natural families of von Neumann algebras and their modular and inclusion properties in this setting. As our analysis is motivated from quantum field theory, we focus on properties of the Fock vacuum state on these algebras, and consider the structure of inclusions of such von Neumann algebras rather than their internal structure.

The basic setup we use is due to Bożejko and Speicher \cite{BozejkoSpeicher:1994} and J\o rgensen, Schmitt, and Werner \cite{JorgensenSchmittWerner:1995} and can briefly be described as follows: Starting from a one-particle Hilbert space $\Hil$ and a bounded selfadjoint operator $T$ on $\Hil\ot\Hil$ satisfying a certain positivity condition, one can construct a ``$T$-twisted Fock space'' $\CF_T(\Hil)$ which specializes to the Bose, Fermi, Boltzmann, or $S$-symmetric Fock spaces for suitable choices of $T$.

Interesting von Neumann algebras $\CL_{T}(H)\subset\CB(\CF_T(\Hil))$ are generated by ``field operators'' $\phi_{L,T}(h)=a_{L,T}(h)+a_{L,T}^\star(h)$, defined as a sum of the creation and annihilation type operators that exist on $\CF_T(\Hil)$. Here $h$ ranges over a standard subspace\footnote{Recall that a standard subspace $H$ of a complex Hilbert space $\Hil$ is a real linear closed subspace such that $H+iH$ is dense in $\Hil$ and $H\cap iH=\{0\}$, see Section~\ref{section:TwistAlgebras}.} $H\subset\Hil$, and the index $L$ indicates that this construction is based on operators acting on the left. Depending on the twist $T$ and the standard subspace~$H$, the algebra $\CL_{T}(H)$ can take various different forms.

In applications to quantum field theory, the standard subspaces serve as a means to encode localization regions in some spacetime (see \cite{LeylandsRobertsTestard:1978,Foit:1983} for the CCR/CAR case and \cite{BrunettiGuidoLongo:2002} for the concept of modular localization). Hence one is immediately interested in inclusion and intersection properties of, say, two von Neumann algebras $\CL_{T}(H)$ and $\CL_{T}(K)$. In the context of free probability, Voiculescu's original approach corresponds to taking $T=0$ and $H$ as the closed real span of an orthonormal basis, in which case $\CL_0(H)$ is the free group factor of the free group on $\dim\Hil$ generators. Shlyakhtenko has generalized this setting to $T=0$ and general standard subspaces $H$, which corresponds to choosing an orthogonal one-parameter group on a related real Hilbert space. Also in this setting and its $q$-variations, one is interested in certain inclusions of von Neumann algebras, for instance when studying MASAs \cite{BikramMukherjee:2017,BikramMukherjee:2020}.

\bigskip

For general twist $T$, only very little is known about the structure of $H\mapsto\CL_{T}(H)$, which provides some of the motivation for this article. We focus on two interrelated questions: On the one hand, we analyze under which conditions the Fock vacuum $\Om\in\CF_T(\Hil)$ is cyclic and separating for $\CL_{T}(H)$, and what the corresponding modular operators from Tomita-Takesaki theory are in this case. On the other hand, we consider inclusions $\CL_{T}(K)\subset\CL_{T}(H)$ arising from inclusions of standard subspaces $K\subset H$ and investigate their relative commutants.

In Section~\ref{section:TwistAlgebras}, we introduce $T$-twisted Fock spaces and our von Neumann algebras $\CL_{T}(H)$ (Def.~\ref{def:MLTH}), called $T$-twisted Araki-Woods algebras, as well as some instructive examples.

In Section~\ref{sect:TwistedAlgebrasAndStandardness} we address the standardness question of $(\CL_{T}(H),\Om)$. It is easy to show that $\Om$ is always cyclic, but usually not separating. Our analysis is then based on a compatibility assumption between $T$ and $H$ (namely $T$ should commute with the modular unitaries $\Delta_H^{it}\ot\Delta_H^{it}$, see Def.~\ref{def:compatible}) and two properties that $T$ might or might not have: {\em crossing symmetry} and the {\em Yang-Baxter equation}. Crossing symmetry is a concept originating from scattering theory; in our abstract setting it amounts to an analytic continuation property of $T$ w.r.t. the modular group of~$H$ (Def.~\ref{def:crossing}). The Yang-Baxter equation, on the other hand, is an algebraic equation that expresses the essential braiding relation of the braid groups in tensor product form.

It is shown that for compatible twists, $\Om$ separates $\CL_{T}(H)$ if and only if $T$ is crossing symmetric w.r.t. $H$ and satisfies the Yang-Baxter equation (Theorem~\ref{thm:separating-necessary} and Theorem~\ref{thm:YBCrossingLocalSeparating}). Both these properties, the crossing symmetry and the Yang-Baxter equation, have their origins in quantum physics and are often taken as an input in quantum field theoretic constructions (see, for example, \cite{LiguoriMintchev:1995,LechnerSchutzenhofer:2013,BischoffTanimoto:2013,AlazzawiLechner:2016,HollandsLechner:2018,DaletskiiKalyuzhnyLytvynovProskurin:2020}). Here we do not need to assume these structures, but can rather derive them from modular theory and actually show that they are equivalent to the separating property of the Fock vacuum in our setting. On a technical level, this relies on various analytic continuation arguments that are related to the KMS condition. These arguments also involve the combinatorial structure of $T$-twisted $n$-point functions which is best captured in a diagrammatic form. This diagram calculus is presented in the appendix (Section~\ref{sect:diagrams}).

It turns out that $\Om$ is cyclic and separating for $\CL_T(H)$ if and only if corresponding ``right'' field operators $\phi_{R,T}(h)$ \eqref{eq:phiR} and right von Neumann algebras $\CR_T(H)$ exist on $\CF_T(\Hil)$ and suitably commute with the left operators. In Proposition~\ref{prop:modulardata} we then determine the modular data $(\CL_{T}(H),\Om)$ in this case, which are linked to the modular data of $H$ via a $T$-twisted second quantization. In particular, we obtain a left-right duality of the form
\begin{align}
	\CL_T(H)'=\CR_T(H').
\end{align}
This generalizes several results known in special cases \cite{EckmannOsterwalder:1973,LeylandsRobertsTestard:1978,Shlyakhtenko:1997,BaumgartelJurkeLledo:2002,BuchholzLechner:2004,Lechner:2012}.

In Section~\ref{sect:inclusions} we turn our attention to {\em relative} properties of the two families of von Neumann algebras, $\CL_T(H)$ and $\CR_T(K)$. This includes in particular the study of inclusions. Namely, we take an inclusion of standard subspaces $K\subset H$ and consider the corresponding inclusion $\CL_{T}(K)\subset\CL_{T}(H)$ of von Neumann algebras and its relative commutant $\CR_T(K')\cap\CL_T(H)$. We can determine the structure of such an inclusion in two completely different cases, showing that $\CL_{T}(K)\subset\CL_{T}(H)$ crucially depends on~$T$ and $K\subset H$.

The first case is that of a so-called half-sided modular inclusion \cite{Wiesbrock:1993-1,Wiesbrock:1993,ArakiZsido:2005}. As we recall in Section~\ref{section:inclusions-twisted}, the property of being half-sided modular means that the small algebra arises from the large one by application of a translation action with certain properties. It turns out that in this situation, a recently developed criterion for determining whether a half-sided inclusion is singular (trivial relative commutant) \cite{LechnerScotford:2022} is applicable in case our twist is compatible with the inclusion and satisfies $\|T\|<1$. The latter condition rules out familiar cases such as Bose/Fermi symmetry, for which the inclusion is non-singular. In Theorem~\ref{thm:twisted-hsm} we can therefore establish many new examples of singular half-sided modular inclusions. After recent more complicated constructions of such inclusions, using free products \cite{LongoTanimotoUeda:2019} and deformation procedures \cite{LechnerScotford:2022}, respectively, these appear to be a much more transparent. This shows that the relative commutant is unstable and very sensitive to perturbations in~$T$.

The results on small (trivial) relative commutants mentioned so far can be interpreted as a consequence of the very non-commutative structure arising from free probability (at $T=0$) respectively its ``neighbourhood'' (at $\|T\|<1$). Nonetheless, such twists can also lead to very different inclusions which might appear counter-intuitive at first. In Section~\ref{section:nuclearity}, we consider standard subspaces $K\subset H$ satisfying the $L^2$-nuclearity condition of Buchholz, D'Antoni, Longo \cite{BuchholzDAntoniLongo:2007} and use this tool and the theory of (quasi-)split inclusions \cite{DoplicherLongo:1984,Fidaleo:2001} to give examples for the relative commutant $\CL_{T}(K)'\cap\CL_{T}(H)$ being type III (Proposition~\ref{prop:L2}). This section is related to work of D'Antoni, Longo, and Radulescu \cite{DAntoniLongoRadulescu:2001} who considered the case $T=0$, and is generalized to other twists here. See also the paper \cite{BikramMukherjee:2020} by Bikram and Mukherjee for related arguments in the case of $q$-deformed Araki-Woods von Neumann algebras.

In Section~\ref{section:QFT}, we explain how our results apply to quantum field theoretic models, in particular regarding the absence or presence of local observables in quantum field theories on Minkowski space or the lightray, using an abstract notion of spacetime. We discuss our results from this point of view, review the resulting perspectives on constructive algebraic quantum field theory, and give an outlook to ongoing research.

\section{Twisted Fock spaces and Araki-Woods algebras}\label{section:TwistAlgebras}

\subsection{Twists and twisted Fock spaces} The von Neumann algebras we are interested in are defined on twisted Fock spaces. These Fock spaces arise as natural representation spaces of algebras defined by a quadratic exchange relation (Wick algebras) and exist in various versions. The version that we use here is the most general one and due to Bożejko and Speicher \cite{BozejkoSpeicher:1994} and J\o rgensen, Schmitt, and Werner \cite{JorgensenSchmittWerner:1995}.

\medskip 

Let $T\in\CB(\Hil\ot\Hil)$ be an operator with $\|T\|\leq1$. We iteratively define operators $R_{T,n},P_{T,n}\in\CB(\Hil^{\ot n})$, $n\in\Nl$, by\footnote{Here and hereafter, we use the notation $T_k:=1^{\ot (k-1)}\ot T\ot1^{\ot(n-k-1)}\in\CB(\Hil^{\ot n})$.\label{fn:tensors}}
\begin{align}
R_{T,n}    &:= 1+T_1+T_1T_2+\ldots+T_1\cdots T_{n-1},
\\
P_{T,1}    &:= 1,\qquad
P_{T,n+1}  := (1\ot P_{T,n})R_{T,n+1}.
\label{def:Pn}
\end{align}

If $T$ is selfadjoint, one can show by induction in $n$ that also $P_{T,n}$ is selfadjoint for any $n\in\Nl$. Those $T$ that lead to positive $P_{T,n}$ will be called ``twists'':

\begin{definition}\label{def:twist}
	A {\em twist} is an operator in
	\begin{align}
		\Tw
		&:=
		\{T\in\CB(\Hil\ot\Hil)\col \|T\|\leq1,\quad P_{T,n}\geq0\;\text{ for all } n\in\Nl\}.
	\end{align}
	A {\em strict twist} is an operator in
	\begin{align}
		\Tws
		&:=
		\{T\in\Tw\col 0\not\in\sigma(P_{T,n})\;\text{ for all } n\in\Nl\}.
	\end{align}
\end{definition}

We note that any twist $T$ is selfadjoint because $P_{T,2}=1+T$ is required to be positive. To avoid misunderstanding, we emphasize that the spectral assumption on strict twists $T\in\Tws$ means that for each $n\in\Nl$, there exists an $\eps_n>0$ such that $P_{T,n}\geq\eps_n1>0$.

The following theorem summarizes some known sufficient conditions for $T\in\CB(\Hil\ot\Hil)$ to lie in $\Tw$ or $\Tws$. Items a) and b) are due to J\o rgensen, Schmitt, and Werner ~\cite[Thm.~2.3.2 and Thm.~2.5.1]{JorgensenSchmittWerner:1995}, and item c) is due to Bożejko and Speicher \cite[Thm.~ 2.2]{BozejkoSpeicher:1994}.

\begin{theorem}\label{theorem:T}
	Let $T=T^*\in\CB(\Hil\ot\Hil)$.
	\begin{enumerate}
		\item\label{item:Tsmallnorm} If $\|T\|\leq\frac{1}{2}$, then\footnote{In \cite[Thm.~2.3.2]{JorgensenSchmittWerner:1995} this conclusion is shown under the slightly stronger assumption $\|T\|<\frac{1}{2}$, but the arguments given there carry over to $\|T\|\leq\frac{1}{2}$.} $T\in\Tws$.
		\item\label{item:Tpositive} If $T\geq0$, then $T\in\Tws$.
		\item\label{item:Tbraided} If $\|T\|\leq1$ and $T$ satisfies the Yang-Baxter equation in its braid form, i.e.
		\begin{align}
			T_1T_2T_1=T_2T_1T_2,
		\end{align}
		then $T\in\Tw$. We will refer to such twists as {\em braided twists}. If a braided twist satisfies $\|T\|<1$, then $T\in\Tws$.
	\end{enumerate}
\end{theorem}

To construct twisted Fock spaces from a twist $T\in\Tw$, we consider the quotient
\begin{align}\label{quot}
	\Hil_{T,n}^0
	:=
	\Hil^{\ot n}/\ker P_{T,n}
	=
	\overline{\Ran P_{T,n}}
	,\qquad n\in\Nl,
\end{align}
where the bar denotes the closure in the canonical Hilbert norm of $\Hil^{\ot n}$. On this space we introduce a new (positive definite) scalar product
\begin{align}
	\langle\,\cdot\,,\,\cdot\,\rangle_T := \langle\,\cdot\,,P_{T,n}\,\cdot\,\rangle,
\end{align}
and define $\Hil_{T,n}$ as the completion of $\Hil_{T,n}^0$ w.r.t. the corresponding norm $\|\cdot\|_T$, written as $\Hil_{T,n} = \overline{\Hil_{T,n}^0}^T=\overline{\Ran P_{T,n}}^T$. In general, $\Hil_{T,n}$ is not a subspace of~$\Hil^{\ot n}$, but, in case $P_{T,n}|_{(\ker P_{T,n})^\perp}$ does not contain $0$ in its spectrum, $P_{T,n}|_{(\ker P_{T,n})^\perp}$ has a bounded inverse and the norms $\|\cdot\|$ and $\|\cdot\|_T$ are equivalent on $(\ker P_{T,n})^\perp$, i.e. $\Hil_{T,n}=\overline{\Ran P_{T,n}}\subset\Hil^{\ot n}$. If $T\in\Tws$ is a strict twist, $P_{T,n}$ is even invertible (with bounded inverse) in $\CB(\Hil^{\ot n})$, i.e. $\Hil_{T,n}=\Hil^{\ot n}$ as vector spaces in this case.

For any twist $T\in\Tw$, the {\em twisted Fock space} is defined as the direct sum
\begin{align*}
	\CF_T(\Hil) := \bigoplus_{n=0}^\infty \Hil_{T,n},\qquad \Hil_{T,0}=\Cl,
\end{align*}
equipped with the scalar product $\langle\Psi,\Phi\rangle_T=\sum_{n=0}^\infty\langle\Psi_n,P_{T,n}\Phi_n\rangle_{\Hil^{\ot n}}$, $\Psi_n,\Phi_n\in\Hil_{T,n}$. We will use the notation $\Psi=\bigoplus_n\Psi_n$ to denote the ``$n$-particle components'' $\Psi_n\in\Hil_{T,n}$ of a vector $\Psi\in\CF_T(\Hil)$, and write $\Om=1\oplus0\oplus0\ldots$ for the Fock vacuum.

Although the inner product $\langle\,\cdot\,,\,\cdot\,\rangle_T$ is the natural inner product in $\CF_T(\Hil)$, we will occasionally also need to argue in $\Hil^{\ot n}$ (or $\Hil^{\ot n}/\ker P_{T,n}$) using untwisted inner products $\langle\,\cdot\,,\,\cdot\,\rangle$. Adjoints w.r.t. $\langle\,\cdot\,,\,\cdot\,\rangle_T$ are denoted $A\mapsto A^\star$ and adjoints w.r.t. $\langle\,\cdot\,,\,\cdot\,\rangle$ are denoted $A\mapsto A^*$. The operator norm of $\CB(\CF_T(\Hil))$ is written $\|\cdot\|_T$, and that of $\CB(\Hil^{\ot n})$ as $\|\cdot\|$, as usual. Note that the scalar products $\scpr$ and $\scpr_T$ coincide on the zero and one-particle vectors, i.e. on $\Cl\Om\oplus\Hil$.

\bigskip

The twisted Fock spaces $\CF_T(\Hil)$ can look quite different depending on the twist~$T$, and we now recall some important examples. For further discussions of these and other examples, see \cite{BozejkoSpeicher:1991,BozejkoSpeicher:1994,JorgensenSchmittWerner:1995,LiguoriMintchev:1995-1,JorgensenProskurinSamoilenko:2001,LechnerSchutzenhofer:2013,DaletskiiKalyuzhnyLytvynovProskurin:2020}.

\begin{example}
	The simplest example of a strict twist is $T=0$. In this case, $P_{T,n}=1$ for all $n\in\Nl$, and hence the $0$-twisted Fock space $\CF_0(\Hil)$ over $\Hil$ coincides with the full (Boltzmann) Fock space 
	\begin{align*}
		\CF_0(\Hil) = \bigoplus_{n=0}^\infty\Hil^{\ot n}
	\end{align*}
	with its canonical inner product.
\end{example}

\begin{example}\label{example:TSym}{\bf (Involutive Yang-Baxter solutions, symmetric twists)}\\
	An important class of twists $T$ are the {\em symmetric twists}, namely the family
	\begin{align}
		\CTSym := \{T\in\B(\Hil\ot\Hil)\col T=T^*=T^{-1},\quad T_1T_2T_1=T_2T_1T_2\},
	\end{align}
	of unitary (and selfadjoint) involutive solutions of the Yang-Baxter equation\footnote{For $\dim\Hil<\infty$, these solutions have been classified in \cite[Thm.~4.8]{LechnerPennigWood:2019}.}. These are in particular braided twists, hence contained in $\Tw$. Any $T\in\CTSym$ defines a unitary representation $\rho_{T,n}$ of the symmetric group $S_n$ on $\Hil^{\ot n}$ by $\rho_{T,n}(\sigma_k):=T_k$, where the $\sigma_k$ are the Coxeter generators of $S_n$.

	Looking at the definition of $R_{T,n}$ and $P_{T,n}$, one realizes that in this case $n!^{-1}P_{T,n}$ is the projection onto the subspace $\Ran P_{T,n}$ of $\Hil^{\ot n}$ on which $\rho_{T,n}$ acts trivially. Thus $T\not\in\Tws$ unless $T=1$.

	Since the projection $n!^{-1}P_{T,n}$ acts trivially on its (closed) range, the maps
	\begin{align}\label{eq:In}
		I_n:(\Ran P_{T,n},\langle\,\cdot\,,\,\cdot\,\rangle_T)\to(\Ran P_{T,n},\langle\,\cdot\,,\,\cdot\,\rangle), \qquad\Psi_n\mapsto\sqrt{n!}\Psi_n,
	\end{align}
	define a unitary $I=\bigoplus_n I_n$ between the $T$-twisted Fock space $\CF_T(\Hil)$ and another Fock space naturally associated with a symmetric twist $T$, the {\em $T$-symmetric Fock space}. The $T$-symmetric Fock space over $\Hil$ is defined as $\bigoplus_{n=0}^\infty \Ran P_{T,n}$ with the $T$-independent scalar products inherited from $\Hil^{\ot n}$ by restriction. For more on $T$-symmetric Fock spaces, see \cite{LiguoriMintchevRossi:1995,Lechner:2003} and the references cited there.

	The symmetric twists include in particular the {\em flip} 
	\begin{align}
		T=F:v\ot w\mapsto w\ot v,
	\end{align}
	giving rise (via the unitary $I$ \eqref{eq:In}) to the Bose Fock space $I\CF_F(\Hil)$, and the negative flip $-F$, giving rise to the Fermi Fock space $I\CF_{-F}(\Hil)$.

	Two other special examples of symmetric twists are $T=\pm\id_{\Hil\ot\Hil}$. In the case of the positive sign, we have $P_{1,n}=n!$ and hence $I$ maps the twisted Fock space onto the Boltzmann Fock space, $I\CF_1(\Hil)=\CF_0(\Hil)$. In the case of the negative sign, we have $P_{-1,n}=0$ for all $n>1$, i.e. $\CF_{-1}(\Hil)=\Cl\oplus\Hil$.

	\medskip

	A subset of symmetric twists are related to solutions of the Yang-Baxter equation with (additive) spectral parameter. Let $\CK$ be a separable Hilbert space and $S:\Rl\to\CU(\CK\ot\CK)$ a measurable function taking values in the unitaries on the tensor square $\CK\ot\CK$, and satisfying the Yang-Baxter equation with  spectral parameter\footnote{That is, the equation $S(\te)_1S(\te+\te')_2S(\te')_1=S(\te')_2S(\te+\te')_1S(\te)_2$ holds in $\CB(\CK^{\ot 3})$ for almost all $\te,\te'\in\Rl$. This condition is automatically satisfied if $\CK=\Cl$.} as well as the symmetry condition $S(\te)^*=S(-\te)$ for almost all $\te\in\Rl$.

	Then (the tensor square of) the Hilbert space $\Hil=L^2(\Rl\to\CK,d\te)$ carries the symmetric twist
	\begin{align}\label{eq:TS}
		(T_S\psi)(\te_1,\te_2)=S(\te_2-\te_1)\psi(\te_2,\te_1).
	\end{align}
	These twists are often considered in the context of generalized statistics \cite{LiguoriMintchev:1995-1} and integrable quantum field theories, where $S$ has the meaning of an elastic two-body scattering matrix -- see, for example, \cite{Smirnov:1992,LechnerSchutzenhofer:2013}.
\end{example}

\begin{example}\label{example:TSym-scaled}{\bf (Scaled Yang-Baxter solutions)}\\
	Given a symmetric twist $T\in\CTSym$ and $-1<q<1$, the scaled twist $qT$ lies in~$\Tws$ but not in $\CTSym$ because $\|qT\|=|q|<1$. An example that has been studied a lot in the literature is the ``$q$-Fock space'' $\CF_{qF}(\Hil)$ defined by the scaled flip $qF$, which interpolates between Bose and Fermi statistics \cite{FrischBourret:1970,BozejkoSpeicher:1991,BikramMukherjee:2017}.
\end{example}

Later on we will see that in our context, braided twists are the most interesting ones. The family of braided twists includes all finite-dimensional contractive selfadjoint solutions $T$ to the Yang-Baxter equation. Contrary to the situation considered in Example~\ref{example:TSym-scaled}, these do not necessarily have spectrum contained in a circle. Such braided twists can therefore not be rescaled to give representations of the symmetric groups, but only define representations of the braid groups $B_n$.
\\
Moreover, there exist also genuinely infinite-dimensional braided twists such as SO$(d,1)$-symmetric twists considered in the context of quantum field theories on de Sitter space \cite{HollandsLechner:2018}.
\\
Because of the lack of general classification results on solutions of the Yang-Baxter equation, the structure of $\Tw$ is very complex and largely unknown in general.

Following \cite{BozejkoSpeicher:1994,JorgensenSchmittWerner:1995}, we now fix an arbitrary twist $T\in\Tw$ and describe creation and annihilation type operators on $\CF_T(\Hil)$ which will be used to generate our von Neumann algebras. Selfadjointness of $P_{T,n}$ implies
\begin{align}
	P_{T,n+1}=R_{T,n+1}^*(1\ot P_{T,n}),
\end{align}
from which we read off that the left creation operators $\Psi_n\mapsto\xi\ot\Psi_n$, $\xi\in\Hil$, $\Psi_n\in\Hil^{\ot n}$, descend to maps
\begin{align}
	a_{L,T}^\star(\xi):\Hil^{\ot n}/\ker P_{T,n} &\to \Hil^{\ot (n+1)}/\ker P_{T,n+1},
	\\
	a_{L,T}^\star(\xi)[\Psi_n] &:= [\xi\ot\Psi_n],
	\label{def:aLTstar}
\end{align}
where $[\,\cdot\,]$ denotes the quotient map, i.e. the $\scpr$-orthogonal projection onto $\overline{\Ran P_{T,n}}\subset\Hil^{\ot n}$. Thus $a_{L,T}^\star(\xi)$ extend to densely defined operators on $\CF_T(\Hil)$ which contain the quotients $\Hil^{\ot n}/\ker P_{T,n}$, $n\in\Nl$, in their domains. As shown by Bożejko and Speicher \cite[Thm.~3.1]{BozejkoSpeicher:1994}, these operators are bounded on each $\Hil_{T,n}$, $n\in\Nl$. In fact, positivity and the defining formula \eqref{def:Pn} imply that
\begin{align*}
	P_{T,n+1}^2&=(1\ot P_{T,n})R_{T,n+1}R_{T,n+1}^*(1\ot P_{T,n})\\
	&\leq(1\ot P_{T,n}) \|R_{T,n+1}^*\|^2(1\ot P_{T,n})\\
	&\leq c_{T,n}^2(1\ot P_{T,n})^2,
\end{align*}
where $c_{T,n}=\sum_{k=0}^n\|T\|^k$.

Hence $P_{T,n+1}\leq c_{T,n}(1\ot P_{T,n})$
and for $[\Psi_n]\in\Hil^{\ot n}/\ker P_{T,n}$, we have
\begin{align*}
	\|a_{L,T}^\star(\xi)[\Psi_n]\|^2_T
	&=
	\langle\xi\ot\Psi_n,P_{T,n+1}\xi\ot\Psi_n\rangle
	\nonumber
	\\
	&\leq
	c_{T,n}\langle\xi\ot\Psi_n,(1\ot P_{T,n})\xi\ot\Psi_n\rangle
	\nonumber
	\\
	&=
	c_{T,n}\|\xi\|^2\|[\Psi_n]\|_T^2
	,
\end{align*}
which implies the operator norm bound
\begin{align}\label{eq:abound1}
	\|a_{L,T}^\star(\xi)|_{\Hil_{T,n}}\|_T
	&\leq
	\sqrt{c_{T,n}}\|\xi\|,\qquad c_{T,n}
	=
	\begin{cases}
		\frac{1-\|T\|^{n+1}}{1-\|T\|} & \|T\|<1,\\
		n+1 & \|T\|=1
	\end{cases}
	.
\end{align}

One may calculate the adjoint $a_{L,T}(\xi)$ of $a_{L,T}^\star(\xi)$ (w.r.t. $\scpr_T$) as
\begin{align}\label{eq:aL}
a_{L,T}(\xi)[\Psi_n] = [a_L(\xi)R_{T,n}\Psi_n],
\end{align}
where $a_L(\xi)$ is the untwisted left annihilation operator, namely the Boltzmann Fock space operator fixed by
\begin{align}\label{eq:aL-untwisted}
	a_L(\xi)(\psi_1\ot\ldots\psi_n) = \langle\xi,\psi_1\rangle\cdot\psi_2\ot\ldots\ot\psi_n.
\end{align}
Note that \eqref{eq:aL} is well-defined on $\Hil^{\ot n}/\ker P_{T,n}$ and extends to a bounded operator $\Hil_{T,n}\to\Hil_{T,n-1}$. We will use the same symbols $a_{L,T}(\xi)$, $a_{L,T}^\star(\xi)$ to denote the operators that are defined on the dense finite particle subspace of $\CF_T(\Hil)=\bigoplus_n\Hil_{T,n}$ consisting of terminating direct sums.

In view of \eqref{eq:abound1}, these operators are bounded in case $\|T\|<1$, namely
\begin{align}\label{eq:altnormbound}
	\|a_{L,T}^\star(\xi)\|_T \leq \frac{\|\xi\|}{\sqrt{1-\|T\|}} < \infty .
\end{align}

\begin{remark}
	Let us comment on a simplified description of the operators $a_{L,T}^\star(\xi)$ that is available for symmetric twists $T\in\CTSym$. In this case, $\Hil_{T,n}=\Ran P_{T,n}\subset\Hil^{\ot n}$ is the closed subspace of $T$-symmetric vectors, i.e. the vectors $\Psi_n\in\Hil^{\ot n}$ satisfying $T_k\Psi_n=\Psi_n$, $k=1,\ldots,n-1$. The $\scpr$-orthogonal projection $Q_{T,n}$ onto it (the ``$T$-symmetrization'') is given by $Q_{T,n}=n!^{-1}P_{T,n}$. From this one infers that the creation operators satisfy
	\begin{align}
		Ia_{L,T}^\star(\xi)I^{-1}Q_{T,n}\Psi_n = \sqrt{n+1} Q_{T,n+1}(\xi\ot\Psi_n),\qquad \Psi_n\in\Hil^{\ot n},
	\end{align}
	where $I$ is the unitary \eqref{eq:In}. This shows that when transported to the $T$-symmetric Fock space, we recover the familiar Zamolodchikov creation and annihilation operators (see \cite{ZamolodchikovZamolodchikov:1992} for their origin in integrable models and \cite{LechnerSchutzenhofer:2013} for a mathematical formulation).

	In particular, these operators form representations of the CCR and CAR algebras \cite[Sect.~5.2.2]{BratteliRobinson:1997} on the Bose and Fermi Fock space for $T=F$ and $T=-F$, respectively.
\end{remark}

We conclude this section with a comment on second quantization of operators on $T$-twisted Fock spaces. Here and in the following, we write $\Dom(X)$ for the domain of an operator $X$ and we denote by $X\ot Y: \Dom(X)\odot \Dom(Y)\to \Hil \ot \Hil$ the linear operator defined by $(X \ot Y)(x \ot y)= X(x) \ot Y(y)$. Furthermore, we write $X^\#$ to denote either $X$ or $X^\star$, i.e. an equation of the form $f(X^\#)=Y^\#$ means $f(X)=Y$ and $f(X^\star)=Y^\star$.

\begin{lemma}\label{lemma:T-second-quantization}
	Let $T\in\Tw$ be a twist and $V:\Dom(V)\to\Hil$ a closable linear or antilinear operator such that 
	\begin{align}\label{eq:VcommuteswithT}
		[V\ot V,T]=0.
	\end{align}
	\begin{enumerate}
		\item\label{item:T-second-quantization} Then $\hat{V}:\bigoplus_{n=0}^\infty\Dom(V)^{\odot n}/\ker(P_{T,n}) \to \Hil_{T,n} $, given by $\hat{V}([\Psi_n])=[V^{\ot n}\Psi_n]$, for all $\Psi_n \in \Dom(V)^{\odot n}/\ker(P_{T,n})$ and for all $n \in \Nl$, where $V^{\ot 0}:=1$ , is a well-defined closable operator on $\CF_T(\Hil)$ whose closure will be denoted by $\Gamma_T(V)$. In case $V$ is (anti)unitary, so is $\Gamma_T(V)$.
		\item\label{item:T-second-quantization-covariance} If $V$ is invertible, we have for any $\xi\in\Dom(V)$
		\begin{align}
			\overline{\Gamma_T(V)a_{L,T}^\#(\xi)\Gamma_T(V^{-1})} &= a_{L,T}^\#(V\xi).
		\end{align}
	\end{enumerate}
\end{lemma}
\begin{proof}
	\ref{item:T-second-quantization} Thanks to the commutation condition \eqref{eq:VcommuteswithT}, one checks $T_k V^{\ot n}=V^{\ot n}T_{k}$, $k=1,\ldots,n-1$. This implies $ R_{T,n} V^{\ot n}=V^{\ot n} R_{T,n}$ and inductively $  P_{T,n} V^{\ot n}=V^{\ot n}P_{T,n}$, i.e. $[V^{\ot n},P_{T,n}]=0$. In particular, $\Dom(V)^{\odot n}/\ker(P_{T,n})\subset \Dom(V^{\ot n})/\ker(P_{T,n})$. It follows that $\hat{V}$ is a $\|\cdot\|_T$-densely well-defined (anti)linear operator in~$\CF_T(\Hil)$. The exact same argument for $V^\ast \ot V^\ast$ proves the $\widehat{V^\ast}\subset \hat{V}^\star$ is also a $\|\cdot\|_T$-densely and well-defined operator in~$\CF_T(\Hil)$. Hence $\hat{V}$ is a closable operator in~$\CF_T(\Hil)$. The fact that $\Gamma_T(V)$ is (anti)unitary for (anti)unitary $V$ is easy to check from this construction.

	\ref{item:T-second-quantization-covariance} can be easily checked on a core for $a_{R,T}^\#(V\xi)$.
\end{proof}

\subsection{Standard subspaces and twisted Araki-Woods algebras}
We now use the twisted left field operators\footnote{These operators generalize free quantum fields, hence the name.}
\begin{align}\label{eq:phiLT}
	\phi_{L,T}(\xi)=a_{L,T}^\star(\xi)+a_{L,T}(\xi),\qquad \xi\in\Hil,
\end{align}
to generate von Neumann algebras. For $\|T\|<1$, these are bounded selfadjoint operators on $\CF_T(\Hil)$ \eqref{eq:abound1}. If $\|T\|=1$, the bound $\|a_{L,T}^\#(\xi)|_{\Hil_{T,n}}\|_T\leq \sqrt{n+1}\cdot\|\xi\|$  \eqref{eq:abound1} implies that any vector of finite particle number is an (entire) analytic vector for $\phi_{L,T}(\xi)$, hence $\phi_{L,T}(\xi)$ is essentially selfadjoint on the domain of finite particle number by Nelson's Theorem and we use the same symbol to denote its selfadjoint closure.

Before we define the von Neumann algebras we are interested in, we recall a few facts about {\em standard subspaces} \cite{RieffelvanDaele:1977,Longo:2008_2}: By definition, a standard subspace of a complex Hilbert space $\Hil$ is a closed real subspace $H\subset\Hil$ that is cyclic in the sense that $H+iH\subset\Hil$ is dense, and separating in the sense $H\cap iH=\{0\}$. 

A standard subspace $H$ defines modular data through the polar decomposition $S_H=J_H\Delta_H^{1/2}$ of the closed operator with domain $\Dom_H:=H+iH$
\begin{align}
    S_H:H+iH\to H+iH,\qquad h_1+ih_2\mapsto h_1-ih_2.
\end{align}
The operator $S_H$ is called the Tomita operator of $H$. It satisfies an analogue of Tomita's theorem for von Neumann algebras, namely
\begin{align}\label{tomita-H}
	J_HH=H',\qquad
	\Delta_H^{it}H=H,\quad t\in\Rl,
\end{align}
where 
\begin{align}
    H':=\{h'\in\Hil\colon \Im\langle h,h'\rangle=0\quad\forall h\in H\}
\end{align}
is the symplectic complement of $H$, which is a standard subspace as well (with $J_{H'}=J_H$, $\Delta_{H'}=\Delta_H^{-1}$).

\begin{definition}\label{def:MLTH}
	Given a closed real subspace $H\subset\Hil$ and a twist $T\in\Tw$, we define the (left) $T$-twisted Araki-Woods von Neumann algebra
	\begin{align}
		\CL_{T}(H) := \{\exp(i\phi_{L,T}(h))\,:h\in H\}''
		\subset\CB(\CF_T(\Hil)).
	\end{align}
\end{definition}
In case the field operators are bounded, this is simply the weak closure of the ${}^\star$-algebra $\CP_{L,T}(H)$ of all polynomials in $\phi_{L,T}(h)$, $h\in H$.

\medskip 

If $T=F$ is the flip, $\CL_{T}(H)$ is a second quantization von Neumann algebra which is well understood, and closely related to the original work of Araki and Woods \cite{ArakiWoods:1968}. In this case, $\CL_{F}(H)$ is a factor if and only if $H\cap H'=\{0\}$, and $\Om$ is cyclic and separating if and only if $H$ is a standard subspace \cite[Thm. I.3.2]{LeylandsRobertsTestard:1978}. For more details on modular structure and a summary of results on the type of $\CL_{F}(H)$, see \cite[Thm.~1.3]{FiglioliniGuido:1994}.

If $T=-F$ is the negative flip, $\CL_{-F}(H)$ is generated by a representation of the CAR algebra, see \cite{Foit:1983,BaumgartelJurkeLledo:2002} for results in this case.

Another special case is $T=0$. These von Neumann algebras have been introduced by Voiculescu \cite{Voiculescu:1985}, initially in the case where $H$ is the closed real span of an orthonormal basis of~$\Hil$. In this case, $\CL_{0}(H)$ is isomorphic to the group von Neumann algebra of the free group on $\dim\Hil$ generators, which explains their central importance in free probability \cite[Thm.~2.6.2]{VoiculescuDykemaNica:1992}. Shlyakhtenko studied $\CL_{0}(H)$ for more general spaces $H$ \cite{Shlyakhtenko:1997}, and in particular showed that $\CL_{0}(H)$ is a factor for any standard subspace $H$. When $\Delta_H$ is non-trivial and $\dim\Hil>1$, these are full factors of type III$_{\la}$, $0<\la\leq1$ \cite[Thm.~6.1]{Shlyakhtenko:1997}.

The scaled flip $T=qF$, $-1<q<1$, has been studied extensively in the literature, as it interpolates between the CCR ($q=1$), CAR ($q=-1$), and free ($q=0$) cases. It has first been considered in the case $H=H'$ \cite{BozejkoSpeicher:1991} and then generalized to the case of general standard subspaces by Hiai \cite{Hiai:2003}. We refer to  \cite{BozejkoKummererSpeicher:1997,Ricard:2003,Shlyakhtenko:2004,Sniady:2004,Nou:2004_2,Junge:2006_2,GuionnetShlyakhtenko:2014,Nelson:2015,BikramMukherjee:2017,BikramMukherjee:2020} for various properties (such as factoriality, type, non-injectivity, (strong) solidity, approximation properties, MASAS) of the von Neumann algebras $\CL_{qF}(H)$.

\smallskip 

In the case of general twists $T$, only very little is known about $\CL_{T}(H)$. The focus of our investigation below is to identify situations in which $\Om$ is cyclic and separating for $\CL_{T}(H)$, and then study properties of (specific) inclusions $\CL_{T}(K)\subset\CL_{T}(H)$ rather than the internal structure of $\CL_{T}(H)$. We begin with the following elementary lemma.

\begin{lemma}\label{lemma:csT}
Let $T\in\Tw$ and $H\subset\Hil$ be a real closed subspace.
	\begin{enumerate}
		\item\label{item:T-cyclicity} If $H$ is cyclic, $\Om$ is cyclic for $\P_{L,T}(H)$ and hence for $\CL_T(H)$.
		\item\label{item:T-separating} If $\Om$ is separating for $\P_{L,T}(H)$, then $H$ is separating.
	\end{enumerate}
\end{lemma}
\begin{proof}
	\ref{item:T-cyclicity} Let us prove that $\overline{\P_{L,T,n}(H)\Om}$ contains $\Hil_{T,n}$ by induction in $n$, where $\P_{L,T,n}(H)$ denotes the linear space of all polynomials in the fields $\phi_{L,T}(h)$, $h\in H$, of degree up to $n$. The case $n=0$ is trivial. For the induction step, fix $n\in\Nl$, a vector $\Psi_{n+1}\in\Hil_{T,n+1}$, and $\eps>0$. Since we have\footnote{A remark on notation: We use square brackets to indicate the equivalence classes in the quotient \eqref{quot}, also extended to direct sums over several particle numbers. Indices $n$ denote projection to the $n$-particle subspace. For example, $[\phi_{L,T}(h_1)\phi_{L,T}(h_2)\Om]_2=[h_1\ot h_2]$ denotes the $2$-particle component of the equivalence class $[h_1\ot h_2]$.} $$[\phi_{L,T}(h_1)\cdots\phi_{L,T}(h_{n+1})\Om]_{n+1}=[h_1\ot\ldots\ot h_{n+1}]$$ and $H$ is cyclic, it follows that we find a field polynomial $Q_{n+1}\in\P_{L,T,n+1}(H)$ of degree $n+1$ such that $\|[Q_{n+1}\Om]_{n+1}-\Psi_{n+1}\|<\eps$. By the induction assumption, there exists $Q_n\in\P_{L,T,n}(H)$ such that $$\|Q_n\Om-(Q_{n+1}\Om-[Q_{n+1}\Om]_{n+1})\|<\eps.$$ Hence $\tilde{Q}_{n+1}:=Q_{n+1}-Q_n\in\P_{L,T,n+1}(H)$ satisfies $\|\tilde{Q}_{n+1}\Om-\Psi_{n+1}\|\leq2\eps$.

	This shows that $\Om$ is cyclic for $\P_{L,T}(H)$, and by standard arguments (only required in case $\phi_{L,T}(h)$ is unbounded), one sees that it is also cyclic for $\CL_{T}(H)$.

	\ref{item:T-separating} Let $h,\tilde h\in H$ and $h=i\tilde h\in H\cap iH$. Then $h=\phi_{L,T}(h)\Om=i\phi_{L,T}(\tilde h)\Om=i\tilde h$; and since $\Om$ separates $\P_{L,T}(H)$ by assumption, we find $\phi_{L,T}(h)=i\phi_{L,T}(\tilde h)$. Taking adjoints, this implies $\phi_{L,T}(h)=-i\phi_{L,T}(\tilde h)$ and hence $\phi_{L,T}(h)=0$, i.e. $h=0$.
\end{proof}

This motivates us to restrict attention to $\CL_T(H)$ with $H$ a standard subspace. In this case $\Om$ is cyclic for $\CL_T(H)$ with no further conditions on $T$. However, in general it is not separating, as we will see below. If $\Om$ {\em is} separating, $(\CL_{T}(H),\Om)$ defines modular data $\Delta_{L,T,H}$ and $J_{L,T,H}$. When $T$ and $H$ are clear from the context, we will use the shorter notation
\begin{align}\label{eq:modulardata}
	\Delta := \Delta_{L,T,H},\qquad J:=J_{L,T,H},\qquad S:=J\Delta^{1/2}.
\end{align}

\section{Twisted Araki-Woods algebras and standard vectors}\label{sect:TwistedAlgebrasAndStandardness}

The aim of this section is to derive necessary and sufficient conditions on $T$ and~$H$ for $\Om$ being cyclic and separating for $\CL_{T}(H)$, and compute the corresponding modular data $J$, $\Delta$. Most of our analysis will be based on the following compatibility assumption between $T$ and $H$.

\begin{definition}\label{def:compatible}
	Let $H\subset\Hil$ be a standard subspace. The twists {\em compatible with $H$} are the elements of
	\begin{align}\label{def:TwH}
		\Tw(H) &:= \{T\in\Tw\colon [\Delta_H^{it}\ot\Delta_H^{it},T]=0\;\text{for all }\, t\in\Rl\},
		\\
		\Tws(H) &:= \Tw(H)\cap \Tws.
	\end{align}
\end{definition}

The advantage of a twist $T$ and a standard subspace $H$ being compatible is that this ensures the existence of the unitaries $\Gamma_T(\Delta_H^{it})$ as discussed in Lemma~\ref{lemma:T-second-quantization}. In the context of quantum field theory, such an assumption occurs naturally when asking a twist given by a two-particle S-matrix to be Poincaré invariant \cite{LechnerSchutzenhofer:2013}. In our present abstract setup, the main consequence of compatibility is that the one-particle restrictions of the modular data $J,\Delta$ \eqref{eq:modulardata} of $(\CL_{T}(H),\Om)$ (if they exist, i.e. if $\Om$ is separating) coincide with the modular data of $H$, as we will show below in Lemma~\ref{lemma:Delta_H=Delta}.

Before this lemma, we introduce some notation. Since we often times deal with analytic functions on strips, we denote strip regions in $\Cl$ by
\begin{align}
	\Strip_a := \{z\in\Cl \col \min\{0,a\}<\Im(z)<\max\{0,a\}\},\qquad a\in\Rl.
\end{align}
For concise formulations, we also introduce the vector space
\begin{align}
	\HB(\Strip_a)
	:=
	\{f:\overline{\Strip_a}\to\Cl\,\text{ continuous and bounded}\,:\,f|_{\Strip_a} \text{ is analytic}\},
\end{align}
which is a Banach space w.r.t. $\|f\|_\infty=\sup\{|f(z)|\,:\,z\in\overline{\Strip_a}\}$ and can be viewed as $C_b(\overline{\Strip_a})\cap{\mathbb H}^\infty(\Strip_a)$. Elements of $\HB(\Strip_a)$ are uniquely determined by their restriction to $\Rl$ or $\Rl+ia$, and we will therefore identify functions on $\Strip_a$ with their boundary values. For example, given $f:\Rl\to\Cl$ we write $f\in\HB(\Strip_a)$ to express that $f$ is the restriction of a function in $\HB(\Strip_a)$ to~$\Rl$.

\begin{lemma}\label{lemma:Delta_H=Delta}
	Let $T\in \Tw(H)$, and assume that $\Omega\in\Hil$ is separating for $\CL_{T}(H)$. Then the modular data $J,\Delta$ of $(\CL_{T}(H),\Om)$ satisfy
	\begin{align}
		\Delta\restr{\Hil\cap \D(\Delta)}=\Delta_H,\qquad J\restr{\Hil} = J_H.
	\end{align}
\end{lemma}
\begin{proof}
	As $T$ lies in $\Tw(H)$, the operators $U(t):=\Gamma_T(\Delta_H^{it})$, $t\in\Rl$, are well-defined unitaries on $\CF_T(\Hil)$ (Lemma~\ref{lemma:T-second-quantization}~\ref{item:T-second-quantization}) and form a strongly continuous one-parameter group fixing the vacuum vector $\Om$.

	For any polynomial $Q\in \P_{L,T}(H)$, we have $U(t)QU(-t)\in \P_{L,T}(H)\subset \CL_{T}(H)$ by Lemma~\ref{lemma:T-second-quantization}~\ref{item:T-second-quantization-covariance} and $\Delta_H^{it}H=H$. Hence $[U(t)QU(-t),A']=0$ for all $A'\in \CL_{T}(H)'$. By taking limits one concludes that $U(t)\cdot U(-t)$ defines a ${}^\star$-automorphism of $\CL_{T}(H)$. Therefore, the Tomita operator $S$ of $(\CL_{T}(H),\Om)$ satisfies $$SU(t)A\Omega=(U(t)A U(-t))^\star\Omega=U(t)A^\star\Omega=U(t)SA\Omega,\qquad A\in\CL_{T}(H),$$ and therefore $U(t)$ commutes with $J$ and $\Delta$. Since $U(t)$ commutes with $\Delta$ and $J$, the modular operator $\Delta_H$ commutes with $\Delta|_{\Hil\cap \D(\Delta)}$ and $J|_{\Hil}$.

	As $\Delta_H$ and $\Delta|_{\Hil\cap \D(\Delta)}$ commute, there is a common core $D \subset \D(\Delta_H)\cap \D(\Delta)$ for these two operators. Now, for $k_1, h_1$ from this core $D$, the function defined by  $f(t)=\ip{k_1}{U(t)h_1}_T=\ip{k_1}{\Delta_H^{it}h_1}$ belongs to $\HB(\Strip_{-1})$ and satisfies the KMS boundary condition $f(-i)=\ip{S_H h_1}{S_Hk_1}=\ip{S h_1}{Sk_1}=\langle k_1,\Delta h_1\rangle$, where we have used that the $S$ and $S_H$ coincide on $\Dom(S_H)$. On the other hand, $f(-i)=\ip{k_1}{\Delta_H h_1}$. Hence $\Delta h_1 = \Delta_H h_1$ for all $h_1\in D$. Hence $\Delta_H=\Delta|_{\Hil\cap \D(\Delta)}$ and $J|_{\Hil} = J_H$.
\end{proof}

For compatible twists, the standardness property of $\Om$ turns out to be encoded in two key properties of the twist $T$: The Yang-Baxter equation ($T$ being braided) and a ``crossing symmetry''. We discuss the relation of these properties to scattering theory later (see Remark~\ref{remark:physical-crossing}), and first give a mathematical formulation suitable for our setup. Note that in this definition and various calculations below, the $T$-independent scalar products $\scpr$ on $\Hil^{\ot n}$ are used.

\begin{definition}\label{def:crossing}{\bf (crossing symmetry)}\\
	Let $H\subset\Hil$ be a standard subspace. A bounded operator $T\in\B(\Hil\ot\Hil)$ is called {\em crossing-symmetric w.r.t. $H$} if for all $\psi_1,\ldots,\psi_4\in \Hil$, the function
	\begin{align}\label{eq:crossing-function}
		T^{\psi_2,\psi_1}_{\psi_3,\psi_4}(t)
		:=
		\ip{\psi_2\otimes \psi_1}{(\Delta_H^{it}\otimes 1)T(1\otimes \Delta_H^{-it})(\psi_3\otimes \psi_4)}
	\end{align}
	lies in $\HB(\Strip_{1/2})$ and, $t\in\Rl$,
	\begin{equation}
	\begin{aligned}\label{eq:Crossingboundary}
		T^{\psi_2,\psi_1}_{\psi_3,\psi_4}(t+\tfrac{i}{2})
		&=
		\ip{\psi_1\otimes J_H\psi_4}{(1\otimes \Delta_H^{it})T(\Delta_H^{-it}\otimes 1)(J_H\psi_2\otimes \psi_3)}\\
		&=\overline{(T^*)^{J_H \psi_2,\psi_3}_{\psi_1,J_H\psi_4}(t)}
		.
	\end{aligned}
	\end{equation}
\end{definition}
This definition is motivated from quantum field theory and generalizes the notion of crossing symmetry from scattering theory to a setting of standard subspaces. We postpone the discussion of this relation to Remark~\ref{remark:physical-crossing} below. For the time being, suffice it to say that from an operator-algebraic perspective, Def.~\ref{def:crossing} is clearly reminiscent of the KMS / modular boundary condition characterizing the modular group of a standard vector. As we shall see in Thm.~3.12~a), crossing symmetry is a consequence of $\Om$ being separating for $\CL_T(H)$.

In case $T\in\Tw(H)$ is a twist that is crossing symmetric w.r.t. a standard subspace $H$ and compatible with $H$, we also use the notation
\begin{align}
	T(t)
	&:= (\Delta_H^{it}\otimes 1)T(1\otimes \Delta_H^{-it})
	=
	(1\otimes \Delta_H^{-it})T(\Delta_H^{it}\otimes 1)
	=T(-t)^\ast
	,\quad\nonumber
	T\in\Tw(H),
\end{align}
so that the boundary condition of crossing symmetry takes the form
\begin{align}\label{crossagain}
	T^{\psi_2,\psi_1}_{\psi_3,\psi_4}(t+\tfrac{i}{2})
	=\overline{(T^*)^{J_H \psi_2,\psi_3}_{\psi_1,J_H\psi_4}(t)}=
	T^{\psi_1,J_H\psi_4}_{J_H\psi_2,\psi_3}(-t).
\end{align}

Our notion of compatibility between $T$ and $H$ only involves the modular unitaries $\Delta_H^{it}$ and not the modular conjugation $J_H$. However, in the presence of crossing symmetry a compatibility between $T$ and $J_H$ is automatic:

\begin{lemma}\label{lemma:Crossing-J}
	If a bounded selfadjoint operator $T$ is crossing symmetric and compatible with a standard subspace $H$, then
	\begin{align}
		FTF=(J_H\ot J_H)T(J_H\ot J_H),
	\end{align}
	where $F$ is the tensor flip on $\Hil\ot\Hil$.
\end{lemma}
\begin{proof}
	For arbitrary vectors $\psi_k$, the function $t\mapsto T^{\psi_2,\psi_1}_{\psi_3,\psi_4}(t)$ lies in $\HB(\Strip_{1/2})$, with boundary value at $\Rl+\frac{i}{2}$ given by \eqref{crossagain}. The value of $s\mapsto T^{\psi_1,J_H\psi_4}_{J_H\psi_2,\psi_3}(s)$ at $s=\frac{i}{2}$ on the one hand coincides with $T^{\psi_2,\psi_1}_{\psi_3,\psi_4}(0)$, and on the other hand coincides with $T^{J_H \psi_4,J_H \psi_3}_{J_H \psi_1,J_H \psi_2}(0)$. Comparing these expectation values yields the claim.
\end{proof}

The main results on the standardness of $(\CL_{T}(H),\Om)$ that we will derive for compatible twists are:
\begin{itemize}
	\item Theorem~\ref{thm:separating-necessary}: $\Om$ separating $\CL_{T}(H)$ implies that $T$ is crossing symmetric and braided.
	\item Theorem~\ref{thm:YBCrossingLocalSeparating}: $T$ crossing symmetric and braided implies that $\Om$ is separating for $\CL_{T}(H)$.
	\item Proposition~\ref{prop:modulardata}: Computation of the modular data $J,\Delta$ of $(\CL_{T}(H),\Om)$ in case $T$ is crossing symmetric and braided.
\end{itemize}

These results are obtained through preparatory work based on several results related to the KMS condition \cite[Sect.~5.3.1]{BratteliRobinson:1997}, which we establish in the following technical section.

\subsection{Analytic continuations of twisted $n$-point functions}

For $\xi\in\D_H$, we consider the field operators 
\begin{align*}
	\phi_{L,T}^H(\xi)
	:=
	a_{L,T}^\star(\xi)+a_{L,T}(S_H\xi).
\end{align*}
Note that $\phi_{L,T}^H(\xi)$ differs from $\phi_{L,T}(\xi)$ \eqref{eq:phiLT} by the Tomita operator $S_H$ in the argument of the annihilation operator. For $\xi=h\in H$, we have $S_Hh=h$ and both operators coincide. For general $\xi\in\D_H$, the Tomita operator is necessary if we want $\phi_{L,T}^H(\xi)$ to be affiliated with $\CL_T(H)$ and $\Om$ separating for this algebra: If $X:=a_{L,T}^\star(\xi)+a_{L,T}(\eta)$ is affiliated with $\CL_T(H)$ and $\Om$ is separating, then the Tomita operator $S$ of $(\CL_T(H),\Om)$ restricts to $S_H$ on $\Hil\cap\D_H$ (Lemma~\ref{lemma:Delta_H=Delta}), and hence we have
\begin{align*}
    \eta
    =
    (a_{L,T}(\xi)+a_{L,T}^\star(\eta))\Om
    =
    X^\star\Om
    =
    S_HX\Om
    =
    S_H\xi.
\end{align*}
The expectation values of these operators will be denoted
\begin{align}
	W_{2n}^{\xi_1,\ldots,\xi_{2n}} :=
	\langle\Om,\phi_{L,T}^H(\xi_1)\cdots\phi_{L,T}^H(\xi_{2n})\Om\rangle_T
	,\qquad
	\xi_1,\ldots,\xi_{2n}\in\D_H.
\end{align}
For an odd number of fields, these expectation values vanish. Expanding the definitions of $\phi_{L,T}^H(\xi)$ and $\scpr_T$, one finds that $W_{2n}^{\xi_1,\ldots,\xi_{2n}}$ can be written as a sum of $1\cdot3\cdot\ldots\cdot(2n-1)$ terms of the form $\langle\Om,A_1\cdots T_l\cdots A_{2n})\Om\rangle$, where the $A_j$ are either (untwisted) creation operators $a_{L}^*(\xi_j)$ or annihilation operators $a_L(S_H\xi_j)$, and $T_l$ denotes various insertions of twists, coming from the $a_{L,T}(S_H\xi_k)$ and $\scpr_T$. The combinatorial aspects of these terms are best captured in a diagrammatic form which was already introduced in \cite{BozejkoSpeicher:1994} in a special case. We present and further develop this diagrammatic form in the appendix (Section~\ref{sect:diagrams}).

In the present section, we do not rely on the diagram notation in our proofs, but still regard it as a helpful tool to keep track of the various contributions to $W_{2n}^{\xi_1,\ldots,\xi_{2n}}$, for instance the 15 terms of $W_6^{\xi_1,\ldots,\xi_{6}}$ that we will need below. The reader is invited to refer to Section~\ref{sect:diagrams} as required.

\bigskip 

We will be interested in a parameter-dependent version of $W_{2n}^{\xi_1,\ldots,\xi_{2n}}$, namely
\begin{align}\label{def:W2n}
	W_{2n}^{\xi_1,\ldots,\xi_{2n}}(t):=W_{2n}^{\xi_1,\ldots,\Delta_H^{it}\xi_{2n}}
	=
	\langle\Om,\phi_{L,T}^H(\xi_1)\cdots\phi_{L,T}^H(\Delta_H^{it}\xi_{2n})\Om\rangle_T
	.
\end{align}

\begin{lemma}\label{lemma:KMS}
	Let $H\subset\Hil$ be a standard subspace, $T\in\Tw(H)$ a compatible twist, and assume that $\Om$ is separating for $\CL_{T}(H)$. Then, for any $n\in\Nl$ and any $\xi_1,\ldots,\xi_{2n}\in\D_H$, the function $W_{2n}^{\xi_1,\ldots,\xi_{2n}}$ lies in $\HB(\Strip_{-1})$ and satisfies
	\begin{align}
		W_{2n}^{\xi_1,\ldots,\xi_{2n}}(t-i)
		&=
		W_{2n}^{\Delta_H^{it}\xi_{2n},\xi_1,\ldots,\xi_{2n-1}}
		=:
		(W_{2n}^{\xi_1,\ldots,\xi_{2n}})'(t),\qquad t\in\Rl.
		\label{def:W2crossn}
	\end{align}
\end{lemma}
\begin{proof}
	Since the vectors $\xi_k$ lie in $\D_H$, it follows that $A:=\phi_{L,T}^H(\xi_{2n})$ and $B:=\phi_{L,T}^H(\xi_{1})\cdots\phi_{L,T}^H(\xi_{2n-1})$ are closable operators with closures affiliated to $\CL_{T}(H)$. By assumption, $\Om$ is separating, so by the KMS condition, $f(t):=\langle\Om, B\Delta^{it}A\Om\rangle_T$ has the analyticity and boundedness properties stated in the lemma, and boundary value $f(t-i)=\langle\Om,A\Delta^{it}B\Om\rangle_T$, $t\in\Rl$. The claim of the lemma now follows by observing that $A\Om$ and $A^\star\Om$ lie in the single particle space $\Hil$, on which $\Delta^{it}$ coincides with $\Delta_H^{it}$ by Lemma~\ref{lemma:Delta_H=Delta}.
\end{proof}

In the following, we will explore properties of $T$ that are consequences of $\Om$ being separating for $\CL_{T}(H)$. For doing so, we need to analyze $W_{2n}^{\xi_1,\ldots,\xi_{2n}}$ for $n=1,2,3$. To lighten our notation, we will often use shorthand notation and denote the functions \eqref{def:W2n} and \eqref{def:W2crossn} by $W_{2n}$ and $W_{2n}'$, respectively, leaving the dependence on the fixed vectors $\xi_1,\ldots,\xi_{2n}\in\D_H$ implicit. We will also refer to $W_{2n}$ as the {\em $(2n)$-point functions} because of their similarity to correlation functions in Wightman QFT.

The most basic continuation result is the following.

\begin{lemma}\label{lemma:2pt}
	Let $H\subset \Hil$ be a standard subspace and $\xi_1,\xi_2\in\D_H$. Then the function $t\mapsto\langle S_H\xi_1,\Delta_H^{it}\xi_2\rangle$ lies in $\HB(\Strip_{-1})$ and evaluates at $t=-i$ to $\langle S_H\xi_2,\xi_1\rangle$.
\end{lemma}
\begin{proof}
	In case $\Om$ is separating, this is exactly Lemma~\ref{lemma:KMS} for $n=1$. In case $\Om$ is not separating, the statement follows from basic properties of modular theory for standard subspaces: $\langle \Delta_H^{1/2}S_H\xi_1,\Delta_H^{1/2}\xi_2\rangle=\langle J_H\xi_1,\Delta_H^{1/2}\xi_2\rangle=\langle S_H\xi_2,\xi_1\rangle$.
\end{proof}

Subsequently we will often be concerned with analytic functions of the form \eqref{eq:crossing-function} or similar, namely expectation values of an operator-valued function in tensor products between various vectors. The following lemma will be helpful to extend such functions in their vector arguments.

We write $B(\Rl,\B(\Hil))$ for the set of bounded functions $A:\Rl \to \B(\Hil)$ equipped with the natural norm $\|A\|_\infty=\sup_{t\in \Rl}\{\|A(t)\|\}$.

\begin{lemma}\label{lemma:extension}
	Let $Z_i: \Dom(Z_i)\to \Hil$, $i=1,2$ be closed antilinear operators, $D_i\subset \Dom(Z_i)$ subsets of the topological space $\Dom(Z_i)$ provided with the graph norm, $E_1,E_2 \subset \Hil^{\ot n}$, $\CF\subset B(\Rl,\B(\Hil^{\ot(n+1)}))$ subsets, and $$f: \CF \times D_1 \times E_1 \times D_2 \times E_2\to \HB(\Strip_{\alpha}).$$
	If there exists a continuous map $\Theta:\CF\to B(\Rl,\B(\Hil^{\ot(n+1)}))$ such that
	\begin{align} \label{eq:boundary1}
		f(R,\xi_1, \Psi, \xi_2, \Phi)(t)&=\ip{\xi_1\ot \Psi}{R(t) \Phi \ot \xi_2},\\
		f(R,\xi_1, \Psi, \xi_2, \Phi)(t+i\alpha)&=\ip{\Psi \ot Z_2\xi_2}{\Theta(R)(t) Z_1\xi_1\ot\Phi},\label{eq:boundary2}
	\end{align}
then $f$ is separately continuous. In particular, $f$ can be uniquely extended to $$\overline{\Span \CF}  \times \overline{\Span D_1} \times \overline{\Span E_1} \times \overline{\Span D_2} \times \overline{\Span E_2},$$
with closures taken in the respective topologies, and the extension still satisfies \eqref{eq:boundary1} and \eqref{eq:boundary2}.
\end{lemma}
\begin{proof}
	By linearity and antilinearity, it is obvious that we can extend $f$ to  $\Span \CF\times \Span D_1\times \Span E_1\times \Span D_2\times \Span E_2$. Hence, we can suppose, $\CF$, $D_1,D_2, E_1$, and $E_2$ are vector spaces.
	
	Fix $R$ and, for the sake of notation, let us omit the dependence on $R$ and denote $M=(\|\xi_1\|+\|Z_1\xi_1\|)(\|\xi_2\|+\|Z_2\xi_2\|)\|\Psi\|\|\Phi\|$. By the three lines theorem it follows that, $z\in\Strip_\alpha$,
	\begin{align*}
		|f(\xi_1, \Psi,\xi_2, \Phi)(z)|
		&\leq\max\left\{\sup_{t\in \Rl}\left\{|f(\xi_1, \Psi, \xi_2, \Phi)(t)|\right\}, \sup_{t\in \Rl}\left\{\left|f(\xi_1, \Psi,\xi_2, \Phi)(t+i\alpha)\right|\right\} \right\}\\
		&\leq \left(\sup_{t\in \Rl}\|R(t)\|+\sup_{t\in \Rl}\|\Theta(R)(t)\|\right)M.
	\end{align*}
	Since $f$ depends linearly or antilinearly on all its four variables, it is continuous in them separately in the appropriate topologies. Therefore, we can continuously extend $f$ to $\overline{\Span D_1} \times \overline{\Span E_1} \times \overline{\Span D_2} \times \overline{\Span E_2}$, since $\HB(\Strip_\alpha)$ is a Banach space.

	For the continuity in $R$, let us now omit the dependence on the fixed vectors $\xi_1$, $\Psi$, $\xi_2$, and $\Phi$ and let $M$ be as above.  Again by the three line theorem, it follows that
	\begin{align*}
		&|f(R)(z)-f(S)(z)|\\
		&\leq\max\left\{\sup_{t\in \Rl}\left\{|f(R)(t)-f(S)(t)|\right\}, \sup_{t\in \Rl}\left\{|f(R)(t+i\alpha)-f(S)(t+i\alpha)|\right\} \right\}\\
		&\leq \max\left\{\sup_{t\in \Rl}\|R(t)-S(t)\|,\sup_{t\in \Rl}\|(\Theta(R)-\Theta(S))(t)\|\right\}M,
	\end{align*}
and the conclusion follows from the continuity of $\Theta$.
\end{proof}

For the $n$-point functions (with $n\geq4$), it will be useful to introduce further shorthand notation in order to increase the readability of our formulae. We will abbreviate the vectors $\xi_1,\ldots,\xi_{2n}$ by their indices $1,\ldots,{2n}$, use a bar to denote the action of $S_H$, an index $t$ to denote the action of $\Delta_H^{it}$, and symbols like $a_k$ to denote $a_L(\xi_k)$, the untwisted left annihilation operators \eqref{eq:aL-untwisted}. For example,
\begin{align*}
	\langle\bar1,2\rangle\langle\bar4\ot\bar3,T(5\ot6_t)\rangle
	&=
	\langle S_H\xi_1,\xi_2\rangle\langle S_H\xi_4\ot S_H\xi_3,T(\xi_5\ot\Delta_H^{it}\xi_6)\rangle,
	\\
	\langle\bar4\ot a_3T(\bar2\ot\bar1),T(5\ot6_t)\rangle
	&=
	\langle S_H\xi_4\ot a_L(\xi_3)T(S_H\xi_2\ot S_H\xi_1),T(\xi_5\ot\Delta_H^{it}\xi_6)\rangle.
\end{align*}

After these preparations, we now prove results based on the analyticity of the $4$-point function.
\begin{proposition}\label{lemma:4pta}
	Let $H\subset\Hil$ be a standard subspace, $T\in\Tw(H)$ a compatible twist, and assume that $\Om$ is separating for $\CL_{T}(H)$. Let $\varphi_1,\varphi_2\in\D_H$, $\psi_1,\psi_2\in\Hil$.
	\begin{enumerate}
		\item\label{item:4ptana1} The function
		\begin{align}\label{eq:basefunction1}
			f(t)
			:=
			\langle \varphi_1\ot \psi_1,T(\psi_2\ot\Delta_H^{it}\varphi_2)\rangle
		\end{align}
		lies in $\HB(\Strip_{-1})$ and satisfies
		\begin{align}
			f(t-i)=\langle\psi_1\ot S_H\Delta_H^{it}\varphi_2,T(S_H\varphi_1\ot\psi_2)\rangle,\qquad t\in\Rl.
		\end{align}

		\item\label{item:4ptana2} The vector $a_L(\psi_2)T(\varphi_1\ot\psi_1)$ lies in the domain of~$S_H$, and
		\begin{align}
			S_Ha_L(\psi_2)T(\varphi_1\ot\psi_1)=a_L(\psi_1)T(S_H\varphi_1\ot\psi_2)
			.
		\end{align}
	\end{enumerate}
\end{proposition}

\begin{proof}
	\ref{item:4ptana1} With the abbreviations introduced before, the $4$-point function reads
	\begin{align}
		W_4(t)
		&=
		\langle\bar1,2\rangle\langle\bar3,4_t\rangle
		+
		\langle\bar2,3\rangle\langle\bar1,4_t\rangle
		+
		\langle\bar2\ot\bar1,T(3\ot4_t)\rangle,
	\end{align}
	as follows by expanding the definitions or relying on the diagrammatic rules explained in the appendix. Since the first two terms have the stated analyticity, boundedness and continuity properties, and evaluate at $t=-i$ to $\langle\bar1,2\rangle\langle\bar4,3\rangle$ and $\langle\bar2,3\rangle\langle\bar4,1\rangle$, respectively, comparison with
	\begin{align}
		W_4'(0)=\langle\bar4,1\rangle\langle\bar2,3\rangle+\langle\bar1,2\rangle\langle\bar4,3\rangle+
		\langle\bar1\ot\bar4,T(2\ot3)\rangle
	\end{align}
	and Lemma~\ref{lemma:KMS} shows that $t\mapsto\langle\bar2\ot\bar1,T(3\ot4_t)\rangle$ lies in $\HB(\Strip_{-1})$ and evaluates at $t=-i$ to $\langle\bar1\ot\bar4,T(2\ot3)\rangle$.

	Up to a relabeling of vectors, this function coincides with $f$. Hence we have shown the lemma in case $\psi_1,\psi_2$ lie in $\D_H$. The extension to $\psi_1,\psi_2\in\Hil$ now follows by applying Lemma~\ref{lemma:extension} (put $R(t):=T(1\ot\Delta_H^{it})$, $\Theta(R)(t)=(1\ot\Delta_H^{-it})T$, and $Z_1=Z_2=S_H$).

	\ref{item:4ptana2} We now take $\varphi_2$ to even be an entire analytic vector for $\Delta_H$, so that $f$ is entire analytic, and satisfies at $t=-i$
	\begin{align}
		f(-i)
		=
		\langle \psi_1\ot S_H\varphi_2,T(S_H\varphi_1\ot\psi_2)\rangle
		=
		\langle \varphi_1\ot\psi_1,T(\psi_2\ot\Delta_H\varphi_2)\rangle.
	\end{align}
	This can be rewritten as
	\begin{align}
		\langle a_L(\psi_1)T(S_H\varphi_1\ot\psi_2),S_H\varphi_2\rangle^*
		=
		\langle a_L(\psi_2)T(\varphi_1\ot\psi_1),S_H^*S_H\varphi_2\rangle.
	\end{align}
	This equation holds for any vector $\varphi_2$ which is entire analytic for $\Delta_H$. As $\varphi_2$ ranges over this space, $S_H\varphi_2$ ranges over a core of $\Delta_H^{-1/2}$, and hence a core of~$S_H^*$. This implies the claim.
\end{proof}

\begin{remark}
	We may rephrase part b) in terms of left and right creation and annihilation operators as follows: For any $\psi_1,\psi_2\in\Hil$, the operator
	\begin{align}
		T_{\psi_1,\psi_2}\in\CB(\Hil),\qquad T_{\psi_1,\psi_2}\xi := a_L(\psi_2)T(\xi\ot\psi_1)
	\end{align}
	is an endomorphism of $\D_H$, with
	\begin{align}
		S_HT_{\psi_1,\psi_2}=T_{\psi_2,\psi_1}S_H.
	\end{align}
	In particular, $T_{\psi,\psi}$ (and $T_{\psi_1,\psi_2}+T_{\psi_2,\psi_1}$, etc.) are endomorphisms of the standard subspace $H$. This provides a link between endomorphisms of standard subspaces \cite{LongoWitten:2010} and crossing symmetry.
\end{remark}

Our next analyticity result is based on the $6$-point function $W_6$.

\begin{lemma}\label{lemma:6pt}
	Let $H\subset\Hil$ be a standard subspace, $T\in\Tw(H)$ a compatible twist, and assume that $\Om$ is separating for $\CL_{T}(H)$. Let $\varphi_1,\varphi_2\in\D_H$ and $\Psi,\Psi'\in\Hil\ot\Hil$.
	\begin{enumerate}
		\item\label{item:crossing6pt} The function $t\mapsto\langle S_H\varphi_1\ot\Psi,T_1T_2(\Psi'\ot\Delta_H^{it}\varphi_2)\rangle$ lies in $\HB(\Strip_{-1})$, with value at $t=-i$ given by
		\begin{align}
			\langle \Psi\ot S_H\varphi_2,T_2T_1(\varphi_1\ot\Psi')\rangle.
		\end{align}
		\item\label{item:triple} The function $t\mapsto\langle S_H\varphi_1\ot\Psi,T_2T_1T_2(\Psi'\ot\Delta_H^{it}\varphi_2)\rangle$ lies in $\HB(\Strip_{-1})$, with value at $t=-i$ given by
		\begin{align}
			\langle \Psi\ot S_H\varphi_2,T_2T_1T_2(\varphi_1\ot\Psi')\rangle.
		\end{align}
	\end{enumerate}
\end{lemma}
\begin{proof}
	The first step of the argument is a calculation: By expanding the definitions of $\phi_{L,T}^H(\xi)$ and $\scpr_T$, one finds, $t\in\Rl$,
\begin{align}\label{eq:6pt-split}
	W_6(t)
	=
	w_0(t)+w_1(t)+w_2(t)+w_3(t),
\end{align}
where
\begin{align*}
	w_0(t)
	&=
	\Big(\langle \bar 1,2\rangle\langle \bar3,4\rangle\langle \bar5,6_t\rangle
	+
	\langle \bar2,3\rangle\langle\bar4,5\rangle\langle\bar1,6_t\rangle\Big)
	\\
	&\quad+
	\Big(
	\langle\bar1,2\rangle\langle\bar4,5\rangle\langle\bar3,6_t\rangle
	+
	\langle\bar1,4\rangle\langle\bar2,3\rangle\langle\bar5,6_t\rangle
	+
	\langle\bar2,5\rangle\langle\bar3,4\rangle\langle\bar1,6_t\rangle\Big)
	\\
	w_1(t)
	&=
	\langle\bar1,2\rangle\langle\bar4\ot\bar3,T(5\ot6_t)\rangle
	+
	\langle\bar2,3\rangle\langle\bar4\ot\bar1,T(5\ot6_t)\rangle
	\\
	&\quad+
	\langle\bar3,4\rangle\langle\bar2\ot\bar1,T(5\ot6_t)\rangle
	+
	\langle\bar4,5\rangle\langle\bar2\ot\bar1,T(3\ot6_t)\rangle
	\\
	&\quad+
	\langle\bar2\ot\bar1,T(3\ot4)\rangle\langle\bar5,6_t\rangle
	+
	\langle\bar3\ot\bar2,T(4\ot5)\rangle\langle\bar1,6_t\rangle,
	\\
	w_2(t)
	&=
	\langle\bar4\ot a_3T(\bar2\ot\bar1),T(5\ot6_t)\rangle
	+
	\langle a_4T(\bar3\ot\bar2)\ot\bar1, T(5\ot6_t)\rangle
	\\
	&\quad+
	\langle\bar2\ot\bar1, T(a_{\bar3}T(4\ot5)\ot6_t)\rangle,
	\\
	w_3(t)
	&=
	\langle\bar3\ot\bar2\ot\bar1,T_2T_1T_2(4\ot5\ot6_t)\rangle.
\end{align*}

This result only becomes transparent when considering the diagram notation (Section~\ref{sect:diagrams}). The function $w_k$ collects all terms corresponding to diagrams with $k$ crossings.

Let us denote by $w_0',\ldots,w_3'$ the analogous functions that we get by cyclically permuting  $1,2,\ldots,6\to6,1,\ldots,5$, e.g. the first term of $w_0'(0)$ is $\langle\bar6,1\rangle\langle \bar2,3\rangle\langle \bar4,5\rangle$.

We claim
\begin{align}\label{ana-state}
	w_k\in\HB(\Strip_{-1}),\qquad w_k(-i)=w_k'(0),\qquad k=0,1,2,3.
\end{align}
To prove this claim, we will investigate $w_0,\ldots,w_3$ one by one.

\begin{description}
	\item[$w_0$] Each of the five terms contributing to $w_0(t)$ depends on $t$ via a $2$-point function $\langle\bar j,6_t\rangle$ and hence lies in $\HB(\Strip_{-1})$ and evaluates at $t=-i$ to $\langle\bar 6,j\rangle$ (Lemma~\ref{lemma:2pt}). Using this result, one sees that the first term of $w_0$, evaluated at $t=-i$, coincides with the second term of $w_0'$, evaluated at $t=0$. The other terms of $w_0(-i)$ and $w_0'(0)$ match up similarly via cyclic permutations of the groups of terms in round brackets. Hence \eqref{ana-state} holds for $k=0$.

	\item[$w_1$] For $k=1$, \eqref{ana-state} follows by using Prop.~\ref{lemma:4pta}~\ref{item:4ptana1}. This result (and Lemma~\ref{lemma:2pt} for the last two terms of $w_1$) imply $w_1\in\HB(\Strip_{-1})$. For the boundary values, one checks that the first term of $w_1(-i)$ is $\langle\bar1,2\rangle\langle\bar3\ot\bar6,T(4\ot5)\rangle$, which coincides with the second term of $w_1'(0)$. The behaviour of the other terms is analogous: Working $\mod 6$, the $\ell$-th term of $w_1(-i)$ coincides with the $(\ell+1)$st term of $w_1'(0)$. Hence \eqref{ana-state} holds for $k=1$.

	\item[$w_2$] For $k=2$, \eqref{ana-state} follows by using Prop.~\ref{lemma:4pta}: The first and third term of $w_2$ are seen to lie in $\HB(\Strip_{-1})$ on the basis of part~\ref{item:4ptana1} of that proposition. For the second term, we also need part~\ref{item:4ptana2} which ensures that $a_4T(\bar3\ot\bar2)$ lies in the domain of $S_H$. Using these results, we conclude $w_2\in\HB(\Strip_{-1})$ and can compute the boundary value of all terms. The first term of $w_2$, $\langle\bar4\ot a_3T(\bar2\ot\bar1),T(5\ot6_t)\rangle$, evaluates at $t=-i$ to $\langle a_3T(\bar2\ot\bar1)\ot\bar6,T(4\ot5)\rangle$, which is seen to coincide with the second term of $w_2'(0)$ by direct comparison. The second term of $w_2(-i)$ is
	\begin{align*}
		\langle\bar1\ot\bar6,T(S_Ha_4T(\bar3\ot\bar2)\ot5)\rangle
		&=
		\langle\bar1\ot\bar6,T(a_{\bar2}T(3\ot4)\ot5)\rangle,
	\end{align*}
	which coincides with the third term of $w_2'(0)$. Similarly, the third term of $w_2(-i)$ coincides with the first term of $w_2'(0)$. Hence \eqref{ana-state} holds for $k=2$.

	\item[$w_3$] We use \eqref{eq:6pt-split}. As we have already shown that $W_6$, $w_0$, $w_1$, $w_2$ lie in $\HB(\Strip_{-1})$ and evaluate at $t=-i$ to their primed counterparts $W_6'(0)$, $w_0'(0)$, $w_1'(0)$, $w_2'(0)$, respectively, we conclude that $w_3\in\HB(\Strip_{-1})$, and~$w_3(-i)=w_3'(0)$.
\end{description}

	We now prove the two claims a) and b) made in the lemma.

	The function in part a) and its claimed boundary value coincide with the second term of $w_2$ and the third term of $w_2'(0)$ under the identifications $\varphi_1=\xi_3$, $\Psi=S_H\xi_2\ot S_H\xi_1$, $\varphi_2=\xi_6$, $\Psi'=\xi_4\ot\xi_5$, respectively. Hence we have already proven a) for vectors $\Psi,\Psi'$ that are pure tensors of vectors from $\D_H$. The general case follows from Lemma~\ref{lemma:extension}.

	The function in part b) and its is claimed boundary value coincide with $w_3(t)$ and $w_3'(0)$, respectively, under the identifications $\varphi_1=\xi_3$, $\Psi=S_H\xi_2\ot S_H\xi_1$, $\Psi'=\xi_4\ot\xi_5$, $\varphi_2=\xi_6$. Similarly as in a), this implies that b) holds for $\Psi,\Psi'$ of tensor product form, from which the general result follows by Lemma~\ref{lemma:extension}.
\end{proof}

We finish this section proving some analytic properties of $n$-crossing functions from crossing-symmetry. Recall that for a twist $T$ crossing-symmetric w.r.t. a standard subspace $H$, we had defined $T(t)=(\Delta_H^{it}\ot1)T(1\ot\Delta_H^{-it})$. Also recall the tensor leg notation from footnote~\ref{fn:tensors}, e.g. $T(t)_1=(\Delta_H^{it}\ot1\ot1)(T\ot1)(1\ot\Delta_H^{-it}\ot1)$ on three tensor factors.

\begin{proposition}\label{prop:ncrossing}
	Let $T\in\Tw$ be crossing symmetric w.r.t. a standard subspace $H\subset\Hil$. Then, for $\xi, \xi^\prime \in \Hil$,
	\begin{align*}
		f(t):&=\langle \xi\ot\Psi_n,T(t)_1\cdots T(t)_n (\Phi_n\ot \xi')\rangle
	\end{align*}
belongs to $\HB(\Strip_{1/2})$ and satisfies
\begin{align}\label{eq:ncrossingboundary}
	f(t+\tfrac{i}{2})=\langle \Psi_n\ot J_H\xi',T(t)_n^\ast\cdots T(t)_1^\ast \left(J_H\xi\ot\Phi_n\right)\rangle,\qquad t\in\Rl.
\end{align}
\end{proposition}
\begin{proof}
	In order to exploit the crossing symmetry of $T$, we initially choose the vectors $\Psi_n=\psi_1\ot\ldots\ot\psi_n$, $\Phi_n=\varphi_1\ot\ldots\ot\varphi_n$ to be pure tensors.	Considering an orthonormal basis $(e_k)_{k\in\Nl}$ of $\Hil$, the tensor structure of $T(t)_1\cdots T(t)_n$ allows us to rewrite the above scalar product as
	\begin{align}\label{eq:T1TnExpantion}
		f(t)&=\sum_{k_1,\ldots,k_{n+1}} \ip{e_{k_1}}{\xi}\prod_{j=1}^{n}T^{e_{k_j},\psi_j}_{\varphi_j,e_{k_{j+1}}}(t) \ip{\xi'}{e_{k_{n+1}}}.
	\end{align}

	Each partial sum is analytic in $\Strip_{1/2}$ due to the assumption of $T$ being crossing symmetric and, to conclude that $f\in \HB(\Strip_{1/2})$, it is enough to show the partial sums of the series above are uniformly Cauchy. By defining $Q_{N,M}$ and $Q^\prime_{N,M}$ to be the projections onto $\Span\{e_j \ | \ N\leq j\leq M \}$  and $\Span\{J_H e_j \ | \ N\leq j\leq M \}$, respectively, and denoting $(X)_k=\mathbbm{1}^{\ot(k-1)} \ot X \ot \mathbbm{1}^{\ot (n+1-k)} $ for $X\in \B(\Hil)$ as usual, one can reverse the expansion to see that
	\begin{align*}
		f_{N,M}(t):&=\sum_{k_1,\ldots,k_{n+1}=N}^M \ip{e_{k_1}}{\xi}\prod_{j=2}^{n-1}T^{e_{k_j},\psi_j}_{\varphi_j,e_{k_{j+1}}}(t) \ip{\xi^\prime}{e_{k_{n+1}}}\\
		&=\langle (Q_{N,M} \xi)\ot\Psi_n,A_{N,M}(t)(\Phi_n\ot (Q_{N,M} \xi')\rangle,
		\end{align*}
	where $A_{N,M}(t)=T(t)_1(Q_{N,M})_2 T(t)_2 \cdots (Q_{N,M})_nT(t)_n $. This situation is similar to Lemma~\ref{lemma:extension}, because as $N,M\to\infty$, the operator $(Q_{N,M})_1 A_{N,M}(t) (Q_{N,M})_{n+1}$ goes uniformly (in $t$) to zero in the strong operator topology, which is enough to guarantee that the series is uniformly Cauchy on the real line.

	Using crossing symmetry \eqref{eq:Crossingboundary} to check the behavior of $f_{N,M}$ on the upper boundary of the strip $\Strip_{1/2}$, we have
	\begin{align*}
		f_{N,M}(t+\tfrac{i}{2})
		&=\sum_{k_1,\ldots,k_{n+1}=N}^M \overline{\ip{\xi}{e_{k_1}}}\prod_{j=1}^{n}\overline{T^{J_H e_{k_j},\varphi_j}_{\psi_j,J_He_{k_{j+1}}}(t)}\overline{\ip{e_{k_{n+1}}}{\xi^\prime}}\\
		&=\overline{\langle (Q_{N,M}\xi)\ot\Phi_n,\Theta(A_{N,M})(t)^\ast (\Psi_n\ot (Q_{N,M}\xi'))\rangle}\\
		&=\langle \Psi_n\ot (Q_{N,M}\xi'),\Theta(A_{N,M})(t) \left((Q_{N,M}\xi)\ot\Phi_n\right)\rangle,
	\end{align*}
	where $\Theta(A_{N,M})(t)=T(t)^\ast_n(Q^\prime_{N,M})_nT(t)^\ast_{n-1}\cdots  (Q^\prime_{N,M})_2T(t)^\ast_1$ and we choose this notation again to stress the similarity with Lemma~\ref{lemma:extension}.

	For the same argument as above, this implies that the partial sums of \eqref{eq:T1TnExpantion} converges uniformly on $\Rl+\frac{i}{2}$ to $t\mapsto \langle \Psi_n\ot J_H \xi,T(t)^\ast_n\cdots T(t)^\ast_1 \left(J_H \xi^\prime\ot\Phi_n\right)\rangle$. Hence, it follows by the three lines theorem that \eqref{eq:T1TnExpantion} converges uniformly on the closure of the strip $\Strip_{1/2}$ and, therefore $f$ lies in $\HB(\Strip_{1/2})$ and satisfies \eqref{eq:ncrossingboundary}.
\end{proof}

\subsection{Standardness and modular properties of $(\CL_{T}(H),\Om)$}

We now apply the results of the previous section to study necessary and sufficient properties of~$T$ for $\Om$ to be separating for $\CL_{T}(H)$.

\begin{theorem}\label{thm:separating-necessary}
	Let $H\subset\Hil$ be a standard subspace, $T\in\Tw(H)$ a compatible twist, and assume that $\Om$ is separating for $\CL_{T}(H)$. Then
	\begin{enumerate}
		\item\label{item:proof-crossing} $T$ is crossing symmetric in the sense of Def.~\ref{def:crossing}.
		\item\label{item:proof-ybe} $T$ satisfies the Yang-Baxter equation
		\begin{align}\label{eq:YBE}
			T_1T_2T_1 = T_2T_1T_2,
		\end{align}
		i.e. $T$ is a braided twist.
	\end{enumerate}
\end{theorem}
\begin{proof}
	\ref{item:proof-crossing} Let $\psi_1,\psi_2,\varphi_1,\varphi_2\in\Hil$, with $\varphi_1$ and $\varphi_2$ entire analytic for $\Delta_H$. We consider the entire analytic function $g:\Cl\times\Cl\to\Cl$,
	\begin{align}
		g(t,s)
		:=
		\langle \Delta_H^{-i\bar s}\varphi_1 \ot \psi_1,T(\psi_2\ot\Delta_H^{-it}\varphi_2)\rangle,
	\end{align}
	Note that the restriction of $g$ to the diagonal, $h(t):=g(t,t)$, coincides with the function appearing in the crossing symmetry condition \eqref{eq:crossing-function}.

	According to Prop~\ref{lemma:4pta}~\ref{item:4ptana1} (observe that $t$ has opposite signs in $g$ and \eqref{eq:basefunction1}), we have, $t,s\in\Cl$,
	\begin{align*}
		g(t+\tfrac{i}{2},s)
		&=
		g((t-\tfrac{i}{2})+i, s)
		\\
		&=
		\langle\psi_1\ot S_H\Delta_H^{-i(t-i/2)}\varphi_2,T(S_H\Delta_H^{-i\bar s}\varphi_1\ot\psi_2)\rangle
		\\
		&=
		\langle\psi_1\ot J_H\Delta_H^{-it}\varphi_2,T(\Delta_H^{-is-1/2}J_H\varphi_1\ot\psi_2)\rangle.
	\end{align*}
	From this it is apparent that for $t\in\Rl$
	\begin{align*}
		h(t+\tfrac{i}{2})
		&=
		g(t+\tfrac{i}{2},t+\tfrac{i}{2})
		\\
		&=
		\langle\psi_1\ot J_H\Delta_H^{-it}\varphi_2,T(\Delta_H^{-it}J_H\varphi_1\ot\psi_2)\rangle
		\\
		&=
		\langle\psi_1\ot J_H\varphi_2,(1\ot\Delta_H^{it})T(\Delta_H^{-it}\ot1)(J_H\varphi_1\ot\psi_2)\rangle,
	\end{align*}
	in agreement with the claimed boundary value \eqref{eq:Crossingboundary}.

	By straightforward estimates, one also sees that $h$ is bounded on $\overline{\Strip_{1/2}}$. Finally, Lemma~\ref{lemma:extension} can be used to extend $h$ for general vectors $\varphi_2, \varphi_4\in \Hil$.

	\ref{item:proof-ybe} The proof that $T$ is braided relies on Lemma~\ref{lemma:6pt}. Let $\xi_1,\ldots,\xi_6\in\D_H$, and consider the function
	\begin{align*}
		f(t) := \langle S_H\xi_3\ot S_H\xi_2\ot S_H\xi_1,T_2T_1T_2(\xi_4\ot\xi_5\ot\Delta_H^{it}\xi_6)\rangle.
	\end{align*}
	By Lemma~\ref{lemma:6pt}~\ref{item:triple}, $f$ analytically continues to $\Strip_{-1}$, and
	\begin{align}
		f(-i)
		=
		\langle S_H\xi_2\ot S_H\xi_1\ot S_H\xi_6,T_2T_1T_2(\xi_3\ot\xi_4\ot\xi_5)\rangle.
		\label{triple-cont-a}
	\end{align}
	On the other hand, we may rewrite $f$ as
	\begin{align*}
		f(t)
		=
		\langle S_H\xi_3\ot \Psi,T_1T_2(\xi_4\ot\xi_5\ot\Delta_H^{it}\xi_6)\rangle,
	\end{align*}
	with $\Psi=T(S_H\xi_2\ot S_H\xi_1)$. According to Lemma~\ref{lemma:6pt}~\ref{item:crossing6pt}, this gives
	\begin{align}
		f(-i)
		&=
		\langle \Psi\ot S_H\xi_6,T_2T_1(\xi_3\ot\xi_4\ot\xi_5)\rangle
		\nonumber
		\\
		&=
		\langle S_H\xi_2\ot S_H\xi_1\ot S_H\xi_6,T_1T_2T_1(\xi_3\ot\xi_4\ot\xi_5)\rangle
		\label{triple-cont-b}
	\end{align}
	Comparing \eqref{triple-cont-a} and \eqref{triple-cont-b} now shows that matrix elements of $T_2T_1T_2$ and $T_1T_2T_1$ between total sets of vectors in $\Hil^{\ot 3}$ coincide, which implies the Yang-Baxter equation because $T$ is bounded.
\end{proof}

This result shows that $\Om$ being separating for $\CL_{T}(H)$ is a strong condition on the twist. It might also explain why other classes of twists (see part~\ref{item:Tsmallnorm} and \ref{item:Tpositive} of Thm.~\ref{theorem:T}) have not received as much attention as the braided case. The following examples illustrate this.

\begin{example}\label{ex:Tw(H)CrossingBraided}
	Let $A,B\in\B(\Hil)$ be selfadjoint operators, $F$ the tensor flip on $\Hil\ot\Hil$ as before, $H\subset \Hil$ a standard subspace, and
	\begin{align}
		T:=A\ot B, \qquad\tilde T:=F(A\ot A).
	\end{align}
	Then $T$ is a twist in various cases -- for example, if $\|A\|\|B\|\leq\frac{1}{2}$ (by Thm.~\ref{theorem:T}~\ref{item:Tsmallnorm}), if $A,B$ are positive (by Thm.~\ref{theorem:T}~\ref{item:Tpositive}), or if $A=q E,B=\tilde E$, where $E,\tilde E$ are commuting orthogonal projections and $-1\leq q\leq 1$ (then~$T$ solves the Yang-Baxter equation and has norm $|q|\leq1$, so Thm.~\ref{theorem:T}~\ref{item:Tbraided} applies). Moreover, $T$ is compatible with~$H$ in case $A$ and $B$ commute with the modular unitaries~$\Delta^{it}_H$.

	However, $T$ is braided only in the last case ($T=qE\ot\tilde E$), and satisfies crossing symmetry only if $q=0$ or $E=\tilde E$ is a one-dimensional projection onto a vector that is an eigenvector of $\Delta_H$ and $J_H$. This can be proved by writing out the function $f$ \eqref{eq:crossing-function} for vectors $\xi_1,\ldots,\xi_4$ analytic for $\Delta_H$ and comparing its crossing boundary value \eqref{eq:Crossingboundary}, which here takes the form $f(\frac{i}{2})=q\langle\xi_2,J_HE\xi_1\rangle\langle J_H\tilde E\xi_3,\xi_4\rangle$ with the form $f(\frac{i}{2})=q\langle\xi_2,E\Delta_H^{-1/2}\xi_3\rangle\langle\tilde E\Delta_H^{1/2}\xi_1,\xi_4\rangle$ that one obtains from directly continuing the modular groups $\Delta_H^{it}$ to $t=\frac{i}{2}$.

	In comparison, 	$\tilde T$ is a braided twist if $\|A\|\leq1$, because it is then a selfadjoint solution to the Yang-Baxter equation. If, in addition $[\Delta_H^{it},A]=0=[J_H,A]$, then it is also crossing symmetric because
	\begin{align*}
		T^{\psi_2,\psi_1}_{\psi_3,\psi_4}(t)=\ip{\psi_1}{A\psi_3}\ip{\psi_2}{A\psi_4}=\overline{T^{J_H \psi_2,\psi_3}_{\psi_1,J_H\psi_4}(t)}.
	\end{align*}

	Hence, $T\in \Tw(H)$  is a compatible crossing-symmetric braided twist in case $[\Delta_H^{it},A]=0=[J_H,A]$.
\end{example}

The second example ($\tilde T$) above can be generalized to braided crossing symmetric twists arising from symmetric twists coming from solutions of the Yang-Baxter equation with spectral parameter. This also provides the link of our terminology to the crossing symmetry of scattering theory.

\begin{example}\label{example:physical-crossing}
	Recall that on the Hilbert space $\Hil=L^2(\Rl\to\CK)$ (where $\CK$ is another Hilbert space), we have symmetric twists of the form \eqref{eq:TS}
	\begin{align}\label{eq:TS2}
		(T_S\psi)(\te_1,\te_2)=S(\te_2-\te_1)\psi(\te_2,\te_1),
	\end{align}
	where $S:\Rl\to\B(\CK\ot\CK)$ is a measurable bounded function with $S(-\te)=S(\te)^*$ almost everywhere, satisfying the Yang-Baxter equation with spectral parameter.

	Let $L\subset L^2(\Rl,d\te)$ denote the standard subspace with $$(\Delta_L^{it}\varphi)(\te)=\varphi(\te-2\pi t),\qquad (J_L\varphi)(\te)=\overline{\varphi(\te)}$$ (see, for example, \cite[Sect.~4]{LechnerLongo:2014}). Let $K\subset\CK$ be the closed real span of an orthonormal basis of $\CK$. Then $K$ is a standard subspace with $K=K'$ (such standard subspaces are called maximally abelian in analogy to the von Neumann algebraic situation). Then the closed real tensor product of these spaces, $$H:=L\ot K\subset L^2(\Rl)\ot\CK\cong\Hil$$ is a standard subspace in $\Hil$. Since $\Delta_K=1$ and $\Delta_L^{it}$ acts by translation, we see that $T_S$ \eqref{eq:TS2} is compatible with $H$.

	To evaluate the crossing symmetry condition, consider vectors of the form $\xi_k=\varphi_k\ot v_k$, $k=1,\ldots,4$, with $\varphi_k\in L^2(\Rl,d\te)$ and $v_k\in\CK$. Then the function \eqref{eq:crossing-function} takes the form
	\begin{align*}
		f(t)
		&=
		\int \overline{\varphi_2(\te_1)}\varphi_4(\te_1)\overline{\varphi_1(\te_2)}\varphi_3(\te_2)
		\langle v_2\ot v_1,S(\te_2-\te_1+2\pi t)\,v_4\ot v_3\rangle\,d^2\te.
	\end{align*}
	It has a bounded analytic continuation to the strip $\Strip_{1/2}$, with upper boundary value
	\begin{align*}
		f(t+\tfrac{i}{2})
		&=
		\int \overline{\varphi_2(\te_1)}\varphi_4(\te_1)\overline{\varphi_1(\te_2)}\varphi_3(\te_2)
		\langle v_1\ot J_Kv_4,S(\te_1-\te_2-2\pi t)\,v_3\ot J_K v_2\rangle\,d^2\te
	\end{align*}
	This readily implies that the matrix-valued function $S$ has bounded analytic continuation to the strip $\Strip_{\pi}$, with boundary value 
	\begin{align}\label{eq:scatter-crossing}
		\langle v_2\ot v_1,S(t+i\pi)\,v_4\ot v_3\rangle
		=
		\langle v_1\ot J_Kv_4,S(-t)\,v_3\ot J_K v_2\rangle.
	\end{align}
	Hence it is clear that there are many functions $S:\Rl\to\B(\CK\ot\CK)$ for which $T_S$ is a symmetric twist, but $\Om$ fails to be separating for $\CL_{T_S}(H)$.
\end{example}

\begin{remark}\label{remark:physical-crossing}
 In quantum field theoretic scattering theory, crossing symmetry is a property stating that the scattering amplitude of particles is related to the amplitude of the corresponding antiparticles by analytic continuation \cite[Sect. IV]{Martin:1969_2}. Whereas this property has not been proven in general quantum field theory (see \cite{BrosEpsteinGlaser:1965} for a proof of crossing for two-particle amplitudes, and \cite{Mizera:2021} for recent work towards general crossing conditions in perturbative QFT), it is well established -- and often taken as an axiom -- in integrable QFT on two-dimensional Minkowski spacetime \cite{Iagolnitzer:1978,Smirnov:1992,AbdallaAbdallaRothe:2001,Schroer:2010}.

In the setting of Example~\ref{example:physical-crossing}, our abstract form of crossing symmetry specializes to the crossing symmetry of scattering theory in integrable models, with $S$ playing the role of elastic two-body scattering matrix, $\te$ is the rapidity and the modular conjugation $J_K$ corresponds to conjugating a particle into an antiparticle \cite{AlazzawiLechner:2016}. See also \cite{BischoffTanimoto:2013,HollandsLechner:2018} for previous work relating standard subspaces and crossing symmetry, and \cite{Niedermaier:1998} for a proof of the cyclic formfactor equation, related to crossing symmetry, from modular theory.
\end{remark}

\medskip
We now proceed to show that the Yang-Baxter equation and crossing symmetry are not only necessary, but also {\em sufficient} conditions for $\Om$ being separating for~$\CL_{T}(H)$. This amounts to establishing a large commutant of $\CL_{T}(H)$.

It is instructive to first look at the case of zero twist $T=0$, with $\CF_0(\Hil)$ the full Fock space over $\Hil$. In this case, one has in addition to the ``left'' creation and annihilation operators \eqref{eq:aL-untwisted} also ``right'' creation and annihilation operators, namely ($\xi,\psi_j\in\Hil$)
\begin{align}
	a_R^*(\xi)(\psi_1\ot\ldots\ot\psi_n) &:= \psi_1\ot\ldots\ot\psi_n\ot\xi,\\
	a_R(\xi)(\psi_1\ot\ldots\ot\psi_n) &:= \langle\xi,\psi_n\rangle\cdot \psi_1\ot\ldots\ot\psi_{n-1}.
	\label{eq:aR-untwisted}
\end{align}
It is easy to see that the left field opeartors $\phi_{0,L}(h)=a_L(h)+a_L^*(h)$ and right field operators $\phi_{0,R}(h')=a_R(h')+a_R^*(h')$ commute if and only if $\Im\langle h,h'\rangle=0$, i.e. if $h\in H$ and $h'\in H'$ for some standard subspace $H$. Hence $\CL_0(H)'$ contains the von Neumann algebra $\CR_0(H')$ generated by right fields $\phi_{0,R}(h')$, $h'\in H'$. As~$\Om$ is also cyclic for $\CR_0(H')$, it is separating for $\CL_0(H)$.

Furthermore, the natural unitary involution
\begin{align}\label{eq:Y}
	Y:\CF_0(\Hil)&\to\CF_0(\Hil)\\
	Y(\psi_1\ot\ldots\ot\psi_n) &:= \psi_n\ot\ldots\ot\psi_1,
\end{align}
relates the left and right operators according to $Ya_L^\#(\xi)Y=a^\#_R(\xi)$, which leads to $Y\CL_0(H)Y=\CR_0(H)$; this involution enters the modular conjugation of $(\CL_0(H),\Om)$ in this case \cite[p. 341]{Shlyakhtenko:1997}.

\medskip

In comparison, in our general $T$-twisted setting there exist no ``right'' operators. This is due to the fact that the very definition of the Hilbert space $\CF_T(\Hil)$ is biased towards the left because of the appearance of $(1\ot P_{T,n})$ instead of $(P_{T,n}\ot 1)$ in the recursive definition of $P_{T,n+1}$ \eqref{def:Pn}.

However, in the case of a {\em braided} twist, the symmetry between left and right is restored. This observation will be a key ingredient to proving that $\Om$ separates~$\CL_T(H)$.

\pagebreak

\begin{lemma}
	\label{lemma:symm}
	Let $T\in\Tw$ be a braided (but not necessarily symmetric) twist.
	\begin{enumerate}
		\item For any $n\in\Nl$,
		\begin{align}\label{eq:Preversed}
			P_{T,n+1}&=(P_{T,n}\ot1)\tilde R_{T,n+1},\\
			\tilde{R}_{T,n+1}&:=1+T_n+T_nT_{n-1}+\ldots+T_n\cdots T_1.
		\end{align}
		\item The right creation operators
		\begin{align}\label{eq:aR*}
			a_{R,T}^\star(\xi)[\Psi_n] = [\Psi_n\ot\xi],\qquad [\Psi_n]\in\Hil^{\ot n}/\ker P_{T,n},
		\end{align}
		are well-defined on $\CF_T(\Hil)$ and their adjoints $a_{R,T}(\xi)$ are given by
		\begin{align}\label{eq:aR}
			a_{R,T}(\xi)[\Psi_n]
			&=
			[a_R(\xi)\tilde R_{T,n}\Psi_n],
		\end{align}
		where $a_R(\xi)$ is the untwisted right annihilation operator \eqref{eq:aR-untwisted}.
		\item\label{item:Upsilon} Let $ Z:\Dom(Z)\to\Hil$ be a closed (anti)linear operator such that $$[F(Z\ot Z),T]=0.$$ Then, $\hat{Z}_Y:\bigoplus_{n=0}^\infty\Dom(Z)^{\odot n}/\ker(P_{T,n}) \to \Hil_{T,n} $, given by $\hat{Z}_Y([\Psi_n])=[Y Z^{\ot n}\Psi_n]$, for all $\Psi_n \in \Dom(Z)^{\odot n}/\ker(P_{T,n})$ and for all $n \in \Nl$, where $Z^{\ot 0}:=1$ , is a well-defined closable operator on $\CF_T(\Hil)$ whose closure will be denoted $\Gamma_T^Y(Z)$. This operator is (anti)unitary if $Z$ is. If $Z$ is invertible, one has for any $\xi\in\Dom(Z)$
		\begin{align}\label{eq:LRY}
			\overline{\Gamma_T^Y(Z)a_{L,T}^\#(\xi)\Gamma_T^Y(Z^{-1})}
			&=
			a_{R,T}^\#( Z \xi),\qquad \xi\in\Dom(Z).
		\end{align}
	\end{enumerate}
\end{lemma}
\begin{proof}
	a) We begin with a calculation in the group algebras of the symmetric groups $S_n\subset S_{n+1}$, where $S_n$ is identified with the subgroup of permutations of $\{1,\ldots,n+1\}$ leaving $n+1$ fixed. We claim
	\begin{align}\label{eq:CSnFormula}
		\sum_{\pi\in S_{n+1}}\pi = \sum_{\rho\in S_n}\rho\sum_{k=1}^{n+1}\gamma_k,
	\end{align}
	where the $\gamma_k$ are defined by $\gamma_{n+1}=e$ and $\gamma_k=\sigma_n\sigma_{n-1}\cdots\sigma_k$ for $1\leq k\leq n$.

	Indeed, it is easily checked that for any $\pi\in S_{n+1}$, one has $\rho:=\pi\gamma_{\pi^{-1}(n+1)}^{-1}\in S_n$, so any $\pi\in S_{n+1}$ is of the form $\pi=\rho\gamma_k$ for suitable $\rho\in S_n$ and $k\in\{1,\ldots,n+1\}$. As $\gamma_{k'}\gamma_k^{-1}$ lies in $S_n$ if and only if $k=k'$, this representation is unique. This implies the claimed formula \eqref{eq:CSnFormula}.

	In case $T$ is involutive, we have $S_n$-representations given by $\rho_{T,n}(\sigma_k)=T_k$ and \eqref{eq:Preversed} follows immediately. In case $T$ is not involutive, $P_{T,n}$ can be formulated as $P_{T,n}=\sum_{\pi\in S_n}t(\pi)$, where $t:S_n\to\CB(\Hil^{\ot n})$ is the quasi-multiplicative extension of $t(\sigma_k):=T_k$, i.e. $t(\sigma_{i_1}\cdots\sigma_{i_l})=T_{i_1}\cdots T_{i_l}$ for every reduced word $\sigma_{i_1}\cdots\sigma_{i_l}\in S_n$ \cite{BozejkoSpeicher:1991}. This map is well-defined because the $T_k$ satisfy all relations of $S_n$ except~$T_k^2=1$.

	Given a reduced word $w$ representing $\rho\in S_n$, it is easy to check that the word $w\sigma_n\sigma_{n-1}\cdots\sigma_k$, which represents $\rho\gamma_k$, is reduced as well. Thus $t$ maps the left hand side of \eqref{eq:CSnFormula} to $P_{T,n+1}$ and the right hand side of \eqref{eq:CSnFormula} to $(P_{T,n}\ot1)\tilde R_{T,n+1}$, which proves \eqref{eq:Preversed}.

	b) Considering the adjoint of equation \eqref{eq:Preversed}, one gets $P_{T,n+1}=\tilde R^*_{n+1}(P_{T,n}\ot1)$, from which it is then clear that \eqref{eq:aR*} is well-defined, namely $\Psi_n\ot\xi\in\ker P_{T,n+1}$ for $\Psi_n\in\ker P_{T,n}$. We can then calculate the adjoint by
	\begin{align*}
		\ip{[\Psi_{n+1}]}{a_{R,T}^\star(\xi)[\Phi_n]}_T&=\ip{[\Psi_{n+1}]}{[\Phi_n\ot \xi]}_T\\
		&=\ip{\Psi_{n+1}}{P_{T,n+1}\Phi_n\ot \xi}\\
		&=\ip{\Psi_{n+1}}{\tilde{R}_{T,n+1}^\ast (P_{T,n}\ot1)\Phi_n\ot \xi}\\
		&=\ip{a_R(\xi)\tilde{R}_{T,n+1}\Psi_{n+1}}{ P_{T,n}\Phi_n}\\
		&=\ip{[a_R(\xi)\tilde{R}_{T,n+1}\Psi_{n+1}]}{[\Phi_n]}_T,
	\end{align*}
	which implies $a_{R,T}(\xi)[\Psi_{n+1}]=[a_R(\xi)\tilde{R}_{T,n+1}\Psi_{n+1}]$, as claimed.

	c) The argument is similar to that in the proof of Lemma~\ref{lemma:T-second-quantization}. The new element is the appearance of $F$ in the commutation relation $[F(Z\ot Z),T]=0$. With $Y_n:=Y|_{\Hil^{\ot n}}$, one checks $T_k Y_n Z^{\ot n}=Y_nZ^{\ot n}T_{n-k}$, $k=1,\ldots,n-1$. This implies $[Y_nZ^{\ot n},P_{T,n}]=0$. The remaining steps are now the same as in Lemma~\ref{lemma:T-second-quantization}, observing that the appearance of $Y$ transforms $a_L^\#$ into $a_R^\#$.
\end{proof}

For braided twists $T$, we therefore also have right field operators
\begin{align}\label{eq:phiR}
	\phi_{R,T}(\xi)=a_{R,T}^\star(\xi)+a_{R,T}(\xi),\qquad \xi\in\Hil,
\end{align}
and	may use them to generate right versions $\CR_{T}(H)$ of the von Neumann algebras~$\CL_{T}(H)$ considered so far.

\begin{definition}\label{def:MRTH}
	Given a closed real subspace $H\subset\Hil$ and a braided twist $T\in\Tw$, we define the right $T$-twisted Araki-Woods algebra
	\begin{align}
		\CR_{T}(H) := \{\exp(i\phi_{R,T}(h))\,:h\in H\}''
		\subset\CB(\CF_T(\Hil)).
	\end{align}
\end{definition}

\begin{remark}
	\leavevmode
	\begin{enumerate}
		\item In complete analogy to Lemma~\ref{lemma:csT}, one shows that $\Om$ is cyclic for $\CR_{T}(H)$.
		\item The right and left algebras $\CL_T(H)$ and $\CR_T(H)$ seem not to be related by a natural involution in general because the unitary $Y$ \eqref{eq:Y} does typically not define an operator on $\CF_T(\Hil)$. Only in the more specific situation of Lemma~\ref{lemma:symm}~\ref{item:Upsilon} such a symmetry exists.
		\item In principle, the definition of $\CR_T(H)$ does not need $T\in \Tw$ to be a braided twist. Indeed, whenever $\Psi_n \in \ker P_{T,n}$ implies $\Psi_n\ot\xi\in \ker P_{T,n+1}$, the operator $a_{R,T}^\star(\xi)[\Psi_n] = [\Psi_n\ot\xi]$, $[\Psi_n]\in\Hil^{\ot n}/\ker P_{T,n}$ is well defined. This is for example the case if $T\in \Tws$. However, our later proof that $\CL_T(H)$ and~$\CR_T(H')$ commute will crucially depend on $T$ being braided.
	\end{enumerate}
\end{remark}

These remarks suggest that $\Om$ being separating for $\CL_{T}(H)$ can be proved by showing that $\CR_{T}(H^\prime)$ and $\CL_{T}(H)$ commute. Hence, we start analyzing the commutations relations between left and right creation and annihilation operators.

\begin{definition}
	A twist $T\in\Tw$ is called {\em local} w.r.t. a standard subspace $H\subset\Hil$ if $\Psi_n\ot\xi\in \ker P_{T,n+1}$ whenever $\Psi_n \in \ker P_{T,n}$ for every $n\in \Nl$ and
	\begin{align}
		[\phi_{L,T}(h),\phi_{R,T}(h')]=0,\qquad h\in H,h'\in H'.
	\end{align}
\end{definition}

\begin{lemma}\label{lemma:relative-commutation}
	Let $T\in\Tw$ and suppose $\Psi_n\ot\xi\in \ker P_{T,n+1}$ whenever $\Psi_n \in \ker P_{T,n}$ for all $n\in \Nl$. Then, for $\xi,\eta\in\Hil$, and $\Psi_n\in\Hil^{\ot n}$
	\begin{align}
		[a_{L,T}^\star(\xi),a_{R,T}^\star(\eta)] &= 0,
		\qquad
		[a_{L,T}(\xi),a_{R,T}(\eta)] = 0,\\
		[a_{L,T}(\xi),a_{R,T}^\star(\eta)]\Omega&=\ip{\xi}{\eta}\Omega,\\
		[a_{L,T}(\xi),a_{R,T}^\star(\eta)][\Psi_n] &= [a_L(\xi)(T_1\cdots T_n)(\Psi_n\ot\eta)].
	\end{align}
	Furthermore if $T\in\Tw$ is a braided twist, we have in addition
	\begin{align}\label{eq:mixedrelations}
		[a_{R,T}(\xi),a_{L,T}^\star(\eta)] [\Psi_n]&= [a_R(\xi)T_{n}\cdots T_{1}(\eta\ot\Psi_n)],
		\nonumber\\
		[\phi_{L,T}(\xi),\phi_{R,T}(\eta)]\Omega&=2i \Im \ip{\xi}{\eta}\Om,
		\\
		[\phi_{L,T}(\xi),\phi_{R,T}(\eta)][\Psi_n]
		=
		[a_L(\xi)T_1&\cdots T_n(\Psi_n\ot \eta)-a_R(\eta)T_n\cdots T_1(\xi\ot\Psi_n)].
		\nonumber
	\end{align}
\end{lemma}
\begin{proof}
	As discussed above, the condition involving kernels guarantees $a_{R,T}^\star(\xi)$ to exist as an operator on $\CF_T(\Hil)$. It is easy to see that the left and right creation operators commute and, by taking adjoints, the left and right annihilation also commute. For the mixed relation we get
	\begin{align*}
		[a_{L,T}(\xi),a_{R,T}^\star(\eta)][\Psi_n]&=
		[a_L(\xi)R_{T,n+1}(\Psi_n\ot \eta)]-[(a_L(\xi)R_{T,n}\Psi_n)\ot \eta]
		\\
		&=[a_L(\xi)T_1\cdots T_{n}(\Psi_n\ot \eta)],
	\end{align*}
	where we have used \eqref{eq:aL} and $R_{T,n+1} = R_{T,n}\ot 1+T_1\cdots T_n$.

	In case $T\in \Tw$ is a braided twist, Lemma~\ref{lemma:symm} ensures the existence of the right creation operators and fields. One can easily calculate the commutator acting on the vacuum directly. In addition,  we have \eqref{eq:aR} and $\tilde R_{T,n+1}=1\ot \tilde R_{T,n}+ T_n\cdots T_1$. Hence, it follows that
	\begin{align*}
		[a_{R,T}(\eta),a_{L,T}^\star(\xi)][\Psi_n]&=
		[a_R(\eta)\tilde{R}_{T,n+1}(\xi\ot\Psi_n)]-[\xi\ot(a_R(\eta)\tilde{R}_{T,n}\Psi_n)]
		\\
		&=[a_R(\eta)T_{n}\cdots T_{1}(\xi\ot\Psi_n)].
	\end{align*}
	Combining these two mixed commutators immediately yields the claimed formula for $[\phi_{L,T}(\xi),\phi_{R,T}(\eta)][\Psi_n]$.
\end{proof}

One can easily see that whenever the right fields are well defined on $\CF_T(\Hil)$---{\emph i.e.}, $\Psi_n \in \ker P_{T,n}$ implies $\Psi_n\ot\xi\in \ker P_{T,n+1}$ for all $\xi,\eta\in\Hil$, $\Psi_n\in\Hil^{\ot n}$, and $n\in \Nl$---it follows from the fields commutator in Lemma~\ref{lemma:relative-commutation} that $T\in \Tw$ is local if, and only if, for all $\Psi_n \in  \Hil^{\ot n}$, $h\in H$ and $h^\prime \in H^\prime$,
\begin{align}\label{eq:Tlocalfield}
\Hil^{\ot n} \ni	a_L(h)T_1\cdots T_n(\Psi_n\ot h')-a_R(h')T_n\cdots T_1(h\ot\Psi_n)\in \ker P_{T,n}.
\end{align}

Comparing to the zero twist case, this imposes an additional constraint on an operator $T\in \Tw$ to be local. An example worth mentioning is $T=-\id_{\Hil\ot\Hil}$, for which $\ker P_{T,n}=\Hil^{\ot n}$ for all $n>1$, so the condition above is automatically fulfilled. More interesting is however the situation when  kernels are not the whole space. In the next result we give a sufficient condition for a braided twist to be local, in particular, we characterize the local braided strict twists.

\begin{proposition}\label{proposition:nCrossing}
	For a braided twist $T\in\Tw$ to be local w.r.t. a standard subspace $H\subset\Hil$ it is sufficient that
	\begin{align}\label{eq:nCrossing}
		\langle h\ot\Psi_n,T_1\cdots T_n (\Phi_n\ot h')\rangle
		=
		\langle \Psi_n\ot h',T_n\cdots T_1(h\ot\Phi_n)\rangle
	\end{align}
	for all $\Phi_n,\Psi_n\in\Hil^{\ot n}$, $h\in H$, $h'\in H'$, and all $n\in\Nl$.
	Furthermore, in case $T\in \Tws$, \eqref{eq:nCrossing} is also necessary.
\end{proposition}
\begin{proof}
	First notice that \eqref{eq:nCrossing} is equivalent to, for all $\Phi_n\in\Hil^{\ot n}$, $h\in H$, and $h'\in H'$,
	\begin{align*}
		\Hil^{\ot n} \ni a_L(h)T_1\cdots T_n(\Phi_n\ot h')-a_R(h')T_n\cdots T_1(h\ot\Phi_n)=0.
	\end{align*}

	Now, Lemma~\ref{lemma:symm} ensures the existence of the right fields and Lemma~\ref{lemma:relative-commutation} yields
	\begin{align*}
		[\phi_{L,T}(h),\phi_{R,T}(h')][\Phi_n]
		=
		[a_L(h)T_1\cdots T_n(\Phi_n\ot h')-a_R(h')T_n\cdots T_1(h\ot\Phi_n)]=0.
	\end{align*}

	In case $T\in \Tws$ is a strict twist, the kernels are trivial. Hence one obtains also the reverse implication in this case.
\end{proof}

We are now in position for proving the converse of Theorem~\ref{thm:separating-necessary}.

\begin{theorem}\label{thm:YBCrossingLocalSeparating}
	Let $H\subset\Hil$ be a standard subspace and $T\in\Tw(H)$ be a compatible braided twist. Assume that $T$ is crossing symmetric w.r.t.~$H$. Then
	\begin{enumerate}
		\item\label{item:locality} $T$ is local w.r.t. $H$, \emph{i.e.}, $\CR_{T}(H^\prime)\subset \CL_{T}(H)^\prime$;
		\item\label{item:standardness} $\Om$ is cyclic and separating for $\CL_{T}(H)$ and $\CR_{T}(H)$.
	\end{enumerate}
\end{theorem}
\begin{proof}
	a) We will verify \eqref{eq:nCrossing}. Let $h\in H$, $h'\in H'$, $\Psi_n,\Phi_n\in\Hil^{\ot n}$ be arbitrary. Thanks to Proposition~\ref{prop:ncrossing}, the function
	\begin{align*}
		f(t)
		&:=
		\langle J_Hh\ot\Psi_n,T(t)_1\cdots T(t)_n (\Phi_n\ot J_Hh')\rangle
	\end{align*}
	has an analytic extension to the strip $\Strip_{1/2}$ which satisfies
	\begin{align*}
		f(t+\tfrac{i}{2})&=\langle \Psi_n\ot h^\prime,T(t)^\ast_n\cdots T(t)^\ast_1 (h\ot \Phi_n)\rangle,\qquad t\in\Rl.
	\end{align*}

	On the other hand, $T$ being compatible with $H$ (i.e. $[\Delta_H^{it}\ot\Delta_H^{it},T]=0$) implies that we may rewrite $f$ as
	\begin{align*}
		f(t)
		&=
		\langle \Delta_H^{-it}J_Hh\ot\Psi_n,T_1\cdots T_n (\Phi_n\ot \Delta_H^{-it}J_Hh')\rangle.
	\end{align*}
	Choosing $J_Hh$ and $J_Hh^\prime$ to be analytic vectors for the modular automorphism group (this is possible because $H$ and $H'$ are invariant under the spectral projections $E_{(\la,\la^{-1})}$ of $\Delta_H$ corresponding to spectrum in $(\la,\la^{-1})\subset\Rl_+$, $0<\la<1$), we have
	\begin{align*}
		f(t+\tfrac{i}{2})
		&=
		\langle \Delta_H^{-it-\frac{1}{2}}J_Hh\ot\Psi_n,T_1\cdots T_n (\Phi_n\ot \Delta_H^{-it+\frac{1}{2}}J_Hh')\rangle
		\\
		&=
		\langle \Delta_H^{-it}h\ot\Psi_n,T_1\cdots T_n (\Phi_n\ot \Delta_H^{-it}h')\rangle.
	\end{align*}
	For analytic $h,h'$, the claim now follows by comparing the two different expressions for $f(t+\frac{i}{2})$ at $t=0$. For general $h\in H$, $h'\in H'$, the claim follows by approximation (Lemma~\ref{lemma:extension}).

	b) It follows from Lemma~\ref{lemma:csT} and its right version that $\Om$ is cyclic for both~$\CL_{T}(H^\prime)$ and $\CR_{T}(H')$. The inclusion $\CR_{T}(H^\prime)\subset \CL_{T}(H)^\prime$ proven in part~\ref{item:locality} then shows that $\Om$ is separating for $\CL_{T}(H)$.

	Since $\Tw(H)=\Tw(H^\prime)$, part~\ref{item:locality} implies also $\CR_{T}(H)\subset \CL_{T}(H^\prime)^\prime$. By taking commutants, $\CL_{T}(H^\prime)\subset \CR_{T}(H)^\prime$ and it follows that $\Om$ is separating for $\CR_{T}(H)$.
\end{proof}

Combining Theorems~\ref{thm:separating-necessary} and \ref{thm:YBCrossingLocalSeparating}, we obtain the following characterization.

\begin{corollary}
	Let $H\subset \Hil$ be a standard subspace and $T\in \Tw(H)$ a compatible twist. The following are equivalent:
	\begin{enumerate}
		\item $\Om$ is separating for $\CL_T(H)$.
		\item $T$ is braided and crossing symmetric w.r.t. $H$.
		\item $T$ is local w.r.t. $H$.
	\end{enumerate}
\end{corollary}

We now proceed to characterizing the modular data $J$, $\Delta$ of $(\CL_{T}(H),\Om)$ \eqref{eq:modulardata} in terms of the modular data $J_H,\Delta_H$ of $H$. The main difficulty in identifying $J,\Delta$ is the fact that a priori it is not clear that they preserve the particle number grading of $\CF_T(\Hil)$.

We will be working with tensor powers of standard subspaces, and recall that for $H\subset\Hil$ a standard subspace, $H^{\ot n}$ is defined as the closed real linear span of $\{h_1\ot\ldots\ot h_n\,:\,h_1,\ldots,h_n\in H\}$. Then $H^{\ot n}\subset\Hil^{\ot n}$ is a standard subspace, and its modular data are \cite[Prop.~2.6]{LongoMorinelliRehren:2015}
\begin{align}
	S_{H^{\ot n}} = S_H^{\ot n},\qquad J_{H^{\ot n}} = J_H^{\ot n},\qquad \Delta_{H^{\ot n}} = \Delta_H^{\ot n}.
\end{align}
In particular, $\D((\Delta_H^{1/2})^{\ot n})=H^{\ot n}+iH^{\ot n}$.

\begin{lemma}\label{lemma:D-invariant}
	\label{Commutation-InvarianteSub}
	Let $H$ be a standard subspace, $T\in \Tw(H)$ a compatible twist, and $n\in\Nl_{\geq2}$, $k\in\{1,\ldots,n-1\}$. Then $\D((\Delta_H^{1/2})^{\ot n})$ is an invariant subspace for $T_k$.
\end{lemma}
\begin{proof}
	As $T$ is compatible with $H$, we have $[(\Delta_{H}^{it})^{\ot n},T_k]=0$ for all $t\in\Rl$. Hence~$T_k$ commutes with all spectral projections of the generator $\sum_{k=1}^n(\log\Delta_H)_k$ of this unitary one-parameter group, and thus leaves the domain of $(\Delta_H^{1/2})^{\ot n}$ invariant.
\end{proof}

\begin{proposition}\label{prop:modulardata}
	\label{ModularOperator}
	Let $H\subset \Hil$ be a standard subspace.
	\begin{enumerate}
		\item If $T\in\Tw(H)$ is a compatible twist and $\Omega$ is separating for $\CL_{T}(H)$, then the Tomita operator $S$ of $(\CL_{T}(H),\Om)$ is given by
		\begin{align}\label{eq:S-T-second-quantized}
			S = \Gamma_T^Y(S_H),
		\end{align}
		i.e. $S[\xi_1\ot\ldots\ot\xi_n]=[S_H\xi_n\ot\ldots\ot S_H\xi_1]$, $\xi_k\in\D_H$ (see Lemma~\ref{lemma:symm}~c)). The modular conjugation and modular unitaries are given by
		\begin{align}\label{eq:modulardata-second-quantized}
			J = \Gamma_T^Y(J_H),\qquad \Delta^{it} &= \Gamma_T(\Delta_H^{it}),\quad t\in\Rl.
		\end{align}

		\item Conversely, suppose $T\in\Tw$ is a twist, $\Omega$ is cyclic and separating for $\CL_{T}$, and the Tomita operator $S$ of $(\CL_T(H),\Om)$ satisfies $S\restr{\Hil_{T,2}}=\Gamma_T^Y(S_H)\restr{\Hil_{T,2}}$ and $S^\ast\restr{\Hil_{T,2}}=\Gamma_T^Y(S_H^\ast)\restr{\Hil_{T,2}}$. Then $[T,\Delta^{it}_H\ot \Delta^{it}_H ]=0$ and $[T,F(J_H\ot J_H)]=0$; in particular, $T$ is compatible with $H$.
	\end{enumerate}
\end{proposition}
\begin{proof}
	a) Compatibility of $T$ with $H$ and $\Om$ separating $\CL_{T}(H)$ imply that the operators on the right hand sides of \eqref{eq:modulardata-second-quantized} exist as (anti)unitaries on $\CF_T(\Hil)$: The unitaries $\Gamma_T(\Delta_H^{it})$ exist by Lemma~\ref{lemma:T-second-quantization}. Since $T$ is crossing symmetric w.r.t. $H$ by Thm.~\ref{thm:separating-necessary}~a), it follows that $T$ commutes with $F(J_H\ot J_H)$ by Lemma~\ref{lemma:Crossing-J}. Hence $\Gamma_T^Y(J_H)$ exists as an antiunitary involution on $\CF_T(\Hil)$ by Lemma~\ref{lemma:symm}~c). Consequently, $\Gamma_T^Y(S_H)=\Gamma_T^Y(J_H)\Gamma_T(\Delta_H^{1/2})$ \eqref{eq:S-T-second-quantized} exists as an operator on $\CF_T(\Hil)$ as well.

	Let us write $\D_{H,n}:=[H^{\ot n}]+i[H^{\ot n}]\subset\Hil_{T,n}$, where the brackets $[\,\cdot\,]$ indicate the $\|\cdot\|_T$-closure of the quotient by $\ker P_{T,n}$. This subspace is invariant under $R_{T,n}$ by Lemma~\ref{lemma:D-invariant}. Thus we have $a_{L,T}(\xi)\D_{H,n}\subset\D_{H,n-1}$ for any $\xi\in\D_H$. As $a_{L,T}^\star(\xi)\D_{H,n}\subset\D_{H,n+1}$ is clear for any $\xi\in\D_H$, it follows inductively that
	\begin{align}\label{eq:polHn}
		\phi_{L,T}^H(\xi_1)\cdots \phi_{L,T}^H(\xi_n)\Omega \in \hat\D_{H,n}:=\bigoplus_{k=0}^n \D_{H,k},\qquad \xi_1,\ldots,\xi_n\in\D_H,
	\end{align}
	where $\phi_{L,T}^H(\xi)=a_{L,T}^\star(\xi)+a_{L,T}(S_H\xi)$. Now we consider the Tomita operator $S$ of $(\CL_{T}(H),\Om)$ and prove inductively
	 \begin{alignat*}{5}
	  	\hat\D_{H,n}&\subset \D(S),
		\quad
		&S\hat\D_{H,n}&\subset
		\hat\D_{H,n},
		\quad
		&
		S[\xi_1\otimes\ldots \otimes \xi_{n}]&=[S_H\xi_n\otimes\ldots \otimes S_H\xi_{1}],
		\\
		\hat\D_{H',n}&\subset \D(S^\star),
		\quad
		&S^\star\hat\D_{H',n}&\subset
		\hat\D_{H',n},
		\quad
		&
		S^\star[\xi_1'\otimes\ldots \otimes \xi_{n}']&=[S_{H'}\xi_n'\otimes\ldots \otimes S_{H'}\xi_{1}'],
	 \end{alignat*}
	where $\xi_1,\ldots,\xi_n\in\D_H$ and $\xi_1',\ldots,\xi_n'\in\D_{H'}$ are arbitrary. This claim clearly holds for $n=1$, as $S\xi_1=S\phi^H_{L,T}(\xi_1)\Om=\phi^H_{L,T}(\xi_1)^\star\Om=S_H\xi_1$.

	For the induction step, we split off terms with highest particle number,
	\begin{align*}
		 \phi_{L,T}^H(\xi_1)\cdots\phi_{L,T}^H(\xi_n)\Om
		 &=
		 [\xi_1\ot\ldots\ot\xi_n]+
		 E_n^\perp\phi_{L,T}^H(\xi_1)\cdots\phi_{L,T}^H(\xi_n)\Om,
		 \\
		 \phi_{L,T}^H(S_H\xi_n)\cdots\phi_{L,T}^H(S_H\xi_1)\Om
		 &=
		 [S_H\xi_n\ot..\ot S_H\xi_1]+E_n^\perp\phi_{L,T}^H(S_H\xi_n)\cdots\phi_{L,T}^H(S_H\xi_1)\Om.
	\end{align*}
	By induction hypothesis, $E_n^\perp\phi_{L,T}^H(\xi_1)\cdots\phi_{L,T}(\xi_n)\Om$ lies in $\hat\D_{H,n-1}\subset\D(S)$ and is mapped by $S$ to a vector in $\hat\D_{H,n-1}$ with vanishing $n$-particle component. As $\phi_{L,T}^H(\xi_1)\cdots\phi_{L,T}^H(\xi_n)\Om\in\D(S)$ as well and is mapped to $\phi_{L,T}^H(S_H\xi_n)\cdots\phi_{L,T}^H(S_H\xi_1)\Om
	$, we obtain $[\xi_1\ot\ldots\ot\xi_n]\in\D(S)$ and $S[\xi_1\ot\ldots\ot\xi_n]=[S_H\xi_n\ot\ldots\ot S_H\xi_1]$ as claimed. This one sees by observing that for any polynomial $Q$ in the fields $\phi_{L,T}^{H'}(h_k')$, $h_k'\in H'$, of degree up to $n-1$, satisfies
	\begin{align*}
		\langle Q\Om,S[h_1\ot\ldots\ot h_{n-1}]\rangle_T
		=
		\langle [h_1\ot\ldots\ot h_{n-1}],S^\star Q\Om\rangle_T
		=0
	\end{align*}
	We now obtain $\hat\D_{H,n}\subset\D(S)$ by closedness of $S$. But the above argument can be repeated in complete analogy for $S^\star$ and $S_{H'}$ instead of $S$ and $S_H$ (and left exchanged with right), noting that the operators $\phi_{L,T}^{H'}(\xi')$, $\xi'\in\D_{H'}$, are affiliated with the commutant $\CL_{T}(H)'\supset\CR_{T}(H')$ (Theorem~\ref{thm:YBCrossingLocalSeparating}~a)), which completes the induction.

	We have shown $\Gamma_T^Y(S_H)\subset S$ and $\Gamma_T^Y(S_H^\ast)\subset S^\ast$. Thus we also obtain $\Gamma_T^Y(S_H)^\star=\Gamma_T^Y(S_{H'})\subset S^\star$ and hence $\Gamma_T^Y(S_H)=S$.

	The claims about $J$ and $\Delta^{it}$ now follow by polar decomposition (Lemma~\ref{lemma:symm}~c)).

	b) For any $\eta_1, \eta_2 \in H+iH$ and for any $\eta_1^\prime, \eta_2^\prime\in H^\prime+iH^\prime$, we have
	\begin{align*}
		\ip{ \eta_1^\prime\ot \eta_2^\prime}{(1+T) F(S_H\ot S_H)\eta_1 \ot \eta_2}
		&=
		\ip{[\eta_1 \ot \eta_2]}{S^\ast [\eta_1^\prime\ot \eta_2^\prime]}_T\\
		&=\ip{[\eta_1 \ot \eta_2]}{\Gamma_T^Y(S_H^\ast)[\eta_1^\prime\ot \eta_2^\prime]}_T\\
		&=\ip{\eta_1 \ot \eta_2}{(1+T)F(S_H^\ast\ot S_H^\ast)\eta_1^\prime\ot \eta_2^\prime}
		\\
		&=
		\ip{\eta_1^\prime\ot \eta_2^\prime}{Y(S_H\ot S_H)(1+T)\eta_1 \ot \eta_2}.
	\end{align*}
	Hence, $[F(S_H\ot S_H),T]=0$ and, by the polar decomposition, we must have that $[T,\Delta^{it}_H\ot \Delta^{it}_H ]=0=[T,F(J_H\ot J_H)]$.
\end{proof}

This result generalizes several theorems known in special cases \cite{EckmannOsterwalder:1973,LeylandsRobertsTestard:1978,BaumgartelJurkeLledo:2002,Shlyakhtenko:1997,BuchholzLechnerSummers:2011,Lechner:2012}.

\begin{remark}\label{remark:incompatibleK}
	Part b) of this theorem states that a Tomita operator of the ``second quantized form''\eqref{eq:S-T-second-quantized} requires $T$ to be compatible with $H$. Such situations ($\Om$ separating for $\CL_T(H)$ despite $T$ not being compatible with $H$) actually occur, as we explain now.

	Notice that, if $H\subset \Hil$ is a standard subspace, we discussed in Example \ref{ex:Tw(H)CrossingBraided} that for each $n\in \Nl$, $$T:=F\left(E^{\Delta_H}_{(\frac{1}{n},n)}\ot E^{\Delta_H}_{(\frac{1}{n},n)}\right) \in \Tw(H)$$ is a braided crossing symmetric twist, where $E^{\Delta_H}_{(\frac{1}{n},n)}$ are the spectral projections of $\Delta_H$ on the intervals $(\frac{1}{n},n)$. Therefore, $\Omega$ is cyclic and separating for $\CL_T(H)$ thanks to Theorem \ref{thm:YBCrossingLocalSeparating}.
	Let us now take $K\subset H$ another standard subspace. Then, $\Omega$ is also cyclic and separating for $\CL_T(K)$. Suppose that $T\in \Tw(K)$ for every $n \in \Nl$. Then,
	\begin{align*}
		&\Delta_K^{is}E^{\Delta_H}_{(\frac{1}{n},n)}\Delta_K^{-is}\ot \Delta_K^{is}E^{\Delta_H}_{(\frac{1}{n},n)}\Delta_K^{-is}=E^{\Delta_H}_{(\frac{1}{n},n)} \ot E^{\Delta_H}_{(\frac{1}{n},n)}.
	\end{align*}
	This forces $[E^{\Delta_K}_{(\frac{1}{n},n)},E^{\Delta_H}_{(\frac{1}{n},n)}]=0$. Hence, the dense subspace $$D:=\bigcup_{n\in\Nl} E^{\Delta_K}_{(\frac{1}{n},n)}E^{\Delta_H}_{(\frac{1}{n},n)}\Hil\subset \D(\Delta_H^{\frac{1}{2}})\cap \D(\Delta_K^{\frac{1}{2}})$$ is invariant under the action  of $\Delta_K^{is}$ and $\Delta_H^{is}$ for all $s\in \Rl$. Therefore $D$ is a common core for $\log(\Delta_K)$ and $\log(\Delta_H)$. On the other hand, we know that the closed operators $S_K,S_H$ form an extension $S_K\subset S_H$ because $K\subset H$. The existence of a common core then yields $S_K=S_H$, \emph{i.e.}, $K=H$.

	From this we see that, if $\dim \Hil=\infty$, there are examples of a standard subspace~$K$ and of a braided twist $T$ such that $\Omega$ is cyclic and separating for~$\CL_T(K)$, but $T$ is not compatible with $K$ and consequently the Tomita operator of $(\CL_T(K),\Om)$ is different from $\Gamma_T^Y(S_K)$.

	We also note that the properties of a twist $T$ being compatible with a standard subspace, or being crossing symmetric w.r.t. $H$, do in general not pass to sub standard subspaces $K\subset H$.
\end{remark}

As a simple corollary to Proposition \ref{prop:modulardata}, we now obtain a duality between left and right twisted Araki-Woods algebras.

\begin{corollary}\label{commutant}
	Let $H\subset\Hil$ be a standard subspace and $T\in \Tw(H)$ a compatible twist that is braided and crossing-symmetric. Then
	\begin{align}
		\CL_{T}(H)^\prime
		=
		\CR_{T}(H^\prime).
	\end{align}
\end{corollary}
\begin{proof}
	We apply Lemma~\ref{lemma:symm}~\ref{item:Upsilon}, noting that $Z=J_H$ satisfies $[J_H^{\ot 2}F,T]=0$ by Lemma~\ref{lemma:Crossing-J}. In view of the form of the modular conjugation $J$ established in Proposition~\ref{ModularOperator}, we therefore have $J\CL_{T}(H)J=\CR_{T}(J_HH)=\CR_{T}(H')$. As $J\CL_{T}(H)J=\CL_{T}(H)'$ by Tomita's Theorem, the claim follows.
\end{proof}

\section{Inclusions of twisted Araki-Woods algebras}\label{sect:inclusions}

Let $\Hil$ be a Hilbert space and $T$ a braided twist on $\Hil\ot\Hil$. In the previous sections we have constructed two maps
\begin{align}\label{eq:abstract-nets}
	\CL_{T},\CR_T : \text{Std}_T(\Hil) \to \text{vN}(\CF_T(\Hil))
\end{align}
from the set $\text{Std}_T(\Hil)$ of all standard subspaces $H\subset\Hil$ that are compatible with~$T$ to the set $\text{vN}(\CF_T(\Hil))$ of von Neumann subalgebras of $\CB(\CF_T(\Hil))$.

\begin{proposition}\label{prop:abstract-nets}
	The maps~\eqref{eq:abstract-nets} have the following properties.
	\begin{enumerate}
		\item $\CL_{T}(K)\subset \CL_{T}(H)$ and $\CR_{T}(K)\subset \CR_{T}(H)$ for standard subspaces $K\subset H$.
		\item If $T$ is crossing-symmetric and $H\in\text{\rm Std}_T(\Hil)$, then~$\Om$ is cyclic and separating for $\CL_{T}(H)$ and $\CR_T(H)$, and duality holds, i.e. $\CL_{T}(H)'=\CR_{T}(H')$.
		\item Let $H\in\text{\rm Std}_T(\Hil)$ and $U$ a unitary on $\Hil$. If $[U\ot U,T]=0$, then also $UH\in\text{\rm Std}_T(\Hil)$, and
		\begin{align}
			\Gamma_T(U)\CL_{T}(H)\Gamma_T(U)^\star
			&=
			\CL_{T}(UH),\\
			\Gamma_T(U)\CR_{T}(H)\Gamma_T(U)^\star
			&=
			\CR_{T}(UH).
		\end{align}
		If $T$ is crossing symmetric w.r.t. $H$, it is crossing symmetric w.r.t. $UH$.
		\item Let $H\in\text{\rm Std}_T(\Hil)$ and $U$ an antiunitary on $\Hil$. If $[F(U\ot U),T]=0$, then also $UH\in\text{\rm Std}_T(\Hil)$, and
		\begin{align}
			\Gamma_T^Y(U)\CL_{T}(H)\Gamma_T^Y(U)^\star
			&=
			\CR_{T}(UH),\\
			\Gamma_T^Y(U)\CR_{T}(H)\Gamma_T^Y(U)^\star
			&=
			\CL_{T}(UH).
		\end{align}
		If $T$ is crossing symmetric w.r.t. $H$, it is crossing symmetric w.r.t. $UH$.
	\end{enumerate}
\end{proposition}
\begin{proof}
	Part a) is obvious from the definition of these algebras and b) is the content of Thm.~\ref{thm:YBCrossingLocalSeparating} and Cor.~\ref{commutant}. The first part of c) follows by application of Lemma~\ref{lemma:T-second-quantization}, and the second part by application of Lemma~\ref{lemma:symm}~c). The last statement follows easily by observing $\Delta_{UH}=U\Delta_H U^*$ and $J_{UH}=UJ_H U^*$. For d), the action of $\Gamma_T(U)$ on $\CL_T(H)$ and $\CR_T(H)$ follows from Lemma~\ref{lemma:symm}~c), so we only need to explain the crossing symmetry w.r.t. $UH$. For arbitrary $\psi_1,\ldots,\psi_4\in\Hil$, we see that the function
	\begin{align*}
		f(t)&:=\langle\psi_2\ot\psi_1,(\Delta_{UH}^{it}\ot1)T(1\ot\Delta_{UH}^{-it})(\psi_3\ot\psi_4)\rangle
		\\
		&=
		\langle U^*\psi_4\ot U^*\psi_3,(\Delta_{H}^{it}\ot1)T(1\ot\Delta_{H}^{-it})(U^*\psi_1\ot U^*\psi_2)\rangle
	\end{align*}
	lies in $\HB(\Strip_{1/2})$, with
	\begin{align*}
		f(t+\tfrac{i}{2})
		&=
		\langle U^*\psi_3\ot J_HU^*\psi_2,(1\ot\Delta_{H}^{it})T(\Delta_{H}^{-it}\ot1)(J_HU^*\psi_4\ot U^*\psi_1)\rangle
		\\
		&=
		\langle \psi_1\ot J_{UH}\psi_4,(1\ot\Delta_{UH}^{it})T(\Delta_{UH}^{-it}\ot1)(J_{UH}\psi_2\ot \psi_3)\rangle.
	\end{align*}
	Hence $T$ is crossing symmetric w.r.t. $UH$.

	We note that for the crossing symmetry statements in c) and d), the assumption of $T$ being compatible with $H$ is not unnecessary.
\end{proof}

Observe that in case $T=F$ is the tensor flip, both nets coincide, i.e.
\begin{align}
	\CL_F(H)=\CR_F(H).
\end{align}
This easily follows from the fact that in this case, $n!^{-1}P_{F,n}$ is the projection of~$\Hil^{\ot n}$ onto its totally symmetric subspace so that tensor multiplication from the left and right become identical. Let us point out that in general, the two nets are different and do not even form inclusions.

\begin{lemma}\label{X}
	Let $T$ be a braided twist, $H\subset K\in\text{\rm Std}_T(\Hil)$, and $T$ crossing symmetric w.r.t. $H$ and $K$. If $\CL_T(H)\subset\CR_T(K)$, then 
	\begin{enumerate}   
		\item $\phi_{L,T}(\psi)=\phi_{R,T}(\psi)$ for all $\psi\in\Hil$,
		\item $(1+T)(1-F)=0$.
	\end{enumerate}
	If $T$ is a strict twist, $\CL_T(H)\subset\CR_T(K)$ is impossible.
\end{lemma}
\begin{proof}
	We have $\phi_{R,T}(h)\Om=h=\phi_{L,T}(h)\Om$ for any $h\in H$. Since the vacuum separates $\CR_T(K)$ and the field operators are affiliated, we conclude $\phi_{L,T}(h)=\phi_{R,T}(h)$, $h\in H$. The annihilation/creation structure then implies $a_{L,T}^\star(h)=a_{R,T}^\star(h)$, which extends to arbitrary $h\in\Hil$.
	
	This implies in particular 
	\begin{align*}
		[h\ot k]=a^\star_{L,T}(h)a^\star_{L,T}(k)\Om=a^\star_{R,T}(h)a^\star_{R,T}(k)\Om=[k\ot h]
	\end{align*}
	for any $h,k\in\Hil$, and hence $P_{T,2}=1+T$ vanishes on antisymmetric vectors, i.e. $(1+T)(1-F)=0$. For strict twist, $(1+T)$ is invertible, which yields a contradiction.
\end{proof}

As $\CL_T(H)\neq\CR_T(H)$ in general, it is interesting to consider the {\em relative} positions of the $\CL_T$- and $\CR_T$-algebras. This relates in particular to inclusions of twisted Araki-Woods algebras: Given an inclusion of standard subspaces $K\subset H\subset \Hil$, we will consider the corresponding inclusion
\begin{align}\label{eq:twisted-subfactor}
	\CL_{T}(K)\subset\CL_{T}(H)
\end{align}
of von Neumann algebras.\footnote{By taking commutants, one also obtains a ``right'' version $\CR_{T}(H')\subset\CR_{T}(K')$, and in case $[T,F]=0$ (e.g., for $T=qF$), the unitary $Y$ \eqref{eq:Y} satisfies $Y\CL_T(H)Y^\star=\CR_T(H)$ for every standard subspace $H$ (Lemma~\ref{lemma:symm}~\ref{item:Upsilon}), so that the inclusions $\CL_{T}(K)\subset\CL_{T}(H)$ and $\CR_{T}(K)\subset\CR_{T}(H)$ are isomorphic. The investigation of the relation between the $\CL_T$- and $\CR_T$-inclusions for general $T$ is left to future work.
}

The structure of this inclusion depends strongly on the standard subspace inclusion $K\subset H$ as well as the twist~$T$. We here focus on the relative commutants
\begin{align}\label{eq:relcomm}
	\CC_{T}(K,H) := \CL_{T}(K)'\cap\CL_{T}(H) = \CR_T(K')\cap\CL_T(H)
\end{align}
which contains information on the relative position of the $\CL_T$- and $\CR_T$-systems.

For general twist $T$ and standard subspaces $K\subset H$, a detailed analysis of \eqref{eq:relcomm} is quite complicated. For instance, setting $K=H$ this question contains the question whether $\CL_T(H)$ is a factor. For the special twists $T=qF$, this question has been completely settled only very recently -- see \cite{KumarSkalskiWasilewski:2023} and the references therein.

In a few special cases, more is known: If $T=F$ is the tensor flip (Bose case), then $\CC_{F}(K,H)=\CL_{F}(K'\cap H)$, and $\Om$ is cyclic for $\CF_{F}(\Hil)$ if and only if $K'\cap H$ is a cyclic space. This follows from  \cite[Thm. I.3.2]{LeylandsRobertsTestard:1978} or \cite[Lemma 1]{EckmannOsterwalder:1973} by mapping our setting to the Bose Fock space via the unitary $I=\bigoplus_n I_n$. Results about the Fermionic case $T=-F$ can be found in \cite[Prop.~2.5]{Foit:1983} and \cite[Prop.~3.4]{BaumgartelJurkeLledo:2002}.

If $T=T_S\in\CTSym$ is a symmetric (in particular braided) twist coming from a solution $S$ of the Yang-Baxter equation with spectral parameter as in \eqref{eq:TS}, and~$S$ satisfies a number of conditions (including crossing symmetry), then the relative commutant $\CC_{T}(K,H)$ is known to be a type III$_{1}$ factor having $\Om$ as a cyclic vector for certain inclusions $K\subset H$ that arise in quantum field theory (see Section~\ref{section:QFT}) \cite{AlazzawiLechner:2016}.

\medskip

In particular, there seem to exist few results on the inclusions $\CL_{T}(K)\subset\CL_{T}(H)$ in case the twist is not unitary, for instance in case $\|T\|<1$. In this section, we investigate two types of inclusions in which the structure of $\CC_{T}(K,H)$ can be decided for $\|T\|<1$.

\subsection{Half-sided modular inclusions}\label{section:inclusions-twisted}

Half-sided modular inclusions were introduced by Wiesbrock \cite{Wiesbrock:1993-1,Wiesbrock:1993} and are of prominent importance in conformal quantum field theory. They are defined as follows.

\begin{definition}
	An inclusion of two von Neumann algebras $\CN\subset\M$ with a joint cyclic separating vector $\Om$ is called {\em half-sided modular} if the modular unitaries $\Delta^{it}$ of $(\M,\Om)$ satisfy
	\begin{align}
		\Delta^{it}\CN\Delta^{-it} \subset \CN,\qquad t\leq0.
	\end{align}
\end{definition}

For a non-trivial half-sided inclusion $(\CN\subset\M,\Om)$, there exists no conditional expectation $\M\to\CN$ (hence no meaningful notion of index), and the inclusion is not split. Hence no standard methods for investigating the relative commutant of a half-sided inclusion are available.

The structure of a half-sided modular inclusion is closely related to that of a one-dimensional Borchers triple, that we also recall here \cite{BuchholzLechnerSummers:2011}.

\begin{definition}\label{def:BT}
    A one-dimensional Borchers triple $(\M,U,\Om)$ consists of a von Neumann algebra $\M$ on a Hilbert space $\Hil$, a strongly continuous unitary one-parameter group $U$ with positive generator $\Pst$, and a unit vector $\Om\in\Hil$ such that
    \begin{enumerate}
        \item $\Om$ is cyclic and separating for $\M$, and $U(x)\Om=\Om$ for all $x\in\Rl$,
        \item $U(x)\M U(x)^{-1}\subset\M$ for $x\geq0$.
    \end{enumerate}
\end{definition}

Let us recall that for a one-dimensional Borchers triple $(\M,U,\Om)$, Borchers' Theorem \cite{Borchers:1992} asserts that
\begin{align}\label{HJB}
    \Delta^{it}U(x)\Delta^{-it}=U(e^{-2\pi t}x),\quad JU(x)J=U(-x),\qquad t,x\in\Rl,
\end{align}
where $J,\Delta$ are the modular data of $(\M,\Om)$. Thus $U$ extends to a (anti)unitary representation of the affine group of $\Rl$ (the ``$ax+b$ group'').

As a consequence of \eqref{HJB}, the inclusion $\N:=U(1)\M U(-1)\subset\M$ coming from a Borchers triple is half-sided modular. Conversely, given a half-sided modular inclusion $(\N\subset\M,\Om)$, there exists a strongly continuous unitary one-parameter group $U$ such that $(\M,U,\Om)$ is a one-dimensional Borchers triple, and $\N=U(1)\M U(-1)$ \cite{Wiesbrock:1993-1,ArakiZsido:2005}. This one-parameter group is related to the modular unitaries of $(\N,\Om)$ and $(\M,\Om)$ by
\begin{align}
	U(e^{2\pi t}-1) = \Delta_\M^{-it}\Delta_\N^{it}.
\end{align}

We may therefore use the structure of one-dimensional Borchers triples to define half-sided modular inclusions.

We will say that a half-sided inclusion has {\em unique vacuum} if the subspace of $U$-invariant vectors is $\Cl\Om$. In case this condition holds and $\dim\Hil>1$, the von Neumann algebras $\M$ and $\CN$ are (not necessarily hyperfinite) type III${}_1$ factors \cite{Longo:1979,Wiesbrock:1993-1}.

To give examples that are of the form \eqref{eq:twisted-subfactor}, we now consider a standard subspace~$H$ and a twist $T$ with a compatible translation representation.

\begin{definition}\label{def:compatible-translation}
	A {\em translation group compatible with a standard subspace $H\subset\Hil$ and a twist $T\in\Tw$} is a strongly continuous one-parameter group $U(x)$, $x\in\Rl$, such that
	\begin{enumerate}
		\item the generator $\Pst$ of $U(x)$ is positive,
		\item $\ker\Pst=\{0\}$,
		\item $U(x)H\subset H$, $x\geq0$,
		\item \label{item:compatibility-with-T} $[U(x)\ot U(x),T]=0$ for all $x\in\Rl$.
	\end{enumerate}
\end{definition}

\begin{proposition}\label{prop:BT-examples}
	Let $H\subset\Hil$ be a standard subspace, $T\in\Tw(H)$ a compatible braided crossing symmetric twist and $U$ a compatible translation group. Then $(\CL_{T}(H),\Gamma_T\circ U,\Om)$ is a one-dimensional Borchers triple with unique vacuum.
\end{proposition}
\begin{proof}
	Thanks to Def.~\ref{def:compatible-translation}~\ref{item:compatibility-with-T} and Lemma~\ref{lemma:T-second-quantization}, $x\mapsto\Gamma_T(U(x))$ is a unitary one-parameter group on $\CF_T(\Hil)$. It is clear that it has positive generator and fixes the Fock vacuum~$\Om$, and the latter is the only (up to multiples) translation invariant vector because $U(x)$ has no non-zero fixed points.

	Furthermore, we have the covariance relation (cf. Lemma~\ref{lemma:T-second-quantization}~\ref{item:T-second-quantization-covariance})
	\begin{align*}
		\Gamma_T(U(x))\CL_{T}(H)\Gamma_T(U(x))^{-1}=\CL_{T}(U(x)H)\subset\CL_{T}(H)\quad\text{  for }\,x\geq0.
	\end{align*}
	As $\Om$ is cyclic and separating for $\CL_{T}(H)$ by Thm.~\ref{thm:YBCrossingLocalSeparating}, the proof is finished.
\end{proof}

Examples of standard subspaces with compatible translation groups can be easily constructed with the help of positive energy representations of the Poincaré group as we shall discuss in Section~\ref{section:QFT}.

To analyze the relative commutants $\CC_{T}(K,H)$, where $K=U(1)H\subset H$ is given by a compatible translation group $U$, we will make use of a criterion that was recently developed in \cite{LechnerScotford:2022}. Given a one-dimensional Borchers triple $(\M,U,\Om)$, one considers bounded open intervals $I\subset\Rl$ and a family of von Neumann algebras $I\mapsto\CA(I)$ indexed by these intervals, namely
\begin{align}
	\CA((a,b))
	:=
	U(a)\M U(a)^* \cap U(b)\M'U(b)^*,
\end{align}
as well as the {\em algebra at infinity},
\begin{align}\label{def:algebra-at-infinity}
	\CA_\infty := \bigcap_{I} \CA(I)'.
\end{align}
Then the following holds.

\begin{proposition}{\em\cite{LechnerScotford:2022}}\label{prop:LS22}
	Let $(\M,U,\Om)$ be a one-dimensional Borchers triple with unique vacuum. Let $J,\Delta$ denote the modular data of $(\M,\Om)$, and $\CN:=U(1)\M U(1)^{-1}\subset\M$.
	\begin{enumerate}
		\item \label{item:POm-triviality} The relative commutants $\CA((0,x))=\M\cap U(x)\M'U(x)^*$, $x>0$, are trivial if and only if the one-dimensional projection $\POm$ onto $\Cl\Om$ lies in~$\CA_\infty$.
		\item \label{item:weak-limit} If $X\in\N\vee J\N J$ is such that $L:=\wlim\limits_{t\to-\infty}\Delta^{it}X\Delta^{-it}$ exists, then $L\in\CA_\infty$.
	\end{enumerate}
\end{proposition}

We now apply this technique to the half-sided inclusions $\CL_{T}(K)\subset\CL_{T}(H)$, $K:=U(1)H\subset H$, defined by Prop.~\ref{prop:BT-examples}. For $T=F$, the relative commutant $\CC_F(K,H)$ is known to be non-trivial for $K'\cap H\neq\{0\}$, and for various symmetric twists it is expected to be non-trivial.

We begin with a preparatory lemma which does not rely on the half-sided structure. The lemma says that an operator $A$ commuting with left {\em and} right fields is necessarily quite complicated in the sense that $A\Om$ cannot be a finite particle number vector. Here we only need to assume that $T$ is strict (Def.~\ref{def:twist}).

\begin{lemma}\label{lemma:infiniteparticlenumber}
	Let $H, K\subset \Hil$ be standard subspaces with $\dim K\geq 2$ and let $T\in \Tws(H)$ be a strict twist. Let $A\in \CL_{T}(H)'\cap \CR_{T}(K)'$, $A\neq0$. If
	\begin{align}
		n(A):=\sup\{n\in\Nl_0\col [A\Omega]_n\neq0\}<\infty,
	\end{align}
	then $A\in\Cl1$.
\end{lemma}
\begin{proof}
	The supremum $n(A)$ exists because $A\neq0$ and $\Om$ separates the intersection $\CL_{T}(H)'\cap \CR_{T}(K)'$ (it even separates $\CL_{T}(H)'$ because $H$ is standard, see Lemma~\ref{lemma:csT}), hence $A\Om\neq0$. We have to show $n(A)=0$ and proceed by contradiction, assuming $n(A)>0$.

	By assumption, $A$ commutes with $\phi^H_{L,T}(\xi)$ and $\phi^K_{R,T}(\eta)$ for all $\xi\in H+iH$ and all $\eta\in K+iK$, i.e.
	\begin{align*}
		A\xi &= A\phi^H_{L,T}(\xi)\Om = \phi^H_{L,T}(\xi)A\Om = a_{L,T}^\star(\xi)A\Om+a_{L,T}(S_H\xi)A\Om,
		\\
		A\eta &= A\phi^K_{R,T}(\eta)\Om = \phi^K_{R,T}(\eta)A\Om = a_{R,T}^\star(\eta)A\Om+a_{R,T}(S_K\eta)A\Om.
	\end{align*}
	Since $[A\Om]_{n(A)+2}=0$, projecting these equations onto $\Hil_{T,n(A)+1}$ gives (note that $\ker P_{T,m}=0$ for all $m$ because $T\in\Tws$, hence we don't have to project onto a quotient)
	\begin{align*}
		(A\xi)_{n(A)+1}
		&=
		a_{L,T}^\star(\xi)(A\Om)_{n(A)}
		+
		a_{L,T}(S_H\xi)(A\Om)_{n(A)+2}
		=
		\xi\ot(A\Om)_{n(A)},
		\\
		(A\eta)_{n(A)+1}
		&=
		a_{R,T}^\star(\eta)(A\Om)_{n(A)}
		+
		a_{R,T}(S_K\eta)(A\Om)_{n(A)+2}
		=
		(A\Om)_{n(A)}\ot\eta.
	\end{align*}
	As $H$ is standard, for each $\eta\in K+iK$ we can find a sequence $(\xi^\eta_l)_{l\in\Nl}\subset H+iH$ such that $\|\xi^\eta_l-\eta\|_\Hil\to0$ as $l\to\infty$. By continuity and the first of the two equations above, we then have
	\begin{align*}
		0
		&=
		(A\xi^\eta_l)_{n(A)+1}-\xi^\eta_l\ot(A\Om)_{n(A)}
		\longrightarrow
		(A\eta)_{n(A)+1}-\eta\ot(A\Om)_{n(A)}
	\end{align*}
	for $l\to\infty$. Comparing this equation with the second equation above yields
	\begin{align*}
		\eta\ot(A\Om)_{n(A)}=(A\Om)_{n(A)}\ot\eta\quad \text{ for all }\,\eta\in K+iK,
	\end{align*}
	which is impossible because $(A\Om)_{n(A)}\neq0$ and $\dim K>1$.
\end{proof}

\begin{theorem}\label{thm:twisted-hsm}
	Let $H\subset\Hil$ be a standard subspace and $T\in\Tw(H)$ a braided crossing-symmetric twist with $\|T\|<1$. Let $U$ be a translation group compatible with $H$ and $T$, and set $K:=U(1)H\subset H$.

	Then the half-sided modular inclusion $\CL_{T}(K)\subset\CL_{T}(H)$ has trivial relative commutant.
\end{theorem}
\begin{proof}
	The idea is to use Prop.~\ref{prop:LS22}, namely to establish that the vacuum projection $\POm$ lies in the algebra at infinity. The easiest way to achieve this would be to realize it as a weak limit of $\Delta^{it}\phi_{L,T}(k)\phi_{R,T}(k')\Delta^{-it}$ as $t\to-\infty$ (with $k\in K$, $k'\in J_H K$), see Prop.~\ref{prop:LS22}~\ref{item:weak-limit}.

	Hence we take $h_1,h_2\in H$ such that $k:=U(1)h_1\in K$ and $k':=U(-1)J_Hh_2$ satisfy $\langle k,k'\rangle=1$ (which is possible because $H$ is cyclic), and consider the field operators
	\begin{align*}
		\phi_{L,T}(k) &\in\CL_{T}(K)=\Gamma_T(U(1))\CL_{T}(H)\Gamma_T(U(1))^\star,
		\\
		\phi_{R,T}(k')&\in\CR_{T}(J_HU(1)H)=
		J(\Gamma_T(U(1))\CL_{T}(H)\Gamma_T(U(1))^\star)J;
	\end{align*}
	these are bounded operators because $\|T\|<1$.

	In view of the form of the modular unitaries $\Delta^{it}$ of $(\CL_{T}(H),\Om)$ (Prop.~\ref{prop:modulardata}), we have
	\begin{align*}
		\Delta^{it}\phi_{L,T}(k)\phi_{R,T}(k')\Delta^{-it}
		&=
		\phi_{L,T}(\Delta_H^{it}k)\phi_{R,T}(\Delta_H^{it}k')
		\\
		&=
		a_{L,T}^\star(\Delta_H^{it}k)a_{R,T}^\star(\Delta_H^{it}k')
		+
		a_{L,T}^\star(\Delta_H^{it}k)a_{R,T}(\Delta_H^{it}k')
		\\
		&\qquad+
		a_{L,T}(\Delta_H^{it}k)a_{R,T}(\Delta_H^{it}k')
		+
		a_{L,T}(\Delta_H^{it}k)a_{R,T}^\star(\Delta_H^{it}k').
	\end{align*}
	To investigate the limit $t\to-\infty$, we first recall that $\Delta_H$ has purely absolutely continuous spectrum and hence $\Delta_H^{it}\to0$ weakly as $t\to-\infty$  by the  Riemann-Lebesgue--Lemma. This is a consequence of the commutation relations~\eqref{HJB} and $\ker\Pst=\{0\}$, which imply that $\log P$ and $\log\Delta_H$ satisfy canonical commutation relations and are hence a multiple of the Schrödinger representation \cite[Prop.~1.7.1]{Longo:2008_2}.

	This implies that for any $\xi\in\Hil$, we have $a_{L,T}(\Delta_H^{it}\xi)\to0$ and $a_{R,T}(\Delta_H^{it}\xi)\to0$ in the strong operator topology as $t\to-\infty$ (evident on vectors of the form $\psi_1\ot\ldots\ot\psi_n$, then apply uniform boundedness -- note that our assumptions imply that $T$ is strict, i.e. $\Hil^{\ot n}=\Hil_{T,n}$).

	Taking into account the uniform bound \eqref{eq:altnormbound}, we see that the first three summands converge weakly to zero as $t\to-\infty$. With Lemma~\ref{lemma:relative-commutation}, the last term may be rewritten on a vector $\Psi_n\in\Hil^{\ot n}$ as
	\begin{align*}
		a_{L,T}(\Delta_H^{it}k)a_{R,T}^\star(&\Delta_H^{it}k')\Psi_n
		=
		a_{R,T}^\star(\Delta_H^{it}k')a_{L,T}(\Delta_H^{it}k)\Psi_n
		\\
		&\;
		+
		\langle\Delta_H^{it}k,\Delta_H^{it}k'\rangle E_0\Psi_n
		+
		a_L(\Delta_H^{it}k)T_1\cdots T_n(E_0^\perp\Psi_n\ot\Delta_H^{it}k').
	\end{align*}
	Here the first term $a^\star a$ converges weakly to zero as $t\to-\infty$, and the second term equals $E_0\Psi_n$ for all $t$.

	Since $\phi_{L,T}(k)\phi_{R,T}(k')$ is a bounded operator and the unit ball in $\B(\CF_T(\Hil))$ is compact in the weak operator topology, $\Delta^{it}\phi_{L,T}(k)\phi_{R,T}(k')\Delta^{-it}$ has a weak limit for $t\to-\infty$ along a suitably chosen net $(t_j)_j$ going to $-\infty$. By construction, the limit lies in the algebra at infinity $\CA_\infty$ \eqref{def:algebra-at-infinity}.

	Hence we conclude that
	\begin{align*}
		W:=\wlim_{j\to\infty} a_L(\Delta_H^{it_j}k)\bigoplus_{n=1}^\infty T_1\cdots T_n\,a_R^\star(\Delta_H^{it_j}k')E_0^\perp
	\end{align*}
	exists as a bounded operator and $Q:=E_0 + W$ lies in $\CA_\infty$. Moreover, $Q$ is selfadjoint because $\Delta^{it}\phi_{L,T}(k)\phi_{R,T}(k')\Delta^{-it}$ is selfadjoint for any $t$. Clearly, $\Om$ is an eigenvector of $Q$, with eigenvalue $1$. Hence the spectral projection $\Pi^Q_1\in\B(\CF_T(\Hil))$ of $Q$ for eigenvalue $1$ also lies in $\CA_\infty$ and satisfies $\Pi^Q_1\geq E_0$.

	If we even have $\Pi^Q_1=E_0$, Prop.~\ref{prop:LS22}~\ref{item:POm-triviality} immediately concludes the proof of the theorem. But in general, the equality $\Pi^Q_1=E_0$ is not clear and we have to use Lemma~\ref{lemma:infiniteparticlenumber}. In order to be able to do so, we now show that there exists $n_T\in\Nl$ such that $\Pi^Q_1$ satisfies
	\begin{align}\label{eq:E-claim}
		\Pi^Q_1\Hil_{T,n}=\{0\},\qquad n\geq n_T.
	\end{align}
	To prove this claim, it will be advantageous to work with the $T$-independent tensor product norm $\|\cdot\|$ instead of $\|\cdot\|_T$. Recall that $\Hil^{\ot n}=\Hil_{T,n}$ as vector spaces because $\Hil^{\ot n}\subset\Hil_{T,n}$ is dense in the norm $\|\cdot\|_T$ and the two norms $\|\cdot\|$, $\|\cdot\|_T$ are equivalent because of our assumptions on $T$.

	As the operators
	\begin{align}
		W_n(t):=a_L(\Delta_H^{it}k) T_1\cdots T_n\,a_R^*(\Delta_H^{it}k')|_{\Hil^{\ot n}}
	\end{align}
	are well defined operators on $\Hil^{\ot n}$, $n\in\Nl$, with operator norm (w.r.t. the norm $\|\cdot\|$) bounded by $\|W_n(t)\|\leq\|k\|\|k'\|\|T\|^n$, the limit $W_n:=\wlim_{j\to\infty} W_n(t_j)$ also exists in the weak operator topology given by the norm $\|\cdot\|$, i.e. $W_n$ satisfies the same bound on its $\|\cdot\|$-operator norm.

	Now any $\Psi_n\in\Ran(\Pi^Q_1)\cap\Hil^{\ot n}$ satisfies
	\begin{align}
		\|\Psi_n\|=\|Q\Psi_n\|=\|W_n\Psi_n\|\leq\|k\|\|\|k'\|\|T\|^n\|\Psi_n\|.
	\end{align}
	As $\|T\|^n\to0$ for $n\to\infty$, this is only possible for $\|\Psi_n\|=0$ (and hence $\|\Psi_n\|_T=0$) for sufficiently large $n$. This proves the existence of $n_T\in\Nl$ such that \eqref{eq:E-claim} holds.

	To finish the proof of the theorem, let $A\in \CR_{T}(U(1)H')\cap\CL_{T}(H)$ be an element of the relative commutant of the half-sided modular inclusion $\CL_{T}(K)\subset\CL_{T}(H)$. Then $A$ commutes with any element of the algebra at infinity and in particular with $\Pi^Q_1\geq E_0$. So we obtain
	\begin{align*}
		\Pi^Q_1A\Om = A\Pi^Q_1\Om = A\Om.
	\end{align*}
	But the vector $\Pi^Q_1A\Om$ is contained in $\bigoplus_{n=0}^{n_T}\Hil_{T,n}$, whereas Lemma~\ref{lemma:infiniteparticlenumber} tells us that $A\Om$ is not an element of this space unless $A\in\Cl1$.
\end{proof}

The significance of this result in quantum field theory will be explained in Section~\ref{section:QFT}. Here we note that the following instability result: If we take $T=qF$, then Theorem~\ref{thm:twisted-hsm} shows that the half-sided inclusion $\CL_{qF}(K)\subset\CL_{qF}(H)$ has trivial relative commutant for any $-1<q<1$, whereas the relative commutant equals $\CL_F(K'\cap H)$ (which has $\Om$ as a standard vector) for $q=1$ because of the familiar structure of second quantization factors. Such instability phenomena of inclusions under deformations are in line with the observations in \cite{BuchholzLechnerSummers:2011,Tanimoto:2011-1,LechnerScotford:2022}.

\subsection{$L^2$-nuclearity and quasi-split inclusions}\label{section:nuclearity}

So far we have obtained results about (half-sided) inclusions $\CL_{T}(K)\subset\CL_{T}(H)$ of twisted Araki-Woods algebras having trivial relative commutant $\CR_{L,T}(H,K)=\Cl$. We now add a complementary result, showing that other inclusions  of the form $\CL_{T}(K)\subset\CL_{T}(H)$ can also have large relative commutants.

The mechanism behind this discussion is the $L^2$-nuclearity condition of Buchholz, D'Antoni, and Longo \cite{BuchholzDAntoniLongo:2007} and the theory of (quasi-)split inclusions \cite{DoplicherLongo:1984,Fidaleo:2001}.

\begin{definition}
	An inclusion $\N\subset\M$ of von Neumann algebras with joint cyclic separating vector and corresponding modular operators $\Delta_\N$, $\Delta_\M$ is said to satisfy $L^2$-{\em nuclearity} if $\Delta_\M^{1/4}\Delta_\N^{-1/4}$ is of trace class.
\end{definition}

If $\N\subset\M$ satisfies $L^2$-nuclearity, then it is quasi-split, namely there exist type~I factors $\CI$, $\tilde\CI$ such that $\N\ot\Cl\subset\CI\subset\M\ot\tilde\CI$. If $\N$ or $\M$ is a factor, we even have the split property, namely $\N\subset\CI\subset\M$ with some type I factor $\CI$.

A consequence of $\N\subset\M$ being quasi-split is that whenever $\M$ is of type~III, then also the relative commutant $\N'\cap\M$ is of type III. These results are generalizations of \cite{DoplicherLongo:1984} to the non-factor case \cite{Fidaleo:2001}, see \cite[Cor.~3.11]{BikramMukherjee:2020}.

We now state our non-triviality result for the relative commutant $\CC_{T}(K,H)$ under quite strong assumptions. This result is closely related to the work of D'Antoni, Longo, and Radulescu \cite{DAntoniLongoRadulescu:2001} who considered the case $T=0$.

\begin{proposition}\label{prop:L2}
	Let $K\subset H$ be an inclusion of standard subspaces and $T$ a braided twist such that 
	\begin{enumerate}   
		\item $L^2$-nuclearity holds on the standard subspace level, i.e. $\|\Delta_H^{1/4}\Delta_K^{-1/4}\|_1<1$, where $\|\cdot\|_1$ is the trace norm on $\Hil$,
		\item $T$ is crossing-symmetric w.r.t. $H$,
		\item $T$ is compatible with $K$ and $H$, and has norm $\|T\|<1$.
	\end{enumerate}
	Then $\CL_{T}(K)\subset\CL_{T}(H)$ satisfies $L^2$-nuclearity and is quasi-split. If $\CL_{T}(H)$ is of type III, also the relative commutant $\CC_{T}(K,H)$ is of type III.
\end{proposition}
\begin{proof}
	We first note that in view of the assumptions, $\Om$ is cyclic and separating for $\CL_{T}(H)$ and hence also separating for the subalgebra $\CL_{T}(K)$. It is also cyclic for the subalgebra because~$K$ is a standard subspace, so we have $\Om$ as a joint cyclic separating vector.
	
	Thanks to the compatibility assumption for $K$ and $H$, the modular operators $\hat\Delta_K$ and $\hat\Delta_H$ of $\CL_{T}(K)$ and $\CL_{T}(H)$ are of second quantized form (Prop.~\ref{prop:modulardata}). Hence the trace norm of $\hat\Delta_H^{1/4}\hat\Delta_K^{-1/4}$ is
	\begin{align}
		\|\hat\Delta_H^{1/4}\hat\Delta_K^{-1/4}\|_{T,1}
		=
		\sum_{n\geq0}\|(\Delta_H^{1/4}\Delta_K^{-1/4})^{\ot n}\|_{T,1};
	\end{align}
	where $\|\cdot\|_{T,1}$ denotes the trace norm of the trace class in $\B(\Hil_{T,n})$.
	
	As $\|T\|<1$ is a strict twist, the operator $P_{T,n}^{1/2}:\Hil^{\ot n}\to\Hil_{T,n}$ is unitary. Hence $\|(\Delta_H^{1/4}\Delta_K^{-1/4})^{\ot n}\|_{T,1}=\|P_{T,n}^{-1/2}(\Delta_H^{1/4}\Delta_K^{-1/4})^{\ot n}P_{T,n}^{1/2}\|_{1}$, where $\|\cdot\|_1$ denotes the trace norm of the trace class in $\B(\Hil^{\ot n})$. But as $T$ is compatible with $K$ and $H$, the operators $P_{T,n}^{\pm1/2}$ commute with $(\Delta_H^{1/4}\Delta_K^{-1/4})^{\ot n}$ and thus drop out. Now we may use the tensor structure of $\|\cdot\|_1$ to conclude 
	\begin{align}
		\|\hat\Delta_H^{1/4}\hat\Delta_K^{-1/4}\|_{T,1}
		=
		\sum_{n=0}^\infty \|\Delta_H^{1/4}\Delta_K^{-1/4}\|_1^n
		=
		\frac{1}{1-\|\Delta_H^{1/4}\Delta_K^{-1/4}\|_1}<\infty.
	\end{align}
\end{proof}

Let us comment on the assumptions entering into this proposition. The compatibility condition for $T$ with $H$ and $K$ is quite strong in general, but automatically satisfied for $T=qF$, $-1<q<1$. In this case, it is also known from Hiai's work that $\CL_{T}(H)$ is a type III${}_1$ factor if $\Delta_H$ does not have eigenvalue~$1$ \cite[Prop.~3.2]{Hiai:2003}.

The $L^2$-nuclearity assumption on the level of the standard subspaces is also a strong assumption, but it is known to hold (including the norm inequality $\|\Delta_H^{1/4}\Delta_K^{-1/4}\|_1<1$) for certain inclusions of standard subspaces arising QFT (see Section~\ref{section:QFT}).

\section{Applications to quantum field theory}\label{section:QFT}

In this section we explain some applications of our constructions in algebraic quantum field theory \cite{Haag:1996,Advances-AQFT:2015}, where the main object of interest are families of von Neumann algebras labeled by open regions in some spacetime manifold, subject to several physically motivated conditions.

Since we will consider quantum field theories on a variety of different spacetimes, we use an abstract formulation. From the structure of spacetime as a globally hyperbolic Lorentzian manifold $\Man$ we only need that $\Man$ has a family~$\COs$ of ``good'' subsets (to be thought of as the causally complete convex regions), a notion of causal complement, mapping subsets $\CO\in\COs$ to their complements $\CO'\in\COs$ such that $(\CO')'=\CO$ for all $\CO$. Furthermore, we require that there exists a Lie group~$G$ acting on $\Man$ such that $g\CO\in\COs$ for $\CO\in\COs$ and preserving complementation, $g\CO'=(g\CO)'$ for all $g\in G$, $\CO\in\COs$.

In this general setting, we require the existence of a reference region with the following properties.

\begin{description}
	\item[\bf(reference wedge)] There exists $W_0\in\COs$ such that there is a one-parameter group $\la_0(t)\in G$, $t\in\Rl$, and an involution $j_0\in G$ such that $\la_0(t)W_0=W_0$ for all $t\in\Rl$, and $j_0W_0=W_0'$ as well as $j_0\la_0(t)=\la_0(t)j_0$.
\end{description}

We will refer to the reference region $W_0$ as ``reference wedge'' although its geometrical shape can be quite different (see examples below), and to elements of the $G$-orbit $\Ws:=GW_0\subset\COs$ as ``wedges''.	

\begin{description}
	\item[\bf (wedge separation property)] For any $\CO_1,\CO_2\in\COs$ with $\CO_1\subset\CO_2'$, there exists a wedge $W\in\Ws$ such that $\CO_1\subset W\subset\CO_2'$. 
	\item[\bf (wedge intersection property)] For any $\CO\in\COs$, there holds 
	\begin{align}
		\CO = \bigcap_{W\in\Ws:W\supset\CO}W.
	\end{align}
\end{description}

Two examples of this setting are the following.

\begin{enumerate}
	\item Minkowski spacetime $\Man=\Rl^{1+s}$, $s\geq1$, with~$\COs$ the set of all convex causally complete regions, $\CO'$ the causal complement, $G$ the Poincaré group, and $W_0=\{x\in\Rl^{1+s}\,:\,x_1>|x_0|\}$ the Rindler wedge. Here $\la_0(t)$ are the Lorentz boosts in $x_1$-direction with parameter $t$, and $j_0$ is the reflection at the edge of $W_0$.
	\item The lightray $\Man=\Rl$ with $\COs$ the set of all open intervals and half-lines, $\CO'$ the set-theoretic complement, $G$ the affine group of $\Rl$, and $W_0=\Rl_+$. Here $\la_0(t)$ are the dilations $x\mapsto e^tx$ and $j_0$ is the reflection $x\mapsto-x$.
\end{enumerate}

Further examples include the circle $\Man=S^1$ with Möbius group symmetry or de Sitter spacetime with Lorentz group symmetry. These examples are well known, and recently a general group-theoretic formulation has been developed which includes these and further cases \cite{MorinelliNeeb:2021,MorinelliNeebOlafsson:2022_2}.

\medskip

The first step to a quantum field theory on $\Man$ consists in describing the one-particle localization structure, which we shall do by modular localization. As explained in the work of Brunetti, Guido and Longo \cite{BrunettiGuidoLongo:2002} and later generalizations \cite{Mund:2003,MorinelliNeeb:2021}, we consider an (anti)unitary representation $U$ of $G$ on a Hilbert space $\Hil$, namely a strongly continuous group homomorphism $U$ from $G$ to the group of (anti)unitary operators on $\Hil$ such that $U(j_0)$ is antiunitary. Denoting the connected component of the identity by $G^\uparrow\subset G$, we then have $G=G^\uparrow\rtimes{\mathbb Z}_2$, where ${\mathbb Z}_2$ acts by conjugation with $j_0$. We also write $G^\downarrow:=G\backslash G^\uparrow=G^\uparrow j_0$.

Then $t\mapsto U(\la_0(t))=e^{itB}$ is a unitary one-parameter group with some selfadjoint generator $B$, and $U(j_0)$ is an antiunitary involution. Setting
\begin{align}\label{def:H0}
	H_0 := \ker(U(j_0)e^{\frac{1}{2}B}-1),\qquad
	\Delta_{H_0}^{it} = U(\la_0(t)),\qquad J_{H_0} = U(j_0),
\end{align}
we obtain a standard subspace $H_0\subset\Hil$ with the modular data specified above. 

We shall assume that $H_0$ is {\em localized in $W_0$} in the sense that
\begin{align}\label{eq:iso}
	g\in G,\; gW_0\subset W_0 \;\Rightarrow\; U(g)H_0\subset H_0.
\end{align}
This inclusion property is known to be closely linked to positive energy properties of the representation $U$: It holds if $U$ is a positive energy representation of the Poincaré group (in the Minkowski space example a)) \cite{BrunettiGuidoLongo:2002}, or if if the generator of translations in the affine group is positive (lightray example b)). See also \cite{MorinelliNeeb:2021} for a general spectral condition implying \eqref{eq:iso} for $g\in G^\uparrow$.

Once \eqref{eq:iso} holds, we may consistently interpret the vectors in $H_0$ as being localized in $W_0$, and obtain a net of standard subspaces for wedges, namely
\begin{align}\label{www}
	gW_0 \mapsto H(gW_0):=U(g)H_0,\qquad g\in G.
\end{align}
Note that this map is well-defined thanks to the inclusion property \eqref{eq:iso} which implies that the stabilizer group of $W_0$ fixes~$H_0$ in the representation $U$. The standard subspace net has the duality property
\begin{align}\label{eq:H-dual}
	H(W')=H(W)',\qquad W\in\Ws.
\end{align}
For $W=W_0$ this follows by choosing $g=j_0$ in \eqref{www} and observing $U(j_0)H_0=J_{H_0}H_0=H_0'$, and for general wedges by covariance.

To connect to the twisted Araki-Woods algebras, we need to amplify $U$ to a representation on $\CF_T(\Hil)$ for suitable twists $T$. Since we also need $T$ to be braided and crossing symmetric for our standardness results to hold, we add these assumptions here.

\begin{definition}
	Let $U$ be an (anti)unitary representation of $G$ such that \eqref{eq:iso} holds. A twist $T\in\Tw$ is called {\em admissible} if it is braided, {\em $G$-invariant} in the sense that
	\begin{align}
		[U(g)\ot U(g),T]=0,\quad g\in G^\uparrow,\qquad
		[F(J_{H_0}\ot J_{H_0}),T]=0,
	\end{align}
	and crossing-symmetric w.r.t. $H_0$.
\end{definition}
The $G$-invariance condition amounts to $T$ respecting the symmetries of the spacetime under consideration. It is trivially satisfied if $T=qF$, $-1\leq q\leq1$. Note that an admissible twist is automatically compatible with $H_0$ (and hence any $H(W)$, $W\in\Ws$), in the sense of Def.~\ref{def:compatible}.

For admissible $T$, the operators
\begin{align}
	U_T(g):=\Gamma_T(U_1(g)),\quad g\in G^\uparrow, \qquad U_T(j_0):=\Gamma_T^Y(J_H),
\end{align}
are well-defined (anti)unitary operators on $\CF_T(\Hil)$ (Lemma~\ref{lemma:T-second-quantization} and Lemma~\ref{lemma:symm}). It is clear that $U_T$ is a (anti)unitary representation that leaves $\Om$ invariant.

Given a representation $U$ and admissible $T$, we thus obtain a representation $U_T$ of $G$ on $\CF_T(\Hil)$ and two maps
\begin{align}\label{eq:wedgenets}
	W \mapsto \CL_T(H(W)),\qquad
	W \mapsto \CR_T(H(W))
\end{align}
from wedges $W\in\Ws$ to von Neumann algebras on $\CF_T(\Hil)$.

\begin{proposition}\label{prop:wedgenets}
	Let $U$ be a (anti)unitary $G$-representation with the inclusion property \eqref{eq:iso}, and $T$ and admissible twist. Then the nets \eqref{eq:wedgenets} satisfy
	\begin{enumerate}
		\item Isotony: $\CL_T(H(W_1))\subset\CL_T(H(W_2))$ for wedges $W_1\subset W_2$, and analogously for $\CR_T$,
		\item Covariance: For any $W\in\Ws$,
		\begin{align}\label{Gcov2}
			U_T(g)\CL_T(H(W))U_T(g)^{-1}=
			\begin{cases}
				\CL_T(H(gW)) & g\in G^\uparrow\\
				\CR_T(H(gW)) & g\in G^\downarrow
			\end{cases},
		\end{align}
		and analogously with $\CL_T$ and $\CR_T$ exchanged.
		\item Reeh-Schlieder property: $\Om$ is cyclic and separating for $\CL_T(W)$ and $\CR_T(W)$, $W\in\Ws$.
		\item Relative locality: For wedges $W_1\subset W_2'$, we have
		\begin{align}
			\CL_T(H(W_1)) \subset \CR_T(H(W_2))',
		\end{align}
		and analogously with $\CL_T$ and $\CR_T$ exchanged.
	\end{enumerate}
\end{proposition}
\begin{proof}
	a) follows from \eqref{eq:iso} and Prop.~\ref{prop:abstract-nets}~a). Item b) is a consequence of Prop.~\ref{prop:abstract-nets}~c), taking into account $U_T(j_0)=J_{H_0}$. Item c) is contained in Thm.~\ref{thm:YBCrossingLocalSeparating}, and d) follows from
	\begin{align*}
		\CL_T(H(W_1)) \subset \CL_T(H(W_2')) = \CL_T(H(W_2)') = \CR_T(H(W_2))'.
	\end{align*}
\end{proof}

We may interpret this structure as a quantum field theory in two different ways. In the first version, we consider $G^\uparrow$ as our symmetry group. Then both nets $W\mapsto\CL_T(H(W))$ and $W\mapsto\CR_T(H(W))$ are $G^\uparrow$-covariant and {\em relatively local} in the sense of item d) above. However, for $T\neq F$ typically neither $\CL_T$ nor $\CR_T$ is local (see Lemma~\ref{X}).

In the second version, we consider $G$ as our symmetry group. Then the transformation behavior from item b) above imposes an additional constraint in case there exists $g_0\in G^\uparrow$ with $g_0W_0=W_0'$. Namely, in that case we must have $\CL_T(H_0')=\CR_T(H_0')$ for $W\mapsto\CL_T(H(W))$ to be $G$-covariant. Apart from $T=F$, this is typically not the case (see Lemma~\ref{X}). However, this conflict does not occur in several situations of interest, such as two-dimensional Minkowski space, the lightray example, or higher-dimensional Minkowski space with a restricted symmetry group (see \cite{BuchholzSummers:2007} for similar considerations). If $W_0'\not\in G^\uparrow W_0$, we may take
\begin{align}
	gW_0 \mapsto \tilde\CL_{T}(H(gW_0)) := U_T(g)\CL_T(H_0)U_T(g^{-1}),\qquad g\in G,
\end{align}
as the definition of our net, which is then $G$-covariant and local as a consequence of items b) and d) above. Analogously, we could work with $\CR_T$ instead of $\CL_T$.

We will adopt this second point of view because it leads to local nets $\tilde\CL_T$, and refer to \cite{Lechner:2012} for a scheme (for special twists) in which the restriction to low-dimensional spacetime or smaller symmetry groups can be overcome.

The net $\tilde\CL_T$ is so far defined on the set of wedges $\Ws$, but naturally extends to the larger set $\COs$ of localization regions. Here the assumptions of wedge separation and wedge intersection enter.

\begin{proposition}
	Let $U$ be a representation with the inclusion property \eqref{eq:iso}, $T$ an admissible twist, and assume that $G^\uparrow W_0$ does not contain $W_0'$. Then the map
	\begin{align}\label{eq:inter}
	 \COs \ni \CO \longmapsto \CA_{U,T}(\CO) := \bigcap_{W\in\Ws:W\supset\CO}\tilde\CL_T(H(W))
	\end{align}
	is well-defined,
	\begin{enumerate}   
		\item $G$-covariant, namely $U_T(g)\CA_{U,T}(\CO)U_T(g^{-1})=\CA_{U,T}(g\CO)$, $g\in G$, $\CO\in\COs$,
		\item isotonous, namely $\CA_{U,T}(\CO_1)\subset\CA_{U,T}(\CO_2)$ for $\CO_1\subset\CO_2$,
		\item local, namely $\CA_{U,T}(\CO_1)\subset\CA_{U,T}(\CO_2)'$ for $\CO_1\subset\CO_2'$.
	\end{enumerate}
\end{proposition}
\begin{proof}
	The map is well-defined because any $\CO\in\COs$ is contained in a wedge by the wedge intersection property. a) and b) are obvious. For c), note that for $\CO_1\subset\CO_2'$ there exists a wedge $W$ such that $\CO_1\subset W$ and $\CO_2\subset W'$ (wedge separation property). Hence either $W$ or $W'$ lies in the $G^\uparrow$-orbit of $W_0$. In case $W=gW_0$, $g\in G^\uparrow$, we have $\CA_{U,T}(\CO_1)\subset\CL_{T}(H(W))$ and $$\CA_{U,T}(\CO_2)\subset\CL_T(H(W'))=\CL_T(H(W)')=\CR_T(H(W))',$$ which implies the claim. The case $W'=gW_0$, $g\in G^\uparrow$, is analogous.
\end{proof}

Given a representation $U$ and a twist $T$ satisfying our assumptions, we therefore get a net $\CA_{U,T}$ satisfying the basic properties of a quantum field theory. In particular, this construction recovers the free field for $T=F$.

The representation $U$ entering the definition of $\CA_{U,T}$ encodes the single particle structure (such as particle masses and spins), and the twist $T$ is expected to be characteristic of the interaction in the underlying theory. This is well understood in case $T$ comes from a Yang-Baxter solution with spectral parameter (Example~\ref{example:TSym}), in which case it precisely encodes the two-particle scattering operator and leads to a scheme for the operator-algebraic construction of integrable QFTs \cite{Lechner:2008,AlazzawiLechner:2016}. In this case, our notion of crossing symmetry also coincides with the scattering theory definition of crossing symmetry (Remark~\ref{remark:physical-crossing}).

In general, it is expected that the local algebras \eqref{eq:inter} may fail to have $\Om$ as a cyclic vector, or be too small to allow for interesting observables localized in~$\CO$. In some cases, we will give proofs of these expectations below. One should therefore view $\CA_{U,T}$ as a ``germ'' of a quantum field theory (also called ``wedge-local quantum field theory''). The class of all wedge-local models contains both, ``strictly local'' models in which $\Om$ is cyclic for all $\CA_{U,T}(\CO)$, $\CO\in\COs$, as well as non-local models in which this property does not hold.

The challenge to efficiently decide which class a given tuple $U,T$ corresponds to constitutes the QFT part of the motivation for the present article. We here comment on models with $\|T\|<1$ which have not been investigated so far.

\medskip

Typically for a given region $\CO\in\COs$, there will exist two wedges $W_1=g_1W_0$, $W_2=g_2W_0$, $g_1,g_2\in G^\uparrow$, such that
\begin{align}
	W_1\subset W_2,\qquad \CO  \subset W_1' \cap W_2.
\end{align}
This is for instance the case for $\CO$ a double cone in Minkowski space, or an interval on the lightray, where $W_1$ is a wedge (half-line) open to the right, and $W_2'$ is a wedge (half-line) open to the left. In this situation,
\begin{align*}
	\CA_{U,T}(\CO) &\subset \tilde\CL_T(H(W_1'))\cap\tilde\CL_T(H(W_2))
	=
	\CR_T(H(W_1))\cap\CL_T(H(W_2))
	\\
	&=
	\CL_T(H(W_1))'\cap\CL_T(H(W_2)),
\end{align*}
namely $\CA_{U,T}(\CO)$ is a subalgebra contained in the relative commutant of the inclusion $\CL_T(H(W_1))\subset\CL_T(H(W_2))$. In the cases mentioned above (two-dimensional Minkowski space or lightray), we actually have equality between $\CA_{U,T}(\CO)$ and the relative commutant.

Our results on relative commutants (Theorem~\ref{thm:twisted-hsm} and Proposition~\ref{prop:L2}) then apply as follows.

\begin{itemize}
	\item {\em Lightlike inclusions of wedges or lightrays}.
	\\
	Consider $W_0+a\subset W_0$ a lightlike inclusion of wedges (i.e. $a=(\la,\la,0,\ldots,0)$, $\la>0$) in the Minkowski example a), or $\Rl_++\la\subset\Rl_+$, $\la>0$, an inclusion of half lines in the lightray example b).
	In a positive energy representation, we are then in the situation of Def.~\ref{def:compatible-translation} and hence the corresponding inclusions $\CL_T(H(W_0+a))\subset\CL_T(H(W_0))$ and $\CL_T(H(\Rl_++\la))\subset\CL_T(H(\Rl_+))$ are half-sided modular.

	Now consider an arbitrary admissible twist with $\|T\|<1$ (for example, $T=qF$ with $|q|<1$, or a scaled Yang-Baxter solution $T=qT_S$ as in Example~\ref{example:TSym-scaled}). Then Theorem~\ref{thm:twisted-hsm} implies that the half-sided inclusion has trivial relative commutant.
	\\
	In the context of the net $\CA_{U,T}$, the interpretation of this fact is that there are no observables strictly localized in the relative complement of the inclusion, i.e. on the lightfront $\overline{W_0'+a}\cap \overline W_0$ or the lightlike interval $(0,\la)$. Hence $\CA_{U,T}$ is non-local from this point of view.
	
	\item {\em $L^2$-modular inclusions for massless theories}.
	\\
	Consider Minkowski space $\Rl^{d+1}$, $d\geq1$, two double cones $\CO_1,\CO_2$ with $\overline{\CO_1}\subset\CO_2$, and the representation $U$ defining the free massless scalar field. Then the corresponding standard subspaces $H(\CO_1)\subset H(\CO_2)$ satisfy $L^2$-nuclearity including the trace norm bound $\|\Delta_{H(\CO_2)}^{1/4}\Delta_{H(\CO_1)}^{-1/4}\|_1<1$ if the distance between $\CO_1$ and $\partial\CO_2$ is sufficiently large \cite{BuchholzDAntoniLongo:2007}.
	
	Now consider an arbitrary admissible twist with $\|T\|<1$; for example, $T=qF$ with $|q|<1$. As the modular operators do not have eigenvalue $1$, in this case $\CL_T(H(\CO_i))$ are type III${}_1$ factors and it follows from Prop~\ref{prop:L2} implies that the relative commutant of the inclusion $\CL_T(H(\CO_1))\subset\CL_T(H(\CO_2))$ is of type III as well. 
	
	The interpretation of this fact is that there are lots of elements of $\CL_T(H(\CO_2))$ that are localized in $\CO_1'\cap\CO_2$ in the sense that they commute with the subalgebra $\CL_T(H(\CO_1))$.
\end{itemize}

The first result can be understood as a sign that the models with $\|T\|<1$ are strongly non-local. It complements results establishing large relative commutants for special twists with $\|T\|=1$, and we view such examples as counterexamples to the construction of strictly local theories. The second result does not strictly fit to the wedge-local setting because of the  different geometry (double cone inclusions), but does indicate that nets which appear strongly non-local might contain more local observables than anticipated.

We expect to learn from both scenarios and other examples how to better control the local observable content of twisted models in future work on constructive algebraic quantum field theory.

\appendix

\section{Diagram calculus for twisted $n$-point functions}\label{sect:diagrams}

In this section we explain a diagram calculus for the twisted $n$-point functions $W_{2n}^{\xi_1,\ldots,\xi_{2n}}(t)$ \eqref{def:W2n}.

Our diagrams consist of a disc with $2n$ marked points on its boundary, labelled $1,2,\ldots,2n$ in clockwise orientation. The points are paired by directed edges, which are curves $c$ lying in the interior of the disc, starting at a marked point $s(c)\in\{1,\ldots,2n\}$ and ending at a marked point $t(c)\in\{1,\ldots,2n\}$, such that
\begin{itemize}
	\item[(D1)] every $k\in\{1,\ldots,2n\}$ is connected to exactly one curve,
	\item[(D2)] for every curve $c$, we have $t(c)>s(c)$,
	\item[(D3)] the curves are drawn in such a way that the number of crossings is minimal and there are no points where three or more curves meet.
\end{itemize}

Examples of such diagrams are
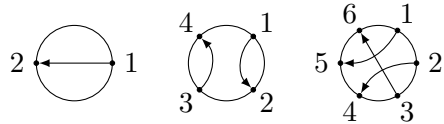
\begin{center}
\begin{tikzpicture}[
  dot/.style={draw,fill,circle,inner sep=1pt,scale=0.6},scale=0.5
  ]
  \node (O1) at (0,0) {};
  \draw (O1) circle (1);

	\node[dot] (a1) at (0:1) {};
	\node[dot] (a2) at (180:1) {};
	\node (n1) at (0:1.5) {\small$1$};
	\node (n2) at (180:1.5) {\small$2$};
\draw[-latex] (a1) to [out=180,in=0,looseness=1] (a2);
  \node (O2) at (4,0) {};
\draw[] (O2) circle (1);

  \foreach \i in {1,...,4} {
	\node[dot] (a\i) at ($(O2) + (135-\i*90:1)$) {};
	\node (n\i) at ($(O2) +(135-\i*90:1.5)$) {\small$\i$} ++(3,0);
  }

\draw[-latex] (a1) to [out=225,in=135,looseness=1] (a2);
\draw[-latex] (a3) to [out=45,in=-45,looseness=1] (a4);
  \node (O3) at (8,0) {};
  \draw[] (O3) circle (1);
  \foreach \i in {1,...,6} {
	\node[dot] (a\i) at ($(O3)+(120-\i*60:1)$) {};
	\node (n\i) at ($(O3)+(120-\i*60:1.5)$) {\small$\i$};
  }

\draw[-latex] (a1) to [out=240,in=0,looseness=1] (a5);
\draw[-latex] (a2) to [out=180,in=60,looseness=1] (a4);
\draw[-latex] (a3) to [out=120,in=-60,looseness=1] (a6);

\end{tikzpicture}
\end{center}
Calling the family of all such diagrams $\Dia_{2n}$, the idea is that any $D\in\Dia_{2n}$ represents a number $\langle D\rangle^{\xi_1,\ldots,\xi_{2n}}$ defined in a manner familiar from knot theory: Every line segment (between two crossings, i.e. an internal line, or between a crossing and a marked boundary point, i.e. an external line) carries a vector in $\Hil$. In case of an external line starting (or ending) at $k$, this vector is $\xi_k$ (or $S_H\xi_k$), and in case of an internal line, it is a vector from some orthonormal basis of $\Hil$. A crossing represents a matrix element of $T$, namely
\begin{center}
 \begin{tikzpicture}[
  dot/.style={draw,fill,circle,inner sep=1pt,scale=0.6},scale=0.5
  ]
  \node (O3) at (0,0) {};
  \foreach \i in {1,...,6} {
	\node (a\i) at ($(O3)+(120-\i*60:1)$) {};
  }

  \node at (0:1) {$\eta$};
  \node at (-120:1) {$\psi$};
  \node at (-60:1) {$\zeta$};
  \node at (120:1) {$\gamma$};

\draw[-latex] (a2) to [out=180,in=60,looseness=1] (a4);
\draw[-latex] (a3) to [out=120,in=-60,looseness=1] (a6);

\node at (5.5,0) {$=\langle\zeta\ot\eta,T(\psi\ot\gamma)\rangle,$};
\end{tikzpicture}
\end{center}
and $\langle D\rangle$ is defined by taking the product over all crossings of $D$ and summing over the orthonormal bases labelling all internal lines. Then
\begin{align}\label{eq:TranslateDiagram}
	W_{2n}^{\xi_1,\ldots,\xi_{2n}} = \sum_{D\in\Dia_{2n}}\langle D\rangle^{\xi_1,\ldots,\xi_{2n}}.
\end{align}
For general twists $T$, the rules (D1)--(D3) are however ambiguous and do not uniquely determine $\langle D\rangle^{\xi_1,\ldots,\xi_{2n}}$. Only in case $T$ is braided the above definition of $\langle D\rangle^{\xi_1,\ldots,\xi_{2n}}$ is unambiguous. This can for example be seen by looking at the two diagrams
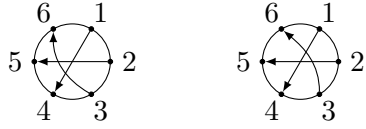
\begin{center}
	\begin{tikzpicture}[
	dot/.style={draw,fill,circle,inner sep=1pt,scale=0.6},scale=0.5
	]
		\node (O4) at (0,0) {};
		\draw[] (O4) circle (1);

		\foreach \i in {1,...,6} {
			\node[dot] (a\i) at ($(O4)+(120-\i*60:1)$) {};
			\node (n\i) at ($(O4)+(120-\i*60:1.5)$) {\small$\i$};
		}

		\draw[-latex] (a1) to [out=240,in=60,looseness=1] (a4);
		\draw[-latex] (a2) to [out=180,in=0,looseness=1] (a5);
		\draw[-latex] (a3) to [out=150,in=-90,looseness=1] (a6);
		
		\node (O5) at (6,0) {};
		\draw[] (O5) circle (1);

		\foreach \i in {1,...,6} {
			\node[dot] (a\i) at ($(O5)+(120-\i*60:1)$) {};
			\node (n\i) at ($(O5)+(120-\i*60:1.5)$) {\small$\i$};
		}

		\draw[-latex] (a1) to [out=240,in=60,looseness=1] (a4);
		\draw[-latex] (a2) to [out=180,in=0,looseness=1] (a5);
		\draw[-latex] (a3) to [out=90,in=-40,looseness=1] (a6);
	\end{tikzpicture}
\end{center}
which both agree with the diagrammatic rules (D1)--(D3), but only give the same value $\langle D\rangle$ in case $T$ is braided (invariance under Reidemeister moves of type III).

For twists that are not braided, this results in somewhat cumbersome diagram rules which we avoid to spell out in detail. Nonetheless, $\langle D\rangle$ is well defined for diagrams with at most two crossings. For the particular diagram with three crossings depicted above, the left version is the correct one for \eqref{eq:TranslateDiagram} to be true for $n=3$ (this can be seen by noting that $P_{T,3}$ contains a term $T_2T_1T_2$ and not $T_1T_2T_1$).

For following the calculations underlying Prop.~\ref{lemma:4pta} and in particular Lemma~\ref{lemma:6pt}, the diagrammatic rule is very useful. The diagrammatic form of the expansion of the $6$-point function \eqref{eq:6pt-split} looks as follows:
\begin{align*}
 w_0(t)
&=
\Big(\langle \bar 1,2\rangle\langle \bar3,4\rangle\langle \bar5,6_t\rangle
+
\langle \bar2,3\rangle\langle\bar4,5\rangle\langle\bar1,6_t\rangle\Big)
\\
&\quad+
\Big(
\langle\bar1,2\rangle\langle\bar4,5\rangle\langle\bar3,6_t\rangle
+
\langle\bar1,4\rangle\langle\bar2,3\rangle\langle\bar5,6_t\rangle
+
\langle\bar2,5\rangle\langle\bar3,4\rangle\langle\bar1,6_t\rangle\Big)
\end{align*}
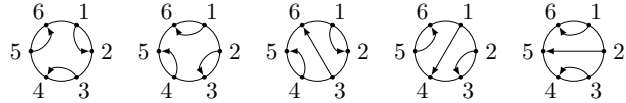
\begin{center}
 \scalebox{0.8}{
\begin{tikzpicture}[
  dot/.style={draw,fill,circle,inner sep=1pt,scale=0.6},scale=0.5
  ]
  \draw[] (0,0) circle (1);

  \node (O) at (0,0) {};
  \foreach \i in {1,...,6} {
	\node[dot] (a\i) at (120-\i*60:1) {};
	\node (n\i) at (120-\i*60:1.5) {\small$\i$};
  }

\draw[-latex] (a1) to [out=240,in=180,looseness=1] (a2);
\draw[-latex] (a3) to [out=120,in=60,looseness=1] (a4);
\draw[-latex] (a5) to [out=0,in=-60,looseness=1] (a6);
\end{tikzpicture}
\begin{tikzpicture}[
  dot/.style={draw,fill,circle,inner sep=1pt,scale=0.6},scale=0.5
  ]
  \draw[] (0,0) circle (1);

  \node (O) at (0,0) {};
  \foreach \i in {1,...,6} {
	\node[dot] (a\i) at (120-\i*60:1) {};
	\node (n\i) at (120-\i*60:1.5) {\small$\i$};
  }

\draw[-latex] (a1) to [out=240,in=-60,looseness=1] (a6);
\draw[-latex] (a2) to [out=180,in=120,looseness=1] (a3);
\draw[-latex] (a4) to [out=60,in=0,looseness=1] (a5);
\end{tikzpicture}
\begin{tikzpicture}[
  dot/.style={draw,fill,circle,inner sep=1pt,scale=0.6},scale=0.5
  ]
  \draw[] (0,0) circle (1);

  \node (O) at (0,0) {};
  \foreach \i in {1,...,6} {
	\node[dot] (a\i) at (120-\i*60:1) {};
	\node (n\i) at (120-\i*60:1.5) {\small$\i$};
  }

\draw[-latex] (a1) to [out=240,in=180,looseness=1] (a2);
\draw[-latex] (a3) to [out=120,in=-60,looseness=1] (a6);
\draw[-latex] (a4) to [out=60,in=0,looseness=1] (a5);
\end{tikzpicture}
\begin{tikzpicture}[
  dot/.style={draw,fill,circle,inner sep=1pt,scale=0.6},scale=0.5
  ]
  \draw[] (0,0) circle (1);

  \node (O) at (0,0) {};
  \foreach \i in {1,...,6} {
	\node[dot] (a\i) at (120-\i*60:1) {};
	\node (n\i) at (120-\i*60:1.5) {\small$\i$};
  }

\draw[-latex] (a1) to [out=240,in=60,looseness=1] (a4);
\draw[-latex] (a2) to [out=180,in=120,looseness=1] (a3);
\draw[-latex] (a5) to [out=0,in=-60,looseness=1] (a6);
\end{tikzpicture}
\begin{tikzpicture}[
  dot/.style={draw,fill,circle,inner sep=1pt,scale=0.6},scale=0.5
  ]
  \draw[] (0,0) circle (1);

  \node (O) at (0,0) {};
  \foreach \i in {1,...,6} {
	\node[dot] (a\i) at (120-\i*60:1) {};
	\node (n\i) at (120-\i*60:1.5) {\small$\i$};
  }

\draw[-latex] (a1) to [out=240,in=-60,looseness=1] (a6);
\draw[-latex] (a2) to [out=180,in=0,looseness=1] (a5);
\draw[-latex] (a3) to [out=120,in=60,looseness=1] (a4);
\end{tikzpicture}
}
\end{center}

\begin{align*}
	w_1(t)
	&=
	\langle\bar1,2\rangle\langle\bar4\ot\bar3,T(5\ot6_t)\rangle
	+
	\langle\bar2,3\rangle\langle\bar4\ot\bar1,T(5\ot6_t)\rangle
	\\
	&\quad+
	\langle\bar3,4\rangle\langle\bar2\ot\bar1,T(5\ot6_t)\rangle
	+
	\langle\bar4,5\rangle\langle\bar2\ot\bar1,T(3\ot6_t)\rangle
	\\
	&\quad+
	\langle\bar2\ot\bar1,T(3\ot4)\rangle\langle\bar5,6_t\rangle
	+
	\langle\bar3\ot\bar2,T(4\ot5)\rangle\langle\bar1,6_t\rangle,
\end{align*}
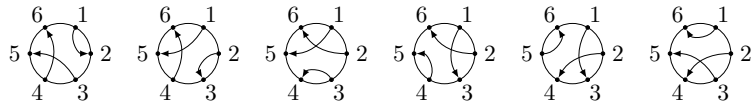
\begin{center}
\scalebox{0.8}{
\begin{tikzpicture}[
  dot/.style={draw,fill,circle,inner sep=1pt,scale=0.6},scale=0.5
  ]
  \draw[] (0,0) circle (1);

  \node (O) at (0,0) {};
  \foreach \i in {1,...,6} {
	\node[dot] (a\i) at (120-\i*60:1) {};
	\node (n\i) at (120-\i*60:1.5) {\small$\i$};
  }

\draw[-latex] (a1) to [out=240,in=180,looseness=1] (a2);
\draw[-latex] (a3) to [out=120,in=0,looseness=1] (a5);
\draw[-latex] (a4) to [out=60,in=-60,looseness=1] (a6);
\end{tikzpicture}
\begin{tikzpicture}[
  dot/.style={draw,fill,circle,inner sep=1pt,scale=0.6},scale=0.5
  ]
  \draw[] (0,0) circle (1);

  \node (O) at (0,0) {};
  \foreach \i in {1,...,6} {
	\node[dot] (a\i) at (120-\i*60:1) {};
	\node (n\i) at (120-\i*60:1.5) {\small$\i$};
  }

\draw[-latex] (a1) to [out=240,in=0,looseness=1] (a5);
\draw[-latex] (a2) to [out=180,in=120,looseness=1] (a3);
\draw[-latex] (a4) to [out=60,in=-60,looseness=1] (a6);
\end{tikzpicture}
\begin{tikzpicture}[
  dot/.style={draw,fill,circle,inner sep=1pt,scale=0.6},scale=0.5
  ]
  \draw[] (0,0) circle (1);

  \node (O) at (0,0) {};
  \foreach \i in {1,...,6} {
	\node[dot] (a\i) at (120-\i*60:1) {};
	\node (n\i) at (120-\i*60:1.5) {\small$\i$};
  }

\draw[-latex] (a1) to [out=240,in=0,looseness=1] (a5);
\draw[-latex] (a2) to [out=180,in=-60,looseness=1] (a6);
\draw[-latex] (a3) to [out=120,in=60,looseness=1] (a4);
\end{tikzpicture}
\begin{tikzpicture}[
    dot/.style={draw,fill,circle,inner sep=1pt,scale=0.6},scale=0.5
	]
  \draw[] (0,0) circle (1);

  \node (O) at (0,0) {};
  \foreach \i in {1,...,6} {
	\node[dot] (a\i) at (120-\i*60:1) {};
	\node (n\i) at (120-\i*60:1.5) {\small$\i$};
  }
\draw[-latex] (a2) to [out=180,in=-60,looseness=1] (a6);
\draw[-latex] (a1) to [out=240,in=120,looseness=1] (a3);
\draw[-latex] (a4) to [out=60,in=0,looseness=1] (a5);
\end{tikzpicture}
\begin{tikzpicture}[
  dot/.style={draw,fill,circle,inner sep=1pt,scale=0.6},scale=0.5
  ]
  \draw[] (0,0) circle (1);

  \node (O) at (0,0) {};
  \foreach \i in {1,...,6} {
	\node[dot] (a\i) at (120-\i*60:1) {};
	\node (n\i) at (120-\i*60:1.5) {\small$\i$};
  }
\draw[-latex] (a2) to [out=180,in=60,looseness=1] (a4);
\draw[-latex] (a1) to [out=240,in=120,looseness=1] (a3);
\draw[-latex] (a5) to [out=0,in=-60,looseness=1] (a6);
\end{tikzpicture}
\begin{tikzpicture}[
  dot/.style={draw,fill,circle,inner sep=1pt,scale=0.6},scale=0.5
  ]
  \draw[] (0,0) circle (1);

  \node (O) at (0,0) {};
  \foreach \i in {1,...,6} {
	\node[dot] (a\i) at (120-\i*60:1) {};
	\node (n\i) at (120-\i*60:1.5) {\small$\i$};
  }

\draw[-latex] (a1) to [out=240,in=-60,looseness=1] (a6);
\draw[-latex] (a2) to [out=180,in=60,looseness=1] (a4);
\draw[-latex] (a3) to [out=120,in=0,looseness=1] (a5);
\end{tikzpicture}
}
\end{center}

\begin{align*}
	w_2(t)
	&=
	\langle\bar4\ot a_3T(\bar2\ot\bar1),T(5\ot6_t)\rangle
	+
	\langle a_4T(\bar3\ot\bar2)\ot\bar1, T(5\ot6_t)\rangle
	\\
	&\quad+
	\langle\bar2\ot\bar1, T(a_{\bar3}T(4\ot5)\ot6_t)\rangle,
\end{align*}

\begin{center}
 \scalebox{0.8}
 {
\begin{tikzpicture}[
  dot/.style={draw,fill,circle,inner sep=1pt,scale=0.6},scale=0.5
  ]
  \draw[] (0,0) circle (1);

  \node (O) at (0,0) {};
  \foreach \i in {1,...,6} {
	\node[dot] (a\i) at (120-\i*60:1) {};
	\node (n\i) at (120-\i*60:1.5) {\small$\i$};
  }

\draw[-latex] (a1) to [out=240,in=120,looseness=1] (a3);
\draw[-latex] (a2) to [out=180,in=0,looseness=1] (a5);
\draw[-latex] (a4) to [out=60,in=-60,looseness=1] (a6);
\end{tikzpicture}
\begin{tikzpicture}[
  dot/.style={draw,fill,circle,inner sep=1pt,scale=0.6},scale=0.5
  ]
  \draw[] (0,0) circle (1);

  \node (O) at (0,0) {};
  \foreach \i in {1,...,6} {
	\node[dot] (a\i) at (120-\i*60:1) {};
	\node (n\i) at (120-\i*60:1.5) {\small$\i$};
  }

\draw[-latex] (a1) to [out=240,in=60,looseness=1] (a4);
\draw[-latex] (a2) to [out=180,in=-60,looseness=1] (a6);
\draw[-latex] (a3) to [out=120,in=0,looseness=1] (a5);
\end{tikzpicture}
\begin{tikzpicture}[
  dot/.style={draw,fill,circle,inner sep=1pt,scale=0.6},scale=0.5
  ]
  \draw[] (0,0) circle (1);

  \node (O) at (0,0) {};
  \foreach \i in {1,...,6} {
	\node[dot] (a\i) at (120-\i*60:1) {};
	\node (n\i) at (120-\i*60:1.5) {\small$\i$};
  }

\draw[-latex] (a1) to [out=240,in=0,looseness=1] (a5);
\draw[-latex] (a2) to [out=180,in=60,looseness=1] (a4);
\draw[-latex] (a3) to [out=120,in=-60,looseness=1] (a6);
\end{tikzpicture}
}
\end{center}

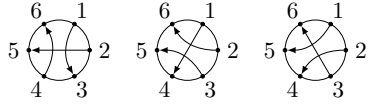
\begin{align*}
	w_3(t)
	&=
	\langle\bar3\ot\bar2\ot\bar1,T_2T_1T_2(4\ot5\ot6_t)\rangle.
\end{align*}
\begin{center}
\scalebox{0.8}
{
	\begin{tikzpicture}[
  dot/.style={draw,fill,circle,inner sep=1pt,scale=0.6},scale=0.5
  ]
  \draw[] (0,0) circle (1);

  \node (O) at (0,0) {};
  \foreach \i in {1,...,6} {
	\node[dot] (a\i) at (120-\i*60:1) {};
	\node (n\i) at (120-\i*60:1.5) {\small$\i$};
  }

\draw[-latex] (a1) to [out=240,in=60,looseness=1] (a4);
\draw[-latex] (a2) to [out=180,in=0,looseness=1] (a5);
\draw[-latex] (a3) to [out=150,in=-90,looseness=1] (a6);
\end{tikzpicture}
}
\end{center}
Here the diagrams are listed in the same order as the terms in the algebraic expansions of the $w_k(t)$.

For the analytic continuation in $t$ from $\Rl$ to $\Rl-i$, another diagrammatic rule is helpful: The analytic continuation of the function $\langle D\rangle^{\xi_1,\ldots,\Delta_H^{it}\xi_{2n}}$ is
\begin{align}
	\langle D\rangle^{\xi_1,\ldots,\Delta_H^{-1}\xi_{2n}}
	=
	\langle D'\rangle^{\xi_{2n},\xi_1,\ldots,\xi_{2n-1}},
\end{align}
where the diagram $D'$ is obtained from $D$ by clockwise rotation by the angle $\frac{\pi}{n}$ and reversing the orientation of the curve ending in $2n$. For example:
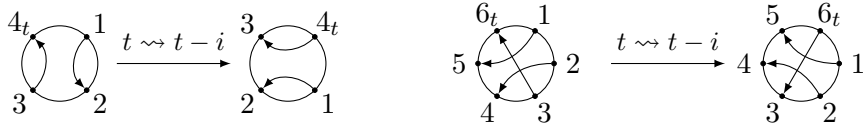
\begin{center}
\begin{tikzpicture}[
  dot/.style={draw,fill,circle,inner sep=1pt,scale=0.6},scale=0.5
  ]
  \node (O2) at (4,0) {};
\draw[] (O2) circle (1);

  \foreach \i in {1,...,3} {
	\node[dot] (a\i) at ($(O2) + (135-\i*90:1)$) {};
	\node (n\i) at ($(O2) +(135-\i*90:1.5)$) {\small$\i$} ++(3,0);
  }
  \node[dot] (a4) at ($(O2) + (135-4*90:1)$) {};
  \node (cr) at ($(O2) +(135-4*90:1.5)$) {\small$4_t$} ++(3,0);

\draw[-latex] (a1) to [out=225,in=135,looseness=1] (a2);
\draw[-latex] (a3) to [out=45,in=-45,looseness=1] (a4);

\draw[-latex] (5.5,0) -- (8.5,0) node[midway,above] {\footnotesize$t\leadsto t-i$};

\node (O2b) at (10,0) {};
\draw[] (O2b) circle (1);

  \foreach \i in {1,...,3} {
	\node[dot] (a\i) at ($(O2b) + (45-\i*90:1)$) {};
	\node (n\i) at ($(O2b) +(45-\i*90:1.5)$) {\small$\i$} ++(3,0);
  }
  	\node[dot] (a4) at ($(O2b) + (45-4*90:1)$) {};
	\node (ncr4) at ($(O2b) +(45-4*90:1.5)$) {\small$4_t$} ++(3,0);

\draw[-latex] (a1) to [out=135,in=45,looseness=1] (a2);
\draw[-latex] (a4) to [out=225,in=-45,looseness=1] (a3);

  \node (O3) at (16,0) {};
  \draw[] (O3) circle (1);
  \foreach \i in {1,...,5} {
	\node[dot] (a\i) at ($(O3)+(120-\i*60:1)$) {};
	\node (n\i) at ($(O3)+(120-\i*60:1.5)$) {\small$\i$};
  }
  \node[dot] (a6) at ($(O3)+(120-6*60:1)$) {};
	\node (n6) at ($(O3)+(120-6*60:1.5)$) {\small$6_t$};

\draw[-latex] (a1) to [out=240,in=0,looseness=1] (a5);
\draw[-latex] (a2) to [out=180,in=60,looseness=1] (a4);
\draw[-latex] (a3) to [out=120,in=-60,looseness=1] (a6);

\draw[-latex] (18.5,0) -- (21.5,0) node[midway,above] {\footnotesize$t\leadsto t-i$};

\node (O4) at (23.5,0) {};
  \draw[] (O4) circle (1);
  \foreach \i in {1,...,5} {
	\node[dot] (a\i) at ($(O4)+(60-\i*60:1)$) {};
	\node (n\i) at ($(O4)+(60-\i*60:1.5)$) {\small$\i$};
  }
  \node[dot] (a6) at ($(O4)+(60-6*60:1)$) {};
	\node (n6) at ($(O4)+(60-6*60:1.5)$) {\small$6_t$};

\draw[-latex] (a1) to [out=180,in=-60,looseness=1] (a5);
\draw[-latex] (a2) to [out=120,in=00,looseness=1] (a4);
\draw[-latex] (a6) to [out=-120,in=60,looseness=1] (a3);
\end{tikzpicture}
\end{center}
With such a diagram calculus, the calculations in Lemma~\ref{lemma:6pt} are easily verified.

\bigskip

We also mention a {\bf diagrammatic representation of crossing symmetry}. In the usual graphical notation for tensor product operators \cite[Chapt.~2.6]{EvansKawahigashi:1998}, crossing symmetry takes the form (with a subscript indicating the action of modular unitaries $\Delta_H^{-it}$ and a line reversing its up/down direction indicating the action of a modular conjugation $J_H$)
\begin{center}
\begin{tikzpicture}[
  dot/.style={draw,red,circle,scale=0.6},scale=0.5
  ]

  \node (F1) at (2,0) {\footnotesize$\langle2_t\ot1,T(3\ot4_t)\rangle=$};
  \node (O2) at (7,0) {$T$};

  \draw[thick,shift={(O2)}] (-1.1,-0.7) rectangle (1.1,0.7);

  \node (tl) at ($(O2) + (-0.5,0.4)$) {};
  \node (tlx) at ($(O2) + (-0.5,2.7)$) {$2_t$};
  \node (tr) at ($(O2) + (0.5,0.4)$) {};
  \node (trx) at ($(O2) + (0.5,2.7)$) {$1$};
  \node (bl) at ($(O2) + (-0.5,-0.4)$) {};
  \node (blx) at ($(O2) + (-0.5,-2.7)$) {$3$};
  \node (br) at ($(O2) + (0.5,-0.4)$) {};
  \node (brx) at ($(O2) + (0.5,-2.7)$) {$4_t$};

  \draw[thick] (tl) -- (tlx);
  \draw[thick] (tr) -- (trx);
  \draw[thick] (bl) -- (blx);
  \draw[thick] (br) -- (brx);

  \draw[-latex] (8.5,0) -- (11.5,0) node[midway,above] {\footnotesize$t\leadsto t+\frac{i}{2}$};

  \node (O2) at (13.5,0) {$T$};

  \draw[thick,shift={(O2)}] (-1.1,-0.7) rectangle (1.1,0.7);

  \node (tl) at ($(O2) + (-0.5,0.4)$) {};
  \node (tlx) at ($(O2) + (-0.5,2.7)$) {$2_t$};
  \node (tr) at ($(O2) + (0.5,0.4)$) {};
  \node (trx) at ($(O2) + (0.5,2.7)$) {$1$};
  \node (bl) at ($(O2) + (-0.5,-0.4)$) {};
  \node (blx) at ($(O2) + (-0.5,-2.7)$) {$3$};
  \node (br) at ($(O2) + (0.5,-0.4)$) {};
  \node (brx) at ($(O2) + (0.5,-2.7)$) {$4_t$};

  \node[inner sep=0, outer sep=0] (ll) at ($(O2) + (-1.5,0)$) {};
  \node[inner sep=0, outer sep=0] (rr) at ($(O2) + (1.5,0)$) {};

  \draw[thick] (tlx) to [out=-90,in=90,looseness=1] (ll) to [out=-90,in=-90,looseness=2.3] (bl);
  \draw[thick] (trx) to [out=-90,in=90,looseness=1] (tl);
  \draw[thick] (blx) to [out=90,in=-90,looseness=1] (br);
  \draw[thick] (brx) to [out=90,in=-90,looseness=1] (rr) to [out=90,in=90,looseness=2.3] (tr);
   \node (F2) at (20,0) {\footnotesize$=\langle1\ot J_H4_t,T(J_H2_t\ot3)\rangle$};
\end{tikzpicture}
\end{center}
It is curious to note that this diagram looks similar to the string Fourier transforms of Jaffe and Liu \cite{JaffeLiu:2017}, but we shall not investigate this connection any further in the present work.

Finally, the diagrammatic view on crossing symmetry gives also a different perspective on the derivation of the Yang-Baxter equation in Thm.~\ref{thm:separating-necessary}~\ref{item:proof-ybe}. The analytic continuation can be seen as a two-fold application of the crossing symmetry above, which transforms $T_2T_1T_2$ into $T_1T_2T_1$:
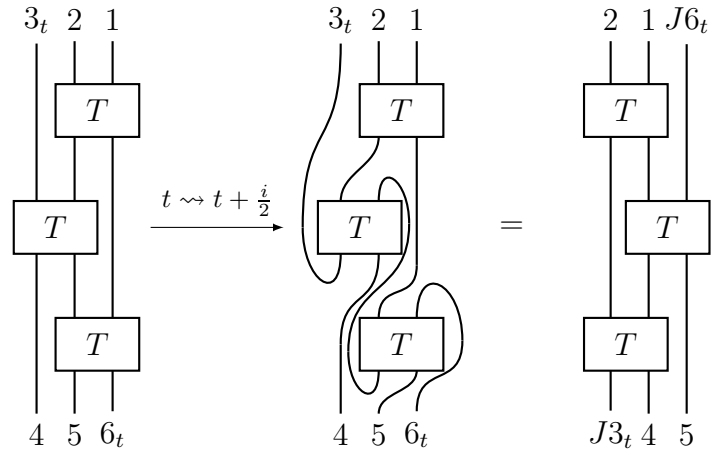
\begin{center}
\begin{tikzpicture}[
  dot/.style={draw,fill,circle,scale=0.4,inner sep=0pt},scale=0.5
  ]

  \node (O2) at (7,0) {$T$};
  \node (O2b) at (8.1,3.1) {$T$};
  \node (O2c) at (8.1,-3.1) {$T$};

  \draw[thick,shift={(O2)}] (-1.1,-0.7) rectangle (1.1,0.7);
  \draw[thick,shift={(O2b)}] (-1.1,-0.7) rectangle (1.1,0.7);
  \draw[thick,shift={(O2c)}] (-1.1,-0.7) rectangle (1.1,0.7);

  \node (tlxx) at ($(O2) + (-0.5,5.5)$) {$3_t$};
  \node (blxx) at ($(O2) + (-0.5,-5.5)$) {$4$};
  \node (tl) at ($(O2) + (-0.5,0.4)$) {};
  \node (tlx) at ($(O2) + (-0.5,2.7)$) {};
  \node (tr) at ($(O2) + (0.5,0.4)$) {};
  \node (trx) at ($(O2) + (0.5,2.7)$) {};
  \node (bl) at ($(O2) + (-0.5,-0.4)$) {};
  \node (blx) at ($(O2) + (-0.5,-2.7)$) {};
  \node (br) at ($(O2) + (0.5,-0.4)$) {};
  \node (brx) at ($(O2) + (0.5,-2.7)$) {};
  \node (tmt) at ($(O2) + (1.5,2.7)$) {};
  \node (tmtt) at ($(O2) + (0.5,3.5)$) {};
  \node (tmtx) at ($(O2) + (0.5,5.5)$) {$2$};
  \node (trtt) at ($(O2) + (1.5,3.5)$) {};
  \node (trtx) at ($(O2) + (1.5,5.5)$) {$1$};
  \node (bmt) at ($(O2) + (1.5,-2.7)$) {};

  \node (bmtt) at ($(O2) + (0.5,-3.5)$) {};
  \node (bmtx) at ($(O2) + (0.5,-5.5)$) {$5$};
  \node (brtt) at ($(O2) + (1.5,-3.5)$) {};
  \node (brtx) at ($(O2) + (1.5,-5.5)$) {$6_t$};

  \draw[thick,inner sep=0, outer sep=0] (tmtt) -- (tmtx);
  \draw[thick,inner sep=0, outer sep=0] (trtt) -- (trtx);
  \draw[thick,inner sep=0, outer sep=0] (bmtt) -- (bmtx);
  \draw[thick,inner sep=0, outer sep=0] (brtt) -- (brtx);
  \draw[thick,inner sep=0, outer sep=0] (bmt) -- (tmt);
  \draw[thick,inner sep=0, outer sep=0] (tl) -- (tlxx);
  \draw[thick,inner sep=0, outer sep=0] (tr) -- (trx);
  \draw[thick,inner sep=0, outer sep=0] (bl) -- (blxx);
  \draw[thick,inner sep=0, outer sep=0] (br) -- (brx);

  \draw[-latex] (9.5,0) -- (13,0) node[midway,above] {\footnotesize$t\leadsto t+\frac{i}{2}$};

  \node (O3) at (15,0) {$T$};
  \node (O3b) at (16.1,3.1) {$T$};
  \node (O3c) at (16.1,-3.1) {$T$};

  \draw[thick,shift={(O3)}] (-1.1,-0.7) rectangle (1.1,0.7);
  \draw[thick,shift={(O3b)}] (-1.1,-0.7) rectangle (1.1,0.7);
  \draw[thick,shift={(O3c)}] (-1.1,-0.7) rectangle (1.1,0.7);

  \node (tlxx) at ($(O3) + (-0.5,5.5)$) {$3_t$};
  \node (blxx) at ($(O3) + (-0.5,-5.5)$) {$4$};
  \node (tl) at ($(O3) + (-0.5,0.4)$) {};
  \node (tlx) at ($(O3) + (-0.5,2.7)$) {};
  \node (tr) at ($(O3) + (0.5,0.4)$) {};
  \node (trx) at ($(O3) + (0.5,2.7)$) {};
  \node (bl) at ($(O3) + (-0.5,-0.4)$) {};
  \node (blx) at ($(O3) + (-0.5,-2.7)$) {};
  \node (br) at ($(O3) + (0.5,-0.4)$) {};
  \node (brx) at ($(O3) + (0.5,-2.7)$) {};
  \node (tmt) at ($(O3) + (1.5,2.7)$) {};
  \node (tmtt) at ($(O3) + (0.5,3.5)$) {};
  \node (tmtx) at ($(O3) + (0.5,5.5)$) {$2$};
  \node (trtt) at ($(O3) + (1.5,3.5)$) {};
  \node (trtx) at ($(O3) + (1.5,5.5)$) {$1$};
  \node (bmt) at ($(O3) + (1.5,-2.7)$) {};

  \node (bmtt) at ($(O3) + (0.5,-3.5)$) {};
  \node (bmtx) at ($(O3) + (0.5,-5.5)$) {$5$};
  \node (brtt) at ($(O3) + (1.5,-3.5)$) {};
  \node (brtx) at ($(O3) + (1.5,-5.5)$) {$6_t$};

  \node[minimum size=0mm,dot] (lh) at ($(O3)+(-1.5,0)$) {};
  \node[minimum size=0mm,dot] (h2) at ($(O3)+(-0.5,-3.1)$) {};
  \node[minimum size=0mm,dot] (h3) at ($(O3)+(2.7,-3)$) {};
  \node[minimum size=0mm,dot] (h4) at ($(O3)+(1.5,-1)$) {};
  \node[minimum size=0mm,dot] (h5) at ($(O3)+(1.3,0.1)$) {};
  \node[minimum size=0mm,dot] (h6) at ($(O3)+(-0.3,-3.3)$) {};

  \draw[thick,inner sep=0, outer sep=0] (tmtt) -- (tmtx);
  \draw[thick,inner sep=0, outer sep=0] (trtt) -- (trtx);
  \draw[thick,inner sep=0, outer sep=0] (bmtx) to [out=90,in=-90,looseness=1]  (brtt);
  \draw[thick,inner sep=0, outer sep=0] (brtx) to [out=90,in=-90,looseness=1.5] (h3) to [out=90,in=90,looseness=3] (bmt);
  \draw[thick,inner sep=0, outer sep=0] (brx) to [out=90,in=-90,looseness=1.5] (h4) to [out=90,in=-90,looseness=1.5] (tmt);
  \draw[thick,inner sep=0, outer sep=0] (tlxx) to [out=-90,in=90,looseness=1.5] (lh) to [out=-90,in=-90,looseness=3] (bl);
  \draw[thick,inner sep=0, outer sep=0] (trx) to [out=-90,in=90,looseness=1] (tl);
  \draw[thick,inner sep=0, outer sep=0] (blxx) to [out=90,in=-90,looseness=1.5] (h2) to [out=90,in=-90,looseness=1.5] (br);
  \draw[thick,inner sep=0, outer sep=0] (bmtt) to [out=-90,in=-90,looseness=3] (h6) to [out=90,in=-90,looseness=1.3] (h5) to [out=90,in=90,looseness=3] (tr);

  \node at (14+5,0) {$=$};

  \node (O4) at (15+8.1,0) {$T$};
  \node (O4b) at (15+7,3.1) {$T$};
  \node (O4c) at (15+7,-3.1) {$T$};

  \draw[thick,shift={(O4)}] (-1.1,-0.7) rectangle (1.1,0.7);
  \draw[thick,shift={(O4b)}] (-1.1,-0.7) rectangle (1.1,0.7);
  \draw[thick,shift={(O4c)}] (-1.1,-0.7) rectangle (1.1,0.7);

  \node (tlxx) at ($(O4) + (0.5,5.5)$) {$J6_t$};

  \node (blxx) at ($(O4) + (0.5,-5.5)$) {$5$};
  \node (tl) at ($(O4) + (0.5,0.4)$) {};
  \node (tlx) at ($(O4) + (0.5,2.7)$) {};
  \node (tr) at ($(O4) + (-0.5,0.4)$) {};
  \node (trx) at ($(O4) + (-0.5,2.7)$) {};
  \node (bl) at ($(O4) + (0.5,-0.4)$) {};
  \node (blx) at ($(O4) + (0.5,-2.7)$) {};
  \node (br) at ($(O4) + (-0.5,-0.4)$) {};
  \node (brx) at ($(O4) + (-0.5,-2.7)$) {};
  \node (tmt) at ($(O4) + (-1.5,2.7)$) {};
  \node (tmtt) at ($(O4) + (-0.5,3.5)$) {};

  \node (tmtx) at ($(O4) + (-0.5,5.5)$) {$1$};
  \node (trtt) at ($(O4) + (-1.5,3.5)$) {};
  \node (bmt) at ($(O4) + (-1.5,-2.7)$) {};

  \node (trtx) at ($(O4) + (-1.5,5.5)$) {$2$};

  \node (bmtt) at ($(O4) + (-0.5,-3.5)$) {};
  \node (bmtx) at ($(O4) + (-0.5,-5.5)$) {$4$};
  \node (brtt) at ($(O4) + (-1.5,-3.5)$) {};
  \node (brtx) at ($(O4) + (-1.5,-5.5)$) {$J3_t$};

  \draw[thick,inner sep=0, outer sep=0] (tmtt) -- (tmtx);
  \draw[thick,inner sep=0, outer sep=0] (trtt) -- (trtx);
  \draw[thick,inner sep=0, outer sep=0] (bmtt) -- (bmtx);
  \draw[thick,inner sep=0, outer sep=0] (brtt) -- (brtx);
  \draw[thick,inner sep=0, outer sep=0] (bmt) -- (tmt);
  \draw[thick,inner sep=0, outer sep=0] (tl) -- (tlxx);
  \draw[thick,inner sep=0, outer sep=0] (tr) -- (trx);
  \draw[thick,inner sep=0, outer sep=0] (bl) -- (blxx);
  \draw[thick,inner sep=0, outer sep=0] (br) -- (brx);

\end{tikzpicture}
\end{center}

Here $J=J_H$, and in combination with the cyclic permutation underlying the KMS condition, one obtains the Yang-Baxter equation.

\section*{Acknowledgements}

We thank Maximilian Duell for discussions at the beginning of this work, Roberto Longo for highlighting the possibility of split inclusions in \cite{DAntoniLongoRadulescu:2001}, and Roland Speicher for an exchange about the von Neumann algebras $\CL_{T}(H)$. Support by the German Research Foundation DFG through the Heisenberg project ``Quantum Fields and Operator Algebras'' (LE 2222/3-1) is gratefully acknowledged.

\section*{Data Availability Statement}
Data sharing not applicable to this article as no datasets were generated or analysed during the current study.

\end{document}